%% file: main.tex
\pdfoutput=1
\documentclass[3p,longtitle]{elsarticle}

\usepackage[utf8]{inputenc}
\usepackage[T1]{fontenc}
\usepackage[english]{babel}
\usepackage{lmodern}
\usepackage{microtype}

\usepackage{amsmath,amssymb,amsthm}
\usepackage{mathtools}
\usepackage{lineno}

\usepackage{graphicx}
\usepackage{subcaption}
\usepackage{booktabs}
\usepackage{placeins}
\usepackage{comment}
\usepackage{xcolor}
\usepackage{tikz}
\usetikzlibrary{backgrounds,calc,patterns,positioning,decorations.pathreplacing}
\usepackage{pgfplots}
\pgfplotsset{compat=1.18}
\usepackage[colorlinks=true,linkcolor=blue,citecolor=blue,urlcolor=blue]{hyperref}
\usepackage[capitalise,nameinlink]{cleveref}

\definecolor{OliveGreen}{rgb}{0,0.49,0}
\definecolor{azure}{rgb}{0.0,0.5,1.0}
\definecolor{rust}{rgb}{0.75,0.15,0.05}
\definecolor{objgray}{rgb}{0.45,0.45,0.45}

\theoremstyle{plain}
\newtheorem{theorem}{Theorem}[section]
\newtheorem{lemma}[theorem]{Lemma}

\theoremstyle{definition}

\theoremstyle{remark}
\newtheorem{remark}[theorem]{Remark}

\numberwithin{equation}{section}

\newcommand{\R}{\mathbb{R}}
\newcommand{\Mh}{\mathcal{M}_h} %
\newcommand{\Mhat}{\widehat{\mathcal{M}}_h} %
\newcommand{\Mimp}{\mathcal{M}_h^{\mathrm{I}}} %
\newcommand{\Mexp}{\mathcal{M}_h^{\mathrm{E}}} %
\newcommand{\Mcut}{\mathcal{C}_h} %
\newcommand{\Mtr}{\mathcal{T}_h} %
\newcommand{\In}{I^n} %
\newcommand{\Fnum}{\mathcal{F}} %
\newcommand{\jump}[1]{[\![#1]\!]}
\newcommand{\avg}[1]{\{\!\!\{#1\}\!\!\}}
\newcommand{\norm}[1]{\left\lVert#1\right\rVert}
\newcommand{\dx}{\,\mathrm{d}\mathbf{x}}
\newcommand{\ds}{\,\mathrm{d}s}
\newcommand{\dt}{\,\mathrm{d}t}
\newcommand{\pc}{{\color{blue}p_h}} %
\newcommand{\qc}{{\color{orange}q_h}} %

\journal{Journal of Computational Physics}

\makeatletter
\renewcommand\elsparagraph{\@startsection{paragraph}{4}{0\z@}{10\p@ \@plus 6\p@ \@minus 3\p@}{-6\p@}{\normalfont\bfseries}}
\makeatother

\allowdisplaybreaks

\begin{document}

\begin{frontmatter}

\title{ADER-ZIST -- a mixed explicit ADER-DG and zonally implicit space-time DG approach for Cartesian cut cell meshes}

\author[UniTN,SUSTech]{M. Dumbser\corref{cor1}}
\ead{michael.dumbser@unitn.it}

\author[UniMS]{C. Engwer}
\ead{christian.engwer@uni-muenster.de}   

\author[TUB]{S. May}
\ead{may@tu-berlin.de}  
\cortext[cor1]{Corresponding author}

\author[TUB]{P. Munch}
\ead{p.muench@tu-berlin.de}  

\address[UniTN]{Laboratory of Applied Mathematics, DICAM, University of Trento, via Mesiano 77, 38123 Trento, Italy}

\address[SUSTech]{Department of Mathematics, Southern University of Science and Technology, \\ Xueyuan Avenue 1088, 518055 Shenzhen, Guangdong, China}

\address[UniMS]{Institute for Analysis and Numerics, University of M\"unster, M\"unster, Germany}

\address[TUB]{Institute of Mathematics, TU Berlin, Berlin, Germany}

\begin{abstract}
We present high order accurate mixed explicit-implicit discontinuous Galerkin (DG) schemes for the solution of hyperbolic systems of conservation laws on Cartesian cut cell meshes. Such meshes may contain arbitrarily small cut cells, which leads to the well-known small cell problem: the CFL condition of explicit schemes would severely restrict the admissible time step. To overcome this difficulty, we use high order explicit fully-discrete one-step ADER-DG schemes only on Cartesian cells away from the embedded boundaries. The zone close to the embedded boundary, which consists of cut cells and their direct face neighbors, is discretized with an implicit space-time DG scheme that is coupled to the explicit scheme directly via the numerical flux. The resulting scheme, called ADER-ZIST (ADER-DG with Zonally Implicit Space-Time DG), is conservative by construction, and its time step is determined only by the size of the Cartesian background cells. Curved embedded boundaries are represented by high order isoparametric sub-triangles inside the cut cells, and the nonlinear system of the implicit cells is solved via an inexact matrix-free Newton--Krylov method. For problems with shock waves, we add an artificial viscosity term whose strength is controlled by an oscillation-eliminating smoothness sensor. ADER-ZIST is assessed on a comprehensive set of test problems for the compressible Euler equations, ranging from smooth flows around straight and curved embedded boundaries to acoustic scattering and shock-dominated flows. The scheme essentially reaches its nominal order of accuracy $N+1$ for smooth solutions and for polynomial degrees up to $N=5$, on meshes with minimal volume fractions of the small cut cells down to the order of $10^{-7}$.
\end{abstract}


\end{frontmatter}

\section{Introduction}\label{sec:intro}

For flow problems in complex geometries, embedded boundary or cut cell meshes offer an
attractive alternative to body-fitted meshes: the computational mesh is obtained by
intersecting the domain with a Cartesian background mesh, a step that requires neither user
interaction nor substantial computational effort. Away from the embedded object, all cells
are regular Cartesian cells. Along its boundary, however, the intersection produces cut
cells of irregular shape, whose volume may be arbitrarily small compared to the volume of
the background cells. When time-dependent hyperbolic conservation laws are solved with
explicit time integration, which is the common choice for this class of problems, this
leads to what is known as the small cell problem \cite{Berger2017}: the step size permitted
by the CFL condition shrinks with the size of the smallest cell, and a time step adapted to
the background mesh renders the update on small cut cells unstable.

A natural remedy is to eliminate small cut cells from the mesh by combining them with
suitable neighbors, which is referred to as cell merging or cell agglomeration, see, e.g., 
\cite{BayyukPowellvanLeer1993,Quirk1994, QinKrivodonova2013,MuellerKraemerEisKummerOberlack2016,GulizziAlmgrenBell2022}. While
conceptually simple, this strategy shifts considerable complexity to the mesh generation,
in particular in three dimensions. Many authors have therefore retained the small cut cells
and modified the numerical scheme instead. In the finite volume (FV) context, examples
include flux redistribution \cite{ChernColella1987,ColellaGravesKeenModiano2006}, $h$-box
methods \cite{BergerHelzelLeVeque2003,BergerHelzel2012}, dimensionally split flux
stabilization \cite{KleinBatesNikiforakis2009,GokhaleNikiforakisKlein2018}, mixed explicit-implicit
time stepping \cite{MayBerger2017,MayLaakmann2024}, an active flux variant for cut
cells \cite{HelzelKerkmann2020}, and state redistribution (SRD) \cite{BergerGiuliani2021},
which has been developed further into weighted
\cite{GiulianiAlmgrenBellBergerHenrydeFrahanRangarajan2022} and provably stable
\cite{BergerGiuliani2024} versions.

Algorithmic approaches for discontinuous Galerkin (DG) schemes have emerged more recently. An early example is the
use of implicit Euler time stepping for a transport problem on a cut cell mesh
\cite{BastianEngwerFahlkeIppisch2011}, which avoids the stability issue at the price of
first-order accuracy in time. Several groups have since adapted the ghost penalty technique
\cite{Burman2010}, originally devised for elliptic problems within cut finite element
methods (for a comprehensive review of this field, we refer to
\cite{BurmanHansboLarsonZahedi2025}), to first-order hyperbolic equations
\cite{FuKreiss2021,FuFrachonKreissZahedi2022}; more recent work in this direction ensures
bound preservation in one space dimension \cite{FuKreissZahedi2024} and a maximum principle
for scalar problems in several space dimensions \cite{FuKreissXinZahedi2026}. The Domain of
Dependence (DoD) stabilization \cite{EngwerMayNuessingStreitbuerger2020} augments the DG
discretization by penalty terms that let information pass through small cut cells within
one time step, so that their outflow neighbors are updated with the correct domain of
dependence. Subsequent work generalized it to nonlinear systems in one dimension
\cite{MayStreitbuerger2022}, to higher polynomial degrees in two dimensions
\cite{StreitbuergerBirkeEngwerMay2023} and to the linear wave equation
\cite{BirkeEngwerMayPetriRanocha2026}, investigated its fully discrete stability
\cite{PetriBirkeEngwerRanocha2025}, and proved an \textit{a priori} error estimate for its
first-order variant in two dimensions \cite{BirkeEngwerGiesselmannMay2026}.
SRD has been formulated for DG schemes as well \cite{Giuliani2022,TaylorWilcoxChan2025}.
Yet another option is offered by Galerkin difference methods, whose basis functions extend
over several cells \cite{KaurYanHicken2023,KreissMunchWang2026}.

Other DG-based works focus on adapting the time stepping scheme in the vicinity of the embedded
boundary, e.g., by local time stepping, often in combination with cell agglomeration
\cite{GeisenhoferKummerMueller2019,SchoederStickoKreissKronbichler2020}. The approach
proposed in this paper also adapts the time stepping near the embedded boundary: instead of
advancing small cut cells with smaller explicit time steps, it treats them implicitly, so
that one common time step can be used on the entire mesh. It combines two families of
high-order schemes, the explicit ADER-DG scheme and the implicit space-time DG scheme, which
we briefly review next.

The ADER approach was originally introduced by Toro \textit{et al.} with the aim of constructing a family of arbitrary high order accurate fully-discrete one-step FV schemes \cite{ToroMillingtonNejad2001,SchwartzkopffMunzToro2002,TitarevToro2002,TitarevToro2005,ToroTitarev2006,toro-book}, in which a high order accurate numerical flux is obtained from the approximate solution of generalized Riemann problems (GRP). It was subsequently extended to unstructured meshes \cite{DumbserKaeserTitarevToro2007} and to DG schemes in \cite{DumbserMunz2006,LTS}. In all early works on ADER schemes high order time derivatives were replaced by high order space derivatives with the aid of the Cauchy-Kovalewskaya procedure, which becomes algebraically very cumbersome for complex hyperbolic systems. In \cite{DumbserEnauxToro2008,DumbserBalsaraToroMunz2008} the Cauchy-Kovalewskaya procedure was then successfully replaced by an element-local space-time Galerkin predictor, see also \cite{LoercherGassnerMunz2007,GassnerDumbserHindenlangMunz2011}. ADER-DG schemes therefore consist of a local space-time predictor, which is computed in each cell in an element-local manner, without any coupling with the neighbor cells, and of a fully discrete one-step corrector that couples neighboring cells with each other via a numerical flux. ADER-DG methods have been successfully extended to non-conservative hyperbolic systems \cite{ADERNC} and have been combined with \textit{a posteriori} subcell FV limiting \cite{DumbserZanottiLoubereDiot2014,DumbserLoubere2016}, with adaptive mesh refinement (AMR) \cite{ZanottiFambriDumbser2015,Exahype} and with moving meshes in the Arbitrary-Lagrangian-Eulerian (ALE) formalism \cite{BoscheriDumbser2017,GaburroBoscheriChiocchettiKlingenbergSpringelDumbser2020,GaburroKlimaBonafiniTavelli2027}, see \cite{Frontiers2020} for an overview. In the context of block-structured AMR, which was introduced by Berger \textit{et al.} \cite{Berger-Oliger1984,Berger-Colella1989}, ADER schemes naturally allow for arbitrary high order time-accurate local time stepping (LTS) \cite{AMR3DCL,ZanottiFambriDumbser2015}, since their space-time structure makes them fully discrete one-step schemes, unlike classical multi-stage Runge-Kutta DG methods \cite{CBS4}. In this paper we once more use the space-time nature of ADER-DG schemes, since it allows for a natural coupling with implicit space-time DG schemes.

In space-time DG schemes \cite{vanderVegtvanderVen2002,KlaijvanderVegtvanderVen2006,SudirhamvanderVegtvanDamme2006,HiltebrandMishra2014,TavelliDumbser2015,TavelliDumbser2017}, the discrete solution is a piecewise polynomial in space and time, and the space-time polynomials of neighboring cells are coupled with each other inside the time slab via numerical fluxes. The resulting schemes are in general fully implicit and allow large time steps, but they require the solution of a large system of nonlinear algebraic equations in each time step, and they are therefore considerably more expensive than explicit DG schemes. Exceptions are the space-time DG schemes of \cite{TavelliDumbser2015,TavelliDumbser2017}, which are written in such a way that they follow the semi-implicit strategy of Casulli \cite{Casulli1990}, so that essentially only a system for the pressure has to be solved in each time step.

In this paper we couple these two families of schemes, which are both of the space-time DG type: in ADER-DG schemes only the predictor is a space-time DG scheme, which is however element-local and does not couple a cell with its neighbors, while the space-time DG scheme is fully implicit. The idea of treating only the cut cells and their direct neighbors implicitly, and the rest of the mesh explicitly, follows the mixed explicit-implicit FV scheme of May and Berger \cite{MayBerger2017}. In the FV setting, the combination of an explicit and an implicit scheme of the same order into a mixed scheme of the same order turned out to be rather difficult \cite{MayLaakmann2024}. The main novelty of the present paper is a zonally implicit high order DG method for Cartesian cut cell meshes, which is not based on classical Runge-Kutta DG schemes, but which combines explicit fully-discrete one-step ADER-DG schemes on all Cartesian cells away from the embedded boundary with the implicit space-time DG scheme used only on a rather thin zone of implicit cells, which consists of all cut cells and of all their direct face neighbors. This is why we call our method also zonally implicit, and we refer to it as ADER-ZIST (ADER-DG with Zonally Implicit Space-Time DG). Since both schemes employ space-time polynomials and are coupled with each other only via the arguments of the numerical flux, this naturally leads to a conservative fully-discrete one-step scheme whose time step is determined by the size of the background cells only. The same idea, namely the use of a locally implicit space-time DG scheme on degenerate elements within an otherwise explicit ADER-DG scheme, has very recently been employed by Gaburro \textit{et al.} \cite{GaburroKlimaBonafiniTavelli2027} to treat isolated space-time sliver elements robustly within ALE ADER-DG schemes on moving unstructured meshes in three space dimensions with topology changes. 
To the best of our knowledge, the coupling of explicit ADER-DG schemes with implicit space-time DG methods has never been applied to cut cell meshes before. In \cite{GaburroKlimaBonafiniTavelli2027} each degenerate space-time element has only standard elements as neighbors and is solved in an isolated manner only together with its direct neighbors, for all of which an explicit ADER-DG predictor solution is already known, while in the present approach an arbitrary number of adjacent cut cells, together with all their direct neighbors, forms one coupled zonally implicit system. The implicit system is totally agnostic to the topology of the embedded boundary and automatically couples as many implicit cells with each other as necessary. 
The present approach also differs from the regionally implicit DG schemes of Guthrey and Rossmanith \cite{GuthreyRossmanith2019,ChristliebGuthreyRossmanith2021}, where the predictor of a Lax-Wendroff type DG scheme is obtained from a truncated space-time DG problem on a region that consists of each element and of its neighbors, in order to remove the dependence of the maximum stable time step on the polynomial degree on all elements of the mesh, while the corrector remains explicit. We note that the use of an implicit space-time DG scheme on the zone of implicit cells can be interpreted as an \textit{algebraic cell merging} technique: instead of merging small cut cells geometrically with their neighbors, or redistributing their state onto a neighborhood \cite{BergerGiuliani2021,Giuliani2022}, the small cut cells are coupled with their neighbors algebraically through the implicit system in each time step, while the geometry of the original cut cells and the full polynomial degree $N$ are kept. Further ingredients of the scheme are the isoparametric treatment of curved embedded boundaries following \cite{BassiRebay1997} and, for problems with shock waves, an artificial viscosity, which is discretized with the non-symmetric interior penalty Galerkin method and which is controlled by an oscillation-eliminating smoothness sensor adapted from \cite{AbgrallLiu2026}.

The rest of the paper is organized as follows. In \cref{sec:method}, we introduce the
governing equations and the cut cell mesh, and we recall the explicit ADER-DG scheme and the
implicit space-time DG scheme. We then present the mixed ADER-ZIST scheme
and its conservation
property, the artificial viscosity for problems with shock waves and some implementation
aspects. In \cref{sec:numerics}, we show numerical results for smooth problems with
straight, curved and non-convex embedded boundaries, for the scattering of an acoustic pulse
at five cylinders and for problems with shock waves. \Cref{sec:conclusion} concludes the
paper with some remarks and an outlook to future work.

\section{Mixed explicit ADER-DG and zonally implicit space-time DG approach}\label{sec:method}

In this section we present the design of our 
ADER-ZIST scheme
for Cartesian cut cell meshes. In \cref{sec:setting} we introduce the governing equations, the cut cell mesh, the discrete spaces and the numerical flux that are used throughout this paper. In \cref{sec:spacetime} we then describe the two building blocks of our approach, namely the explicit ADER-DG scheme run on the Cartesian background mesh away from the embedded boundaries and the zonally implicit space-time DG scheme run on the cut cells and their neighbors. We emphasize that both schemes are derived from the same weak formulation of the PDE in space-time and that both compute a local space-time polynomial of degree $N$ on each cell, the only difference being whether the interaction with the neighbor cells is taken into account already inside the time slab, or not. This observation is the key to the mixed scheme, which is introduced in \cref{sec:mixed}. In \cref{sec:av} we describe the artificial viscosity that is added for problems with shock waves, and we conclude this section with a discussion of some implementation aspects in \cref{sec:implementation}, in particular the treatment of curved boundaries, the solution of the nonlinear system of the implicit cells and its conditioning. Although the basic ideas of our new algorithm in principle also apply to the three-dimensional case, throughout this paper we restrict ourselves to the two-dimensional setting. 

\subsection{Governing equations, cut cell mesh and notation}\label{sec:setting}

We consider hyperbolic systems of conservation laws in two space dimensions of the form
\begin{eqnarray}
  \partial_t \mathbf{Q} + \nabla\cdot\mathbf{F}(\mathbf{Q}) = 0, \qquad \mathbf{x} \in \Omega, \quad t \in (0,T),
  \label{eq:conslaw}
\end{eqnarray}
with initial condition $\mathbf{Q}(\mathbf{x},0) = \mathbf{Q}_0(\mathbf{x})$ and suitable boundary conditions. Here, $\Omega \subset \R^2$ is an open and bounded domain, $\mathbf{x} = (x,y)$ is the vector of spatial coordinates, $\mathbf{Q} \in \Omega_{\mathbf{Q}} \subset \R^m$ is the vector of conserved variables, defined in the space of admissible states $\Omega_{\mathbf{Q}}$, and $\mathbf{F} = (\mathbf{f}, \mathbf{g})$ is the flux tensor, whose columns $\mathbf{f}(\mathbf{Q})$ and $\mathbf{g}(\mathbf{Q})$ are the fluxes in $x$- and $y$-direction, respectively. The flux in the direction of a unit vector $\mathbf{n} = (n_x,n_y)$ is denoted by $\mathbf{F}(\mathbf{Q})\cdot\mathbf{n} = n_x \mathbf{f}(\mathbf{Q}) + n_y \mathbf{g}(\mathbf{Q})$, and we assume that the associated Jacobian $\mathbf{A}(\mathbf{Q},\mathbf{n}) = \partial (\mathbf{F}(\mathbf{Q})\cdot\mathbf{n}) / \partial \mathbf{Q}$ has $m$ real eigenvalues $\lambda_1(\mathbf{Q},\mathbf{n}), \ldots, \lambda_m(\mathbf{Q},\mathbf{n})$ and a complete set of linearly independent eigenvectors for all $\mathbf{Q} \in \Omega_{\mathbf{Q}}$ and all unit vectors $\mathbf{n}$. The governing PDE system considered in this paper is given by the compressible Euler equations of gasdynamics, which read 
\begin{eqnarray}
  \mathbf{Q} = \left( \begin{array}{c} \rho \\ \rho u \\ \rho v \\ \rho E \end{array} \right), \qquad
  \mathbf{f}(\mathbf{Q}) = \left( \begin{array}{c} \rho u \\ \rho u^2 + p \\ \rho u v \\ u (\rho E + p) \end{array} \right), \qquad
  \mathbf{g}(\mathbf{Q}) = \left( \begin{array}{c} \rho v \\ \rho u v \\ \rho v^2 + p \\ v (\rho E + p) \end{array} \right),
  \label{eq:euler}
\end{eqnarray}
with the density $\rho$, the velocity vector $\mathbf{v} = (u,v)$, the total energy per unit mass $E$ and the pressure $p$, which is given by the equation of state of an ideal gas,
\begin{eqnarray}
  p = (\gamma-1) \left( \rho E - \frac{1}{2} \rho \, \| \mathbf{v} \|^2 \right),
  \label{eq:eos}
\end{eqnarray}
with the ratio of specific heats $\gamma$. 

\paragraph{Cut cell mesh}
Following \cite{EngwerMayNuessingStreitbuerger2020}, we construct the computational mesh by intersecting a Cartesian background mesh with the domain $\Omega$, see \cref{fig:cutcellmesh}. Let $\widehat\Omega \supset \Omega$ be a rectangle that is sufficiently large to contain the physical domain $\Omega$, e.g.\ its bounding box, and let $\Mhat$ be a uniform Cartesian mesh of $\widehat\Omega$ with cells $\widehat E$ of size $h_x \times h_y$. The embedded boundary $\partial\Omega \cap \widehat\Omega$ is approximated by one or several polygons. Curved boundaries are represented by a cubic spline interpolation of this polygonal representation of $\partial \Omega$, where we always assume that the boundary has been sampled densely enough by the vertices of the polygon so that the cubic spline approximates $\partial \Omega$ well enough, see \cref{sec:implementation}. The cut cell mesh is then given by
\begin{eqnarray}
  \Mh := \left\{ E := \widehat E \cap \Omega \;\, \big| \;\, \widehat E \in \Mhat, \ \alpha_E \geq \alpha_0 \right\},
  \qquad
  \alpha_E := \frac{|E|}{|\widehat E|} \in (0,1],
  \label{eq:cutmesh}
\end{eqnarray}
where $\alpha_E$ is the so-called \emph{volume fraction} of the cell $E$ compared to the corresponding background cell $\hat{E}$ and $\alpha_0$ is a cut-off value for tiny cut cells. Cells with $E = \widehat E$ and therefore with $\alpha_E = 1$ are called \emph{Cartesian cells}, all other cells are called \emph{cut cells}, and we denote the set of cut cells by $\Mcut$. Cut cells are general polygons, which may be non-convex, and their volume fraction $\alpha_E$ can become arbitrarily small. Cut cells with $\alpha_E < \alpha_0 = 10^{-9}$, which appear for example when the boundary passes almost exactly through a vertex of the background mesh, a situation that in practice cannot always be avoided for general $\Omega$, are simply removed from the mesh; their faces then become part of the embedded boundary. 
Conversely, a cell of the background mesh with $\alpha_E = 1$ is nevertheless treated as a cut cell if one of its faces lies on a slip wall, i.e., if the embedded boundary coincides with this face. 
Apart from this, no special treatment of small cut cells is applied: in particular, cut cells are neither merged with their neighbors nor is the polynomial approximation degree of the DG schemes reduced on small cut cells. For a cell $E \in \Mh$ we denote by $\mathbf{n}_E$ its outward pointing unit normal vector.

\begin{figure}[tb]
  \centering
  \scalebox{2}{\input{cutcell_mesh.tex}}
  \caption{Typical cut cell mesh for the flow around a circular cylinder. Cut cells are marked in red. They have irregular shapes and can be arbitrarily small, i.e.\ their volume fraction $\alpha$ can be very close to zero, see zoom on the right.}
  \label{fig:cutcellmesh}
\end{figure}
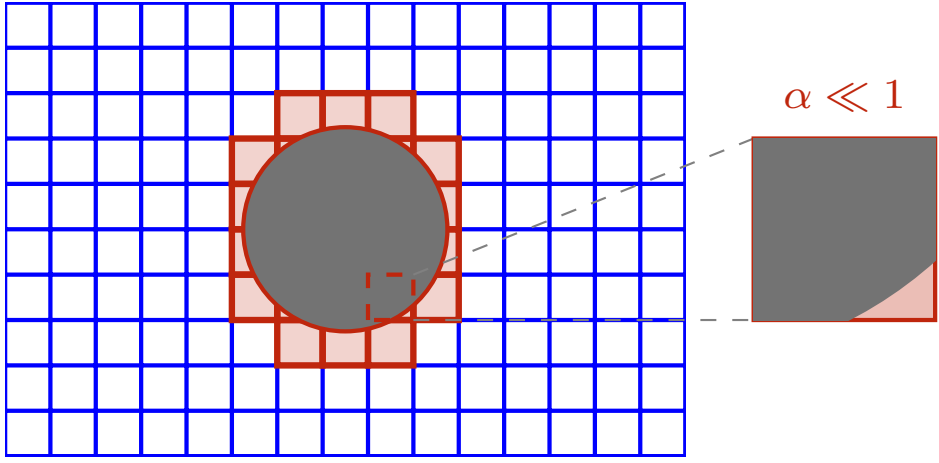

\paragraph{The small cell problem}
Throughout this paper we always apply the explicit CFL condition on the background mesh,
\begin{eqnarray}
  \Delta t = \frac{\mathrm{CFL}}{2N+1} \, \min_{E \in \mathcal{S}} \left( \frac{a_{x,E}}{h_x} + \frac{a_{y,E}}{h_y} \right)^{-1},
  \qquad \mathrm{CFL} = 0.4,
  \label{eq:cfl}
\end{eqnarray}
where $N$ denotes the polynomial degree and $a_{x,E} = \max_i |\lambda_i(\mathbf{u}_h^n,\mathbf{e}_x)|$ and $a_{y,E} = \max_i |\lambda_i(\mathbf{u}_h^n,\mathbf{e}_y)|$ are the maximal absolute eigenvalues in $x$- and $y$-direction, evaluated at the nodes of the discrete solution $\mathbf{u}_h^n$ of cell $E$ at time $t^n$. We set $\mathcal{S} = \Mexp$, the set of explicit cells defined in \cref{sec:classification} below, if there is at least one explicit cell, and $\mathcal{S} = \Mh$ otherwise. In all cases, $\Delta t$ is reduced, if necessary, such that prescribed intermediate output times and the final time are reached exactly. Explicit schemes are not stable on cut cells with $\alpha_E \ll 1$ for this choice of the time step, since their CFL condition would require a time step that tends to zero with the size of the cut cell, e.g.\ $\Delta t = \mathcal{O}(\alpha_E h)$ with $h = \max(h_x,h_y)$ for a thin cut cell that is parallel to a face of its background cell. This is the well-known \emph{small cell problem} of Cartesian cut cell methods \cite{ChernColella1987,BayyukPowellvanLeer1993,Quirk1994,BergerHelzelLeVeque2003,ColellaGravesKeenModiano2006,BergerHelzel2012,QinKrivodonova2013,MuellerKraemerEisKummerOberlack2016,MayBerger2017,EngwerMayNuessingStreitbuerger2020,BergerGiuliani2021,Giuliani2022}. The natural and primary goal of the new scheme presented in this work is therefore to use the time step \eqref{eq:cfl} \textit{independently} of the size of the cut cells. To achieve this we propose a mixed explicit-implicit method, ADER-ZIST, where the explicit cells are updated with a cheap ADER-DG scheme, while cut cells and their direct neighbors are updated via a zonally implicit space-time DG scheme which, in all our numerical experiments, remained stable irrespective of the cut cell sizes. Within the DG framework, the coupling between the different schemes is achieved in a very natural and conservative manner with the aid of the numerical flux. As we will see later, ADER-DG schemes are particularly well-suited for this kind of coupling thanks to their fully-discrete one-step nature, see also \cite{GaburroKlimaBonafiniTavelli2027} in the context of Lagrangian schemes on moving unstructured meshes in three space dimensions with topology changes, where a similar explicit-implicit coupling has recently been employed successfully for the locally implicit treatment of tiny space-time sliver elements.   

\paragraph{Time slabs and discrete spaces}
We partition the time interval $[0,T]$ into time slabs $\In = [t^n, t^{n+1}]$ with $t^{n+1} = t^n + \Delta t$, and for a cell $E \in \Mh$ we call $E \times \In$ a \emph{space-time cell}. On each space-time cell, the discrete solution is a space-time polynomial of degree $N$, where the degree refers to each coordinate direction on the Cartesian cells and to the total degree on the cut cells. On the Cartesian cells we use the tensor-product space $\mathbb{Q}_N(E) \otimes \mathbb{P}_N(\In)=\mathbb{Q}_N(E\times\In)$ with a \textit{nodal Lagrange basis}, while on the cut cells, which have a general polygonal shape, we use the space $\mathbb{P}_N(E \times \In)$ of polynomials of total degree $N$ in space \emph{and} time with a \textit{modal basis}, see \cref{sec:implementation} for details. 
We denote the local spatial and space-time spaces by
\begin{eqnarray}
  V(E) := \left\{ \begin{array}{ll} \mathbb{Q}_N(E), & E \notin \Mcut, \\ \mathbb{P}_N(E), & E \in \Mcut, \end{array} \right.
  \qquad
  W(E \times \In) := \left\{ \begin{array}{ll} \mathbb{Q}_N(E) \otimes \mathbb{P}_N(\In), & E \notin \Mcut, \\ \mathbb{P}_N(E \times \In), & E \in \Mcut, \end{array} \right.
  \label{eq:spaces}
\end{eqnarray}
and the global spatial DG space by
\begin{eqnarray}
  V_h := \left\{ \mathbf{v}_h \in L^2(\Omega)^m \;\, \big| \;\, \mathbf{v}_h|_E \in V(E)^m \ \ \forall E \in \Mh \right\}.
  \label{eq:Vh}
\end{eqnarray}
The discrete solution at time $t^n$ is denoted by $\mathbf{u}_h^n \in V_h$. We note that in both cases the restriction of a space-time polynomial $\theta \in W(E\times\In)$ to the time $t^{n+1}$ lies in $V(E)$, so that the space-time polynomial of an implicit cell directly defines the solution at the new time level. 
More precisely and throughout this paper, for the explicit Cartesian ADER-DG cells and for the implicit Cartesian transition cells, we use the same space-time tensor product of the \textit{nodal} Lagrange basis functions that pass through the $N+1$ Gauss--Legendre quadrature nodes in each coordinate direction. Instead, for the cut cells we use a \textit{modal} basis up to degree $N$ in space and time that employs Taylor monomials which are \textit{shifted} and \textit{rescaled} with respect to the \textit{bounding box} of the cut cell. Both bases are defined more precisely in \cref{sec:basis}, see in particular \eqref{eq:modalbasis} for the modal basis of the cut cells. 

\paragraph{Numerical flux}
The numerical flux is the \textit{key ingredient} for a consistent and conservative coupling of the two schemes: the explicit ADER-DG scheme, which is used on the Cartesian cells away from the embedded boundary, and the zonally implicit space-time DG scheme, which is used on all remaining cells.
We thus start the presentation of the scheme in an unusual way, namely with the numerical flux, rather than with the weak form.  
For a face $e \subset \partial E$ shared with a neighbor cell $E'$ we denote by $\mathbf{Q}^-$ and $\mathbf{Q}^+$ the traces of a function $\mathbf{Q}$ from inside $E$ and from inside $E'$, respectively, i.e.\ the so-called boundary-extrapolated values. 
Consequently, $\mathbf{F}^-$ and $\mathbf{F}^+$ denote the boundary-extrapolated fluxes, i.e.\ the traces of the discrete representations of the flux in $E$ and in $E'$, which in general differ from $\mathbf{F}(\mathbf{Q}^-)$ and $\mathbf{F}(\mathbf{Q}^+)$ and which are specified for each type of cell in \cref{sec:spacetime}. Throughout this paper, on all faces the interaction between two neighboring cells is governed by the Rusanov (local Lax--Friedrichs) flux \cite{Rusanov1961}, which is chosen to be a \textit{linear function} of \textit{four arguments}, rather than a nonlinear function of two arguments, as usual, see also \cite{DumbserKaeserTitarevToro2007}. It reads
\begin{eqnarray}
  \Fnum(\mathbf{F}^-,\mathbf{F}^+,\mathbf{Q}^-,\mathbf{Q}^+,\mathbf{n}) = \frac{1}{2} \left( \mathbf{F}^+ + \mathbf{F}^-  \right) \cdot \mathbf{n}
  - \frac{1}{2} \, s_{\max} \left( \mathbf{Q}^+ - \mathbf{Q}^- \right),
  \label{eq:numflux}
\end{eqnarray}
with the maximum signal speed given by 
\begin{eqnarray}
  s_{\max} = \max_{i} \left\{ |\lambda_i(\bar{\mathbf{Q}}^-,\mathbf{n})|, |\lambda_i(\bar{\mathbf{Q}}^+,\mathbf{n})| \right\},
  \label{eq:smax}
\end{eqnarray}
which is a \textit{parameter} rather than a function of the arguments. 
In addition to the particular formulation of the numerical flux, in our method the wave speed estimate $s_{\max}$ is \emph{frozen} at the old time level, i.e.\ the states $\bar{\mathbf{Q}}^\pm$ in \eqref{eq:smax} are the traces of the solution $\mathbf{u}_h^n$ at time $t^n$, while $\mathbf{F}^\pm$ and $\mathbf{Q}^\pm$ in \eqref{eq:numflux} are the traces of the space-time polynomials of the fluxes and the states in the left and right cell at the respective spatial and temporal quadrature point. Hence, $s_{\max}$ is computed only once per spatial quadrature point of a face and per time step and remains unchanged during the Newton iterations of the implicit scheme. This makes the flux linear in its four arguments, so that its derivative with respect to the degrees of freedom of the two cells only requires the Jacobians $\mathbf{A}(\mathbf{Q},\mathbf{n})$ of the physical flux, which enter through $\mathbf{F}^\pm$, and no derivative of the wave speed estimate with respect to the left and right states. The numerical flux is consistent, $\Fnum(\mathbf{F}(\mathbf{Q}),\mathbf{F}(\mathbf{Q}),\mathbf{Q},\mathbf{Q},\mathbf{n}) = \mathbf{F}(\mathbf{Q})\cdot\mathbf{n}$, for any value of $s_{\max}$. It is also conservative, $\Fnum(\mathbf{F}^-,\mathbf{F}^+,\mathbf{Q}^-,\mathbf{Q}^+,\mathbf{n}) = -\Fnum(\mathbf{F}^+,\mathbf{F}^-,\mathbf{Q}^+,\mathbf{Q}^-,-\mathbf{n})$, provided that both cells sharing a face use the same value of $s_{\max}$. Since \eqref{eq:smax} is symmetric in $\bar{\mathbf{Q}}^-$ and $\bar{\mathbf{Q}}^+$ and invariant under $\mathbf{n} \mapsto -\mathbf{n}$, this is guaranteed by computing $s_{\max}$ on both sides from the same traces $\bar{\mathbf{Q}}^\pm$ of $\mathbf{u}_h^n$.

For brevity, we write in the following $\Fnum(\mathbf{Q}^-,\mathbf{Q}^+,\mathbf{n})$, where the flux traces $\mathbf{F}^\pm$ that belong to $\mathbf{Q}^\pm$ are implied. For a detailed overview of further, more accurate, approximate Riemann solvers that may also be used inside our scheme, see the well-known textbook of Toro \cite{toro-book}. However, more sophisticated Riemann solvers may require the computation of more complex approximate Jacobian matrices of the numerical flux, while in the simple numerical flux chosen above only the original Jacobian matrix of the flux of the continuous PDE system is needed.

\paragraph{Boundary conditions}
All boundary conditions are imposed weakly via a suitable ghost state $\mathbf{Q}^+$ that is inserted as outer state into the numerical flux \eqref{eq:numflux}. At the embedded boundary, i.e.\ at the cut boundaries we impose slip wall (reflective) boundary conditions $\mathbf{v}\cdot\mathbf{n} = 0$ via the mirrored state, while at Dirichlet boundaries, where a boundary state $\mathbf{Q}_B$ is prescribed, for example from a known exact solution, we do not use $\mathbf{Q}_B$ directly, but choose the ghost state $\mathbf{Q}^+_{\rm Dir}$ such that the arithmetic average $\frac{\mathbf{Q}^+_{\rm Dir} + \mathbf{Q}^-}2$ with the interior state $\mathbf{Q}^-$ is $\mathbf{Q}_B$, i.e.
\begin{eqnarray}
  \mathbf{Q}^+_{\rm wall} = \left( \begin{array}{c} \rho^- \\ \rho^- \left( \mathbf{v}^- - 2 (\mathbf{v}^- \cdot \mathbf{n}) \, \mathbf{n} \right) \\ (\rho E)^- \end{array} \right),
  \qquad
  \mathbf{Q}^+_{\rm Dir} = 2 \, \mathbf{Q}_B - \mathbf{Q}^-.
  \label{eq:ghost}
\end{eqnarray}
The corresponding outer flux is $\mathbf{F}^+ = \mathbf{F}(\mathbf{Q}^+)$, and on boundary faces of cut cells also the inner flux $\mathbf{F}^- = \mathbf{F}(\mathbf{Q}^-)$ is evaluated pointwise rather than using the inner trace of the discrete flux. 
\subsection{A common starting point: weak formulation of the PDE system in space-time}\label{sec:spacetime}

Both building blocks of our mixed explicit-implicit ADER-space-time DG approach start from the same weak formulation of the governing PDE system in space-time: we multiply \eqref{eq:conslaw} by a test function $\theta \in W(E \times \In)$ and integrate over the space-time cell $E \times \In$, which yields
\begin{eqnarray}
  \int_{\In} \int_E \theta \left( \partial_t \mathbf{Q} + \nabla \cdot \mathbf{F}(\mathbf{Q}) \right) \dx \dt = 0.
  \label{eq:stweak}
\end{eqnarray}
In both methods the time derivative is integrated by parts and the initial value at time $t^n$ is given by the known solution $\mathbf{u}_h^n$ of the cell itself (upwinding in time). The two schemes differ in the way the interaction with the neighbor cells is taken into account. Considering the coupling directly within a time slab leads to a coupled implicit system (space-time DG), while 
considering it only \textit{a posteriori} in a corrector step leads to an element-local problem and an explicit update (ADER-DG).

\subsubsection{Explicit scheme: ADER-DG}\label{sec:ader}

Away from the cut cells, i.e.\ \textit{almost everywhere} in the computational domain, we use the explicit ADER-DG scheme in its more recent formulation given in  \cite{DumbserBalsaraToroMunz2008,DumbserEnauxToro2008,DumbserZanottiLoubereDiot2014,Zanotti2015}, see also \cite{GassnerDumbserHindenlangMunz2011,Frontiers2020}. It consists of two steps: an element-local space-time \emph{predictor} and a fully discrete one-step \emph{corrector}.

\textbf{ADER predictor step.} In the predictor step, \eqref{eq:stweak} is integrated by parts in \emph{time only}, while the flux divergence is deliberately kept in strong form without coupling to any neighbor elements. This results in the following element-local problem, in which the discrete predictor $\mathbf{p}_h$ takes the place of $\mathbf{Q}$: find $\mathbf{p}_h \in W(E \times \In)^m$ such that for all $\theta \in W(E \times \In)$
\begin{equation}
  \int_E \theta(\mathbf{x},t^{n+1}) \, \mathbf{p}_h(\mathbf{x},t^{n+1}) \dx - \int_E \theta(\mathbf{x},t^n) \, \mathbf{u}_h^n(\mathbf{x}) \dx
   - \int_{\In} \int_E \partial_t \theta \, \mathbf{p}_h \dx \dt
  + \int_{\In} \int_E \theta \, \nabla \cdot \mathbf{F}_h(\mathbf{p}_h) \dx \dt = 0,
  \label{eq:predictor}
\end{equation}
where, as usual in ADER-DG schemes, $\mathbf{F}_h(\mathbf{p}_h)$ denotes the representation of the flux in the same nodal space-time basis as $\mathbf{p}_h$, i.e.\ the flux is evaluated pointwise at the space-time nodes. The trace of $\mathbf{F}_h(\mathbf{p}_h)$ is the flux $\mathbf{F}^-$ of an explicit cell in \eqref{eq:numflux}. Since no face integrals appear in \eqref{eq:predictor}, there is \emph{no coupling} with neighbor cells, and $\mathbf{p}_h$ can be computed locally and independently on each cell. The predictor is a high order accurate local approximation of the solution of the Cauchy problem in the small with initial data $\mathbf{u}_h^n|_E$, which only uses information from inside $E$. The nonlinear system \eqref{eq:predictor} is solved by the discrete Picard iteration
\begin{equation}
   \int_E \theta(\mathbf{x},t^{n+1}) \, \mathbf{p}_h^{r+1}(\mathbf{x},t^{n+1}) \dx - \int_{\In} \int_E \partial_t \theta \, \mathbf{p}_h^{r+1} \dx \dt
 = \int_E \theta(\mathbf{x},t^n) \, \mathbf{u}_h^n(\mathbf{x}) \dx - \int_{\In} \int_E \theta \, \nabla \cdot \mathbf{F}_h(\mathbf{p}_h^r) \dx \dt, 
  \label{eq:picard}
\end{equation}
with iteration index $r = 0, 1, \ldots$, 
starting from the initial guess $\mathbf{p}_h^0(\mathbf{x},t) = \mathbf{u}_h^n(\mathbf{x})$, which is constant in time. For linear homogeneous systems with constant coefficients the iteration matrix of \eqref{eq:picard} is nilpotent \cite{Jackson2017}, so that the iteration yields the exact solution of \eqref{eq:predictor} after a finite number of iterations, while for general nonlinear systems it is contractive and thus converges for sufficiently small time steps, i.e.\ under a CFL-type condition that depends on the Lipschitz constant of the flux, the time step and the size of the cell $E$, see  \cite[Sec.~2.6]{Frontiers2020} for details.
In our implementation we always carry out a fixed number of $N+2$ Picard iterations for simplicity; all our numerical experiments indicate that this is enough to retain the nominal order of accuracy.

\textbf{ADER corrector step.} In the corrector step, the solution is updated with the standard DG discretization in space, which is integrated over the time slab $\In$, where all space-time integrals are evaluated with the predictor: find $\mathbf{u}_h^{n+1}|_E \in V(E)^m$ such that for all $\varphi \in V(E)$
\begin{eqnarray}
  \int_E \varphi \, \mathbf{u}_h^{n+1} \dx = \int_E \varphi \, \mathbf{u}_h^{n} \dx
  + \int_{\In} \int_E \nabla \varphi \cdot \mathbf{F}_h(\mathbf{p}_h) \dx \dt
  - \int_{\In} \int_{\partial E} \varphi \, \Fnum(\mathbf{p}_h^-,\mathbf{p}_h^+,\mathbf{n}_E) \ds \dt.
  \label{eq:corrector}
\end{eqnarray}
Since \eqref{eq:corrector} only involves the known predictors of $E$ and of its neighbors, the corrector is an explicit fully-discrete one-step update. The resulting scheme is of nominal order $N+1$ in space and time and is stable under the CFL condition \eqref{eq:cfl}, see \cite{DumbserBalsaraToroMunz2008}. We stress that on a small cut cell not only the explicit corrector would violate its CFL condition, but that also the convergence of the Picard iteration \eqref{eq:picard} of the predictor can no longer be guaranteed, since the contraction condition of \cite[Sec.~2.6]{Frontiers2020} involves the local size of the cell, while the time step \eqref{eq:cfl} is determined by the background mesh.

\subsubsection{Implicit scheme: space-time DG}\label{sec:stdg}

On cut cells and their direct neighbors (see \cref{sec:mixed}), we use the classical space-time DG scheme
of van der Vegt \textit{et al.} 
 \cite{vanderVegtvanderVen2002,KlaijvanderVegtvanderVen2006}, in which the coupling with the neighbor cells is taken into account inside the time slab via the numerical flux. On the Cartesian cells of the implicit region, \eqref{eq:stweak} is integrated by parts in \emph{time and space}, which results in the following problem: find $\mathbf{q}_h \in W(E \times \In)^m$ such that for all $\theta \in W(E \times \In)$
\begin{eqnarray}
  & & \int_E \theta(\mathbf{x},t^{n+1}) \, \mathbf{q}_h(\mathbf{x},t^{n+1}) \dx - \int_E \theta(\mathbf{x},t^n) \, \mathbf{u}_h^n(\mathbf{x}) \dx
  - \int_{\In} \int_E \left( \partial_t \theta \, \mathbf{q}_h + \nabla \theta \cdot \mathbf{F}(\mathbf{q}_h) \right) \dx \dt
  \nonumber \\
  & & + \int_{\In} \int_{\partial E} \theta \, \Fnum(\mathbf{q}_h^-,\mathbf{w}_h^+,\mathbf{n}_E) \ds \dt = 0,
  \label{eq:stdg}
\end{eqnarray}
and we set $\mathbf{u}_h^{n+1}|_E := \mathbf{q}_h(\cdot,t^{n+1})$. As on the explicit cells, the flux $\mathbf{F}^-$ of a transition cell (a Cartesian face neighbor of a cut cell, see \cref{sec:classification}) in \eqref{eq:numflux} is the trace of the flux $\mathbf{F}_h(\mathbf{q}_h)$ in the nodal space-time basis. Here, $\mathbf{w}_h^+$ denotes the outer trace on the space-time face $e \times \In$, which will be specified later in \cref{sec:threestep} when we discuss the coupling between the different schemes employed in ADER-ZIST. On the cut cells, we use the so-called \emph{strong form} 
of the space-time DG scheme, in which the flux divergence is integrated by parts a second time, as in the strong form of the ADER-DG corrector introduced in \cite[Eq.~(10)]{DumbserLoubere2016}: find $\mathbf{q}_h \in W(E \times \In)^m$ such that for all $\theta \in W(E \times \In)$
\begin{eqnarray}
  & & \int_E \theta(\mathbf{x},t^{n+1}) \, \mathbf{q}_h(\mathbf{x},t^{n+1}) \dx - \int_E \theta(\mathbf{x},t^n) \, \mathbf{u}_h^n(\mathbf{x}) \dx
  - \int_{\In} \int_E \partial_t \theta \, \mathbf{q}_h \dx \dt
  + \int_{\In} \int_E \theta \, \nabla \cdot \Pi_h \mathbf{F}(\mathbf{q}_h) \dx \dt
  \nonumber \\
  & & + \int_{\In} \int_{\partial E} \theta \left( \Fnum(\mathbf{q}_h^-,\mathbf{w}_h^+,\mathbf{n}_E) - \left( \Pi_h \mathbf{F}(\mathbf{q}_h) \right)^- \cdot \mathbf{n}_E \right) \ds \dt = 0.
  \label{eq:stdgstrong}
\end{eqnarray}
Here, $\Pi_h$ denotes the componentwise $L^2$ projection onto $W(E \times \In)$, i.e.\ $\int_{\In} \int_E \psi \, (\Pi_h \mathbf{F}(\mathbf{q}_h) - \mathbf{F}(\mathbf{q}_h)) \dx \dt = 0$ for all $\psi \in W(E \times \In)$, and $(\Pi_h \mathbf{F}(\mathbf{q}_h))^-$ denotes its trace on $\partial E$, which is also the flux $\mathbf{F}^-$ of a cut cell in \eqref{eq:numflux} on interior faces. Integrating the volume term of \eqref{eq:stdgstrong} by parts and using that $\nabla\theta \in W(E \times \In)^2$ for $E \in \Mcut$, one finds that \eqref{eq:stdgstrong} coincides with \eqref{eq:stdg} under exact integration. We note that for this it is essential that the trace of the projected flux $\Pi_h \mathbf{F}(\mathbf{q}_h)$, and not the one of $\mathbf{F}(\mathbf{q}_h)$, is subtracted from the numerical flux in the boundary term, since otherwise the difference $\int_{\In}\int_{\partial E} \theta \, ( \Pi_h \mathbf{F}(\mathbf{q}_h) - \mathbf{F}(\mathbf{q}_h^-) ) \cdot \mathbf{n}_E \ds \dt$ would remain, which in general destroys conservation for nonlinear fluxes.

In contrast to \eqref{eq:predictor}, the face integrals in \eqref{eq:stdg} and \eqref{eq:stdgstrong} involve the \emph{unknown} space-time polynomials of the neighboring cells, so that the space-time DG scheme couples neighboring cells and leads to a nonlinear system of equations, i.e.\ to a  \emph{zonally implicit} scheme. It is therefore not subject to the CFL condition of an explicit scheme on the cut cells, and in all computations of \cref{sec:numerics} it is used with the time step \eqref{eq:cfl} computed on the uniform background mesh, supplemented by \eqref{eq:cflvisc} when the artificial viscosity is active, independently of the size of the small cut cells.

\subsection{The mixed explicit-implicit ADER-space-time DG scheme}\label{sec:mixed}

\subsubsection{Main idea and classification of cells}\label{sec:classification}

The main idea of ADER-ZIST follows the mixed explicit-implicit FV scheme of May and Berger \cite{MayBerger2017,MayLaakmann2024}: the cut cells and their direct neighbors are treated implicitly for stability reasons, but only \emph{zonally}, while away from the cut cells and their neighbors the cheaper explicit ADER-DG scheme is used everywhere. More precisely, we partition $\Mh$ into
\begin{itemize}
  \item the set of \emph{cut cells} $\Mcut$,
  \item the set of \emph{transition cells} $\Mtr$, consisting of all Cartesian cells that share a face with a cut cell,
  \item the set of \emph{explicit cells} $\Mexp := \Mh \setminus (\Mcut \cup \Mtr)$.
\end{itemize}
We refer to $\Mimp := \Mcut \cup \Mtr$ as the set of \emph{implicit cells}. On the implicit cells we use the space-time DG scheme, i.e.\ \eqref{eq:stdg} on the transition cells and \eqref{eq:stdgstrong} on the cut cells, while on the explicit cells we use the ADER-DG scheme \eqref{eq:predictor}--\eqref{eq:corrector}. All cut cells are treated implicitly, irrespective of their volume fraction. By construction, a cut cell never shares a face with an explicit cell. The implicit region for a typical flow around a cylinder example is shown in \cref{fig:implicitregion}. As in the FV setting \cite{MayBerger2017}, the Cartesian face neighbors of the cut cells are included in the implicit region, since the domain of dependence of a neighbor of a small cut cell includes the cells on the opposite side of the small cut cell, which can only be reached within one time step if the information is passed \emph{through} the small cut cell. In our practical implementation, the type of each cell is stored in a flag array, with the value 0 for the \textit{explicit} ADER-DG cells using a \textit{nodal} tensor-product space-time basis, 1 for the cut cells using the \textit{implicit} space-time DG scheme based on a \textit{modal} space-time basis, which has fewer degrees of freedom than the tensor-product basis of the same degree $N$, and 2 for the transition cells, which use the \textit{implicit} space-time DG method with the same nodal tensor-product space-time basis as the explicit ADER-DG scheme and which are therefore particularly well-suited for the explicit-implicit coupling proposed in this paper.

\begin{figure}[tb]
  \centering
  \scalebox{1.5}{\input{implicit_region.tex}}
  \caption{Implicit region for the flow around a cylinder example: cut cells and their Cartesian face neighbors (transition cells) are treated implicitly (orange).
    All other cells are treated explicitly (blue).}
  \label{fig:implicitregion}
\end{figure}
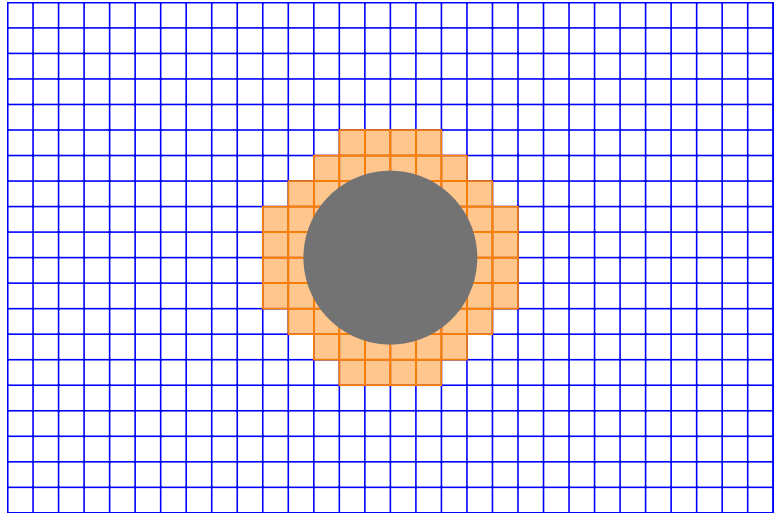

On each face, the numerical flux is evaluated with the space-time polynomials of the two adjacent cells: the ADER predictor ${\color{blue}\mathbf{p}_h}$ if the neighbor is an explicit cell, and the space-time DG polynomial ${\color{orange}\mathbf{q}_h}$ if the neighbor is an implicit cell. In particular, on a face between an explicit cell and a transition cell, the flux is evaluated as $\Fnum({\color{blue}\mathbf{p}_h},{\color{orange}\mathbf{q}_h},\mathbf{n})$ from the explicit side and as $\Fnum({\color{orange}\mathbf{q}_h},{\color{blue}\mathbf{p}_h},-\mathbf{n})$ from the transition side. Note that at the explicit-implicit boundary, the element space-time basis functions are exactly the same, i.e.\ the coupling can be done immediately and without any overhead, making full use of the tensor-product structure of the basis on each side of the interface. This is illustrated in \cref{fig:mixedidea} in more detail for a one-dimensional configuration.
 
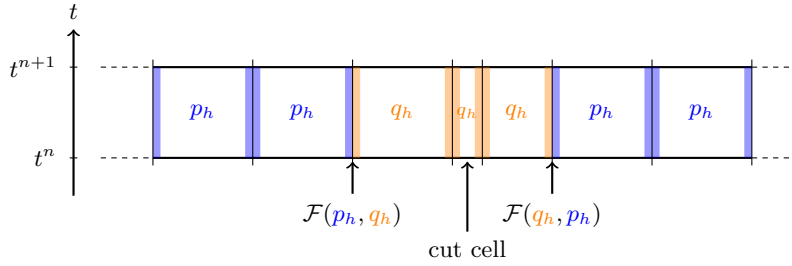
\begin{figure}[tb]
  \centering
  \input{mixed_idea.tex}
  \caption{Idea of the mixed explicit-implicit scheme in one space dimension. 	Explicit cells use the (explicit) ADER
    predictor $\pc$, the small cut cell and its face neighbors use the (implicit)
    space-time DG polynomial $\qc$. The two schemes are coupled via the arguments of the numerical flux.}
  \label{fig:mixedidea}
\end{figure}

\subsubsection{The three-step algorithm}\label{sec:threestep}

Given $\mathbf{u}_h^n$, one time step of ADER-ZIST consists of the following three steps, which are illustrated in \cref{fig:threestep}.
\begin{description}
  \item[Step 1 (explicit ADER predictor).] For every explicit cell $E \in \Mexp$, compute the ADER predictor ${\color{blue}\mathbf{p}_h}|_E \in W(E \times \In)^m$ by solving \eqref{eq:predictor}. This is an element-local computation.
  \item[Step 2 (implicit space-time DG).] Compute the space-time DG polynomials ${\color{orange}\mathbf{q}_h}|_E \in W(E \times \In)^m$ of all implicit cells $E \in \Mimp$ simultaneously by solving \eqref{eq:stdg} on the transition cells and \eqref{eq:stdgstrong} on the cut cells, where the outer trace in the numerical flux is given by
    \begin{eqnarray}
      \mathbf{w}_h^+ = \left\{ \begin{array}{ll}
        {\color{orange}\mathbf{q}_h}|_{E'} & \text{if the neighbor } E' \text{ is an implicit cell}, \\
        {\color{blue}\mathbf{p}_h}|_{E'} & \text{if the neighbor } E' \text{ is an explicit cell}, \\
        \mathbf{Q}^+({\color{orange}\mathbf{q}_h^-}) & \text{on the embedded or on the outer boundary},
      \end{array} \right.
      \label{eq:outertrace}
    \end{eqnarray}
    with the ghost state $\mathbf{Q}^+$ defined in \eqref{eq:ghost}. As the explicit predictors are known from Step~1, this results in one nonlinear system that couples only the unknowns of the implicit cells. Set $\mathbf{u}_h^{n+1}|_E := {\color{orange}\mathbf{q}_h}|_E(\cdot,t^{n+1})$ for $E \in \Mimp$.
  \item[Step 3 (explicit ADER corrector).] For every explicit cell $E \in \Mexp$, compute $\mathbf{u}_h^{n+1}|_E$ with the corrector \eqref{eq:corrector}, where the outer trace $\mathbf{w}_h^+$ in the numerical flux is again given by \eqref{eq:outertrace}, i.e.\ ${\color{blue}\mathbf{p}_h}|_{E'}$ if the neighbor $E'$ is an explicit cell and ${\color{orange}\mathbf{q}_h}|_{E'}$ if it is a transition cell. On the outer boundary of the domain, the ghost state \eqref{eq:ghost} is evaluated from the interior predictor trace, i.e.\ $\mathbf{w}_h^+ = \mathbf{Q}^+({\color{blue}\mathbf{p}_h^-})$.
\end{description}

\begin{figure}[tb]
  \centering
  \begin{subfigure}{\textwidth}\centering
    \input{step1.tex}
  \caption{Step 1: compute the (explicit) ADER-DG predictor $\pc$.}
  \end{subfigure}\\[0.3cm]
  \begin{subfigure}{\textwidth}\centering
    \input{step2.tex}
  \caption{Step 2: using $\pc$ as flux argument, e.g., $\Fnum(\pc,\qc)$, compute the
      (implicit) space-time DG solution $\qc$.}
  \end{subfigure}\\[0.3cm]
  \begin{subfigure}{\textwidth}\centering
    \input{step3.tex}
  \caption{Step 3: update the remaining (explicit) cells to obtain the full
      solution at $t^{n+1}$.}
  \end{subfigure}
  \caption{The three steps of the mixed ADER-space-time DG scheme in one space dimension. Colored
    vertical lines indicate the faces whose fluxes are used in the respective
    step, colored horizontal lines the cells that compute a solution at time $t^{n+1}$.}
  \label{fig:threestep}
\end{figure}
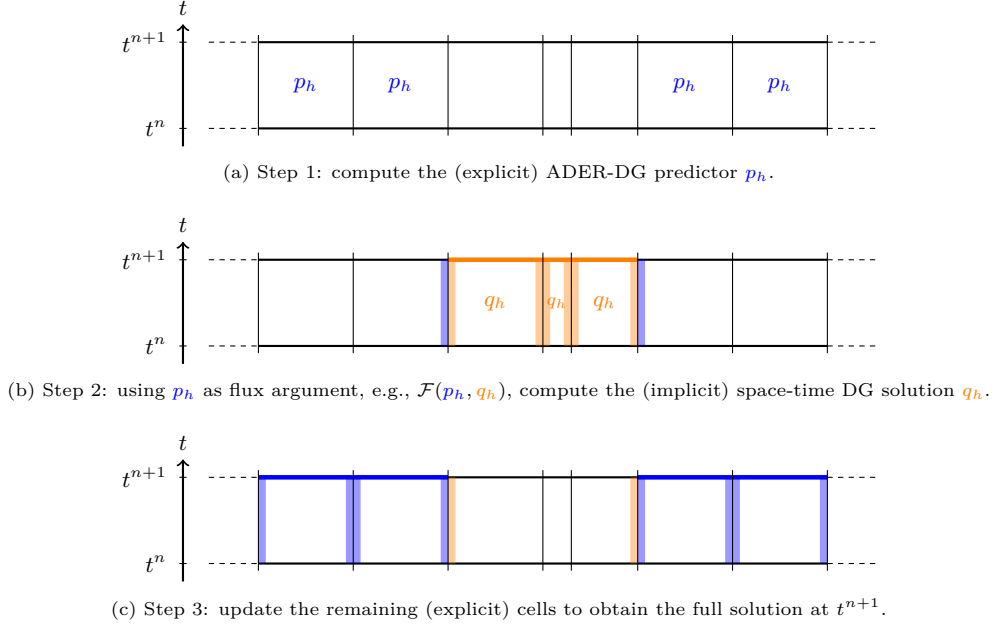

We note that Steps~1 and 3 are exactly the standard ADER-DG scheme on all explicit cells, except that in Step~3 the neighboring predictor on the faces adjacent to transition cells is replaced by the space-time polynomial obtained by the implicit space-time DG scheme. The nonlinear algebraic system of equations that arises from the implicit discretization contained in Step~2, which involves the cut cells and the transition cells in one single monolithic system, is only as large as the implicit region itself, which is the thin layer of cut cells along the embedded boundary and their direct neighbors. In two space dimensions the number of implicit cells grows only like $\mathcal{O}(h^{-1})$ under mesh refinement, compared to $\mathcal{O}(h^{-2})$ for the whole mesh.

\subsubsection{Conservation}\label{sec:conservation}

The coupling of the explicit ADER-DG cells with the implicit space-time DG cells via the numerical flux makes the mixed scheme naturally conservative: on a face $e$ between an explicit cell $E'$ and a transition cell $E$, the explicit cell uses $\Fnum(\mathbf{p}_h|_{E'},\mathbf{q}_h|_E,\mathbf{n}_{E'})$ in Step~3 and the transition cell uses $\Fnum(\mathbf{q}_h|_E,\mathbf{p}_h|_{E'},\mathbf{n}_E)$ in Step~2. Note that the traces of the discrete flux space-time polynomials are implied in our notation to facilitate readability, but this fact should be explicitly kept in mind. Since $\mathbf{n}_E = -\mathbf{n}_{E'}$ and both fluxes use the same frozen wave speed $s_{\max}$, computed from the traces of $\mathbf{u}_h^n$, the conservation property of the numerical flux implies that these two fluxes cancel.

\begin{lemma}[Conservation]\label[lemma]{lem:conservation}
  Assume periodic boundary conditions, quadrature rules that use the same quadrature points and weights on both sides of each face and that integrate polynomials of degree at most $N$ exactly on each cut cell and on each of its faces, and an exact solve of the implicit system. Then the mixed scheme is conservative, i.e.
  \begin{eqnarray}
    \int_\Omega \mathbf{u}_h^{n+1} \dx = \int_\Omega \mathbf{u}_h^{n} \dx.
    \label{eq:conservation}
  \end{eqnarray}
  With slip wall boundary conditions \eqref{eq:ghost} on the embedded boundary, i.e.\ on boundary faces of cut cells, the same holds for the mass and the total energy.
\end{lemma}
\begin{proof}
  For an explicit cell $E \in \Mexp$, choosing $\varphi = 1$ in \eqref{eq:corrector} gives
  \begin{eqnarray}
    \int_E \mathbf{u}_h^{n+1} \dx = \int_E \mathbf{u}_h^n \dx - \int_{\In} \int_{\partial E} \Fnum({\color{blue}\mathbf{p}_h^-},\mathbf{w}_h^+,\mathbf{n}_E) \ds \dt,
    \label{eq:conservation-exp}
  \end{eqnarray}
  with $\mathbf{w}_h^+ \in \{ {\color{blue}\mathbf{p}_h^+}, {\color{orange}\mathbf{q}_h^+} \}$ as described in Step~3. 
  For an implicit cell $E \in \Mimp$, we choose $\theta = 1$ in \eqref{eq:stdg} and \eqref{eq:stdgstrong}, so that $\partial_t \theta = 0$ and $\nabla \theta = 0$. On a transition cell, all volume terms of \eqref{eq:stdg} then vanish. On a cut cell, the volume term of \eqref{eq:stdgstrong} reduces to $\int_{\In} \int_E \nabla \cdot \mathbf{P} \dx \dt$ with $\mathbf{P} := \Pi_h \mathbf{F}(\mathbf{q}_h)$, and the boundary term contains $-\int_{\In} \int_{\partial E} \mathbf{P}^- \cdot \mathbf{n}_E \ds \dt$. Since $\mathbf{P}(\cdot,t)$ is a polynomial of degree at most $N$ in space, these two terms cancel at every quadrature point by the divergence theorem, which the quadrature reproduces exactly by assumption. This argument only requires the exact integration of polynomials, but not of the nonpolynomial flux $\mathbf{F}(\mathbf{q}_h)$, and it holds irrespective of how the projection $\Pi_h$ is evaluated.
  Since, moreover, $\mathbf{u}_h^{n+1}|_E = \mathbf{q}_h|_E(\cdot,t^{n+1})$ according to Step~2, we obtain in both cases
  \begin{eqnarray}
    \int_E \mathbf{u}_h^{n+1} \dx = \int_E \mathbf{u}_h^n \dx - \int_{\In} \int_{\partial E} \Fnum({\color{orange}\mathbf{q}_h^-},\mathbf{w}_h^+,\mathbf{n}_E) \ds \dt,
    \label{eq:conservation-imp}
  \end{eqnarray}
  with the outer trace $\mathbf{w}_h^+ \in \{ {\color{blue}\mathbf{p}_h^+}, {\color{orange}\mathbf{q}_h^+} \}$ given by \eqref{eq:outertrace}, i.e.\ the inner state and flux traces in the numerical flux are now those of the space-time DG polynomial ${\color{orange}\mathbf{q}_h}$ of the cell itself and not those of an ADER predictor. Summing \eqref{eq:conservation-exp} and \eqref{eq:conservation-imp} over all cells, the contributions of each interior face cancel, since both adjacent cells use the same state and flux traces in the numerical flux, with the two sides interchanged, the same frozen wave speed $s_{\max}$, the same quadrature points and weights and opposite normal vectors. The same holds for the artificial viscosity of \cref{sec:av}, whose volume and stabilization terms vanish for $\theta = 1$ and whose face terms are single-valued. At a slip wall on a boundary face of a cut cell, where also the inner flux $\mathbf{F}^- = \mathbf{F}(\mathbf{Q}^-)$ is evaluated pointwise, the mirrored state \eqref{eq:ghost} has the same density and total energy as the interior state and the opposite normal velocity, so that the mass and energy components of \eqref{eq:numflux} vanish; by construction, the same holds for the viscous fluxes of mass and total energy at a slip wall, see \cref{sec:av}.
\end{proof}

The quadrature rules of \cref{sec:implementation} satisfy the assumption of \cref{lem:conservation} on all cut cells with straight faces. On cut cells with curved boundary faces, the isoparametric map increases the polynomial degree of the integrands, so that for higher polynomial degrees $N$ the divergence theorem for $\Pi_h \mathbf{F}(\mathbf{q}_h)$, and hence \eqref{eq:conservation}, only holds up to the quadrature error.

In all two-dimensional Euler computations of \cref{sec:numerics}, including the test problems with shocks and artificial viscosity, which use quadrature and an inexact solution of the nonlinear system, the relative balance defect $|M_c^{n+1} - M_c^{n} + B_c^n| / |M_c^0|$ of the mass and of the total energy is below $2 \cdot 10^{-11}$ per time step, also on curved cut cells. Here, $M_c^n$ denotes the integral of the corresponding component of $\mathbf{u}_h^n$ over $\Omega$, evaluated with the quadrature rules of the scheme, and $B_c^n$ the integral of the outward numerical flux over the outer boundary of the domain and the time slab $\In$. Apart from quadrature errors on curved cut cells, this defect stems also from the residual of the implicit system that is only solved up to a given tolerance, tested with $\theta = 1$, since the cell averages of the implicit cells are taken from the final Newton iterate. Conservation up to machine precision, independently of the solver tolerance, could be obtained by treating the final iterate $\mathbf{q}_h$ on the implicit cells as a predictor and computing $\mathbf{u}_h^{n+1}$ there with an ADER-type corrector analogous to \eqref{eq:corrector}, i.e.\ from $\mathbf{u}_h^n$ and the volume and face fluxes of $\mathbf{q}_h$. Since time-independent test functions belong to $W(E \times \In)$, this corrector reproduces $\mathbf{q}_h(\cdot,t^{n+1})$ for an exact solve of the implicit system.

\begin{remark}[Relation to flux bounding]
  The coupling described above is the DG analogue of \emph{flux bounding} introduced in \cite{MayBerger2017}: there, the fluxes computed by the explicit scheme on faces adjacent to transition cells are used as input for the implicit scheme. Here, the role of these fluxes is taken over by the ADER predictor $\mathbf{p}_h$, which is computed for all explicit cells in Step 1 of the three-step algorithm. In the FV setting, combining a $p$-th order explicit and a $p$-th order implicit scheme in a way that results in a $p$-th order mixed scheme turned out to be very difficult \cite{MayLaakmann2024}. For ADER-ZIST, the consistent treatment of time in both schemes (\cref{sec:match}) suggests that full accuracy is retained, which is confirmed by our numerical results in \cref{sec:numerics}.
\end{remark}

\subsubsection{Why ADER-DG and space-time DG are a good match for explicit-implicit coupling}\label{sec:match}

Comparing \eqref{eq:predictor} with \eqref{eq:stdg}, we observe that the ADER predictor $\mathbf{p}_h$ of the explicit Cartesian cells and the space-time DG polynomial $\mathbf{q}_h$ of the implicit Cartesian transition cells are elements of the same local space-time space $W(E \times \In)^m$, that both use $\mathbf{u}_h^n$ as initial value at time $t^n$ via the same upwinding in time, and that both have exactly the same nodal tensor-product structure. The only difference is that the predictor neglects the interaction with the neighbor cells inside the time slab, which is taken into account only \textit{a posteriori} by the numerical flux in the corrector \eqref{eq:corrector}, while the space-time DG scheme accounts for it directly through the numerical flux in its weak formulation in space-time, see \eqref{eq:stdg}. In other words, the predictor \eqref{eq:predictor} is the strong form \eqref{eq:stdgstrong} of the space-time DG scheme in which the fluctuations on the spatial faces have been switched off, up to the representation of the flux. For the nodal tensor-product basis through the Gauss--Legendre points and the associated tensor-product Gauss--Legendre quadrature used on the Cartesian cells, see \cref{sec:basis}, the two representations of the flux coincide, since the projection $\Pi_h \mathbf{F}(\mathbf{q}_h)$ evaluated with this quadrature reduces to the pointwise evaluation of the flux at the space-time nodes, i.e.\ to $\mathbf{F}_h(\mathbf{q}_h)$, so that in this case the predictor \eqref{eq:predictor} is exactly the strong form \eqref{eq:stdgstrong} of the space-time DG scheme with the fluctuations on the spatial faces switched off. Therefore, on a face between an explicit and an implicit cell, the numerical flux couples two space-time polynomials of the same kind and of the same order of accuracy in time, and no interpolation or extrapolation in time is needed at the interface between the explicit and the implicit region. This makes it natural to couple the two schemes directly via the numerical flux, using the explicit ADER predictor $\mathbf{p}_h$ as outer trace in the space-time DG scheme \eqref{eq:stdg} and the implicit space-time solution $\mathbf{q}_h$ as outer trace in the explicit ADER corrector step \eqref{eq:corrector}. We also emphasize that the fully-discrete one-step nature of the ADER scheme is again, as in our previous work on local time stepping \cite{LTS,ALELTS1D,AMR3DCL}, the key feature that allows a very simple and natural coupling. The implicit space-time DG scheme needs information from the explicit region for the entire time step of size $\Delta t$. This information is not readily available in classical Runge-Kutta DG (RKDG) schemes \cite{CBS4}, which are advanced only from one stage to the next and which obtain the full temporal information only after the last Runge-Kutta stage has completed. Simultaneously, explicit RKDG schemes would need the information from the implicit space-time DG scheme as boundary condition in each Runge-Kutta stage, but this information is not computable until the last Runge-Kutta stage has finished. A coupling of the implicit space-time DG scheme with an explicit RKDG scheme would therefore require additional ingredients, such as a prediction of the temporal evolution in the explicit region, a special coupling of the Runge-Kutta stages or coupling iterations, while the ADER predictor directly provides the required space-time information before the implicit solve. Furthermore, there are well-known issues with imposing proper time-dependent Dirichlet boundary conditions for high order schemes in space that are based on Runge-Kutta time integration, see \cite{ClainRKBCProblems} for an interesting study of this problem with more details.  

A related situation arises in direct Arbitrary-Lagrangian-Eulerian ADER-DG schemes on moving Voronoi meshes with topology changes \cite{GaburroBoscheriChiocchettiKlingenbergSpringelDumbser2020}. There, degenerate space-time \emph{sliver} elements appear, whose spatial volume vanishes at time $t^n$, so that no initial condition can be imposed in the local predictor. Their predictor is therefore obtained by upwinding on the whole space-time boundary of the sliver element, using the already known predictors of the neighbor elements as input for the numerical flux. In both cases, the missing or too small local spatial volume is compensated by taking the information from the neighbor cells into account already in the local space-time problem, see the recent work of Gaburro \textit{et al.} \cite{GaburroKlimaBonafiniTavelli2027}. Finally, we note that, in contrast to cell merging or to state redistribution \cite{BergerGiuliani2021,Giuliani2022}, which modify the explicit update on small cut cells by the construction of neighborhoods with sufficiently large volume, the present approach keeps the original cut cells and does not require the construction of merging or redistribution neighborhoods,
since the cut cells are handled by the implicit space-time DG scheme.

\subsection{Shock capturing via artificial viscosity}\label{sec:av}

For problems with discontinuities, high order DG schemes require an additional mechanism to suppress spurious oscillations. For ADER-DG schemes on Cartesian and unstructured meshes, the \textit{a posteriori} subcell FV limiter \cite{DumbserZanottiLoubereDiot2014,DumbserLoubere2016} is a robust choice. It relies, however, on a subcell tessellation of each cell, which is not naturally available on cut cells of general shape, and on an \textit{a posteriori} recomputation of troubled cells, which would be cumbersome for the cells of the implicit region, which are coupled with each other. We therefore use artificial viscosity instead, which is switched on only for the test problems with shock waves in \cref{sec:num-shocks}, i.e.\ instead of \eqref{eq:conslaw} we solve the regularized version of the PDE 
\begin{eqnarray}
  \partial_t \mathbf{Q} + \nabla \cdot \mathbf{F}(\mathbf{Q}) = \nabla \cdot \left( \boldsymbol{\varepsilon}_E \nabla \mathbf{Q} \right) = \partial_x \left( \varepsilon_{x,E} \, \partial_x \mathbf{Q} \right) + \partial_y \left( \varepsilon_{y,E} \, \partial_y \mathbf{Q} \right),
  \qquad \boldsymbol{\varepsilon}_E = \mathrm{diag}(\varepsilon_{x,E},\varepsilon_{y,E}),
  \label{eq:av}
\end{eqnarray}
where the viscosity coefficients $\varepsilon_{x,E}$ and $\varepsilon_{y,E}$ are constant in each cell $E$ and in each time step and are applied componentwise to all conserved variables.

\paragraph{Discretization}
The viscous terms are discretized with the non-symmetric interior penalty Galerkin (NIPG) method \cite{OdenBabuskaBaumann1998,RiviereWheelerGirault1999} in the form used in \cite[Sec.~3.2]{HTCDG}, whose viscous numerical flux is based on \cite{GassnerLoercherMunz2007}. For a space-time polynomial $\mathbf{w}_h$ on the cell $E$ and a test function $\theta$, we define
\begin{eqnarray}
  \mathcal{A}_E(\mathbf{w}_h,\theta) & = & \int_{\In} \int_E \nabla \theta \cdot \left( \boldsymbol{\varepsilon}_E \nabla \mathbf{w}_h \right) \dx \dt
  - \int_{\In} \int_{\partial E} \theta \, \frac{1}{2} \left( \boldsymbol{\varepsilon}_e \nabla \mathbf{w}_h^- + \boldsymbol{\varepsilon}_e \nabla \mathbf{w}_h^+ \right) \cdot \mathbf{n}_E \ds \dt
  \nonumber \\
  & & + \int_{\In} \int_{\partial E} \frac{1}{2} \left( \boldsymbol{\varepsilon}_e \nabla \theta \right) \cdot \mathbf{n}_E \left( \mathbf{w}_h^- - \mathbf{w}_h^+ \right) \ds \dt
  + \int_{\In} \int_{\partial E} \eta_e \, \theta \left( \mathbf{w}_h^- - \mathbf{w}_h^+ \right) \ds \dt,
  \label{eq:nipg}
\end{eqnarray}
with the face viscosity $\boldsymbol{\varepsilon}_e = \mathrm{diag}(\varepsilon_{x,e},\varepsilon_{y,e})$, $\varepsilon_{x,e} = \min(\varepsilon_{x,E},\varepsilon_{x,E'})$ and $\varepsilon_{y,e} = \min(\varepsilon_{y,E},\varepsilon_{y,E'})$ on a face shared with the neighbor $E'$ and $\boldsymbol{\varepsilon}_e = \boldsymbol{\varepsilon}_E$ on boundary faces, although harmonic means, as they are common for DG discretizations of elliptic operators, would also be possible.
The penalty coefficient is chosen as
\begin{eqnarray}
  \eta_e = (2N+1) \left( \frac{\varepsilon_{x,e} \, n_x^2}{h_{x,e}} + \frac{\varepsilon_{y,e} \, n_y^2}{h_{y,e}} \right),
  \label{eq:penalty}
\end{eqnarray}
where $h_{x,e} = \min(h_{x,E},h_{x,E'})$ and $h_{y,e} = \min(h_{y,E},h_{y,E'})$ on a face shared with the neighbor $E'$ and $h_{x,e} = h_{x,E}$, $h_{y,e} = h_{y,E}$ on boundary faces, with the local cell sizes $h_{x,E}$ and $h_{y,E}$ defined below, so that all face quantities in \eqref{eq:nipg} are single-valued. Recall that $(n_x,n_y) = \mathbf{n}_E$ denotes the outward unit normal of $E$ on $e$; since only $n_x^2$ and $n_y^2$ enter \eqref{eq:penalty}, $\eta_e$ is the same for both adjacent cells. On the implicit cells, $\mathcal{A}_E(\mathbf{q}_h,\theta)$ is added to the left-hand side of \eqref{eq:stdg} and \eqref{eq:stdgstrong}, with the outer trace $\mathbf{w}_h^+$ given by \eqref{eq:outertrace}; we stress that the viscous terms are always discretized in the form \eqref{eq:nipg}, also on the cut cells, where the inviscid flux is used in strong form. On the explicit cells, $\mathcal{A}_E(\mathbf{p}_h,\varphi)$ is subtracted from the right-hand side of the corrector \eqref{eq:corrector}, while the local predictor \eqref{eq:predictor} is only augmented by the volume term $-\int_{\In}\int_E \theta \, \nabla \cdot ( \boldsymbol{\varepsilon}_E \nabla \mathbf{p}_h ) \dx \dt$, without any face terms, consistent with its local nature. Summing the third term of \eqref{eq:nipg} over the two cells adjacent to a face yields $\avg{\boldsymbol{\varepsilon}_e \nabla \theta} \cdot \mathbf{n}_E \, \jump{\mathbf{w}_h}$ with $\jump{\mathbf{w}_h} = \mathbf{w}_h^- - \mathbf{w}_h^+$ and a \emph{positive} sign, i.e.\ the sign is opposite to the one of the symmetric interior penalty method and of \cite{GassnerLoercherMunz2007}, as in \cite{HTCDG}, which leads to the NIPG method. For the NIPG method, the two consistency terms cancel in the energy estimate, so that the viscous bilinear form is nonnegative for any nonnegative penalty coefficient; the factor $(2N+1)/h$ in \eqref{eq:penalty} is therefore not a tuning parameter needed for stability, as it would be for the symmetric interior penalty method, but the scaling of the jump term according to   \cite{GassnerLoercherMunz2007}. Since the viscosity is kept constant during a time step, $\mathcal{A}_E$ is linear in $\mathbf{w}_h$, and its contribution to the Jacobian of the implicit system is exact without any additional effort. At the embedded wall, $\mathbf{w}_h^+$ is the mirrored state \eqref{eq:ghost} and $\nabla \mathbf{w}_h^+$ is obtained by the same mirror reflection of the gradient of the primitive variables, which corresponds to an even extension of density, pressure and tangential velocity and to an odd extension of the normal velocity across the wall. In addition, the second term of \eqref{eq:nipg} is set to zero at the wall for the mass and the total energy, so that, since the jumps of these two components vanish there, neither mass nor total energy crosses the wall also in the presence of artificial viscosity. At the outer Dirichlet boundaries, the viscous terms are omitted.

\paragraph{Choice of the artificial viscosity}
The artificial viscosity coefficients are simply chosen as
\begin{eqnarray}
  \varepsilon_{x,E} = \frac{1}{2} \, \lambda_{\max} \, \frac{h_{x,E}}{2N+1} \, \chi_E,
  \qquad
  \varepsilon_{y,E} = \frac{1}{2} \, \lambda_{\max} \, \frac{h_{y,E}}{2N+1} \, \chi_E,
  \label{eq:epsilon}
\end{eqnarray}
see also \cite{TavelliDumbser2017}. Here, $\lambda_{\max}$ is the maximum absolute eigenvalue of the flux Jacobians in $x$- and $y$-direction over all cells of the mesh, evaluated with the solution $\mathbf{u}_h^n$, $\chi_E \in [0,1]$ is a shock sensor defined below, and $h_{x,E}$ and $h_{y,E}$ denote the local cell sizes in $x$- and $y$-direction. For Cartesian cells, these are the mesh sizes $h_x$ and $h_y$ of the background mesh, while for cut cells we use the extent of the bounding box of $E$ in $x$- and $y$-direction. Scaling the viscosity \emph{separately} in each direction with the \emph{local} cell size is essential for cut cells, whose extent in $x$- and $y$-direction can differ by orders of magnitude. It also ensures that $\boldsymbol{\varepsilon}_E \to 0$ under mesh refinement, i.e.\ the added term is a vanishing viscosity regularization. Moreover, by \eqref{eq:epsilon}, $\varepsilon_{x,E}/h_{x,E} = \lambda_{\max}\chi_E/(2(2N+1)) \leq \lambda_{\max}/(2(2N+1))$ for every cell, independently of its size. On a face $e$, let $E^*$ denote the adjacent cell with the smaller $h_x$. Then $h_{x,e} = h_{x,E^*}$ and $\varepsilon_{x,e} \leq \varepsilon_{x,E^*}$, and hence
\begin{eqnarray}
  \frac{\varepsilon_{x,e}}{h_{x,e}} \leq \frac{\varepsilon_{x,E^*}}{h_{x,E^*}} \leq \frac{\lambda_{\max}}{2(2N+1)}.
  \label{eq:avbound}
\end{eqnarray}
The same bound holds in $y$-direction. Inserting both bounds into \eqref{eq:penalty} and using $n_x^2+n_y^2=1$ shows that
\begin{eqnarray}
  \eta_e \leq \frac{1}{2} \, \lambda_{\max},
  \label{eq:penaltybound}
\end{eqnarray}
i.e.\ independently of the size and shape of the cells, the penalty coefficient is bounded by one half of the maximum signal speed, which is the order of magnitude of the dissipation coefficient $\frac{1}{2} s_{\max}$ of the Rusanov flux \eqref{eq:numflux}, compare with the coefficient $\frac{1}{2} s_{\max} + (2N+1) \varepsilon / \Delta x$ of the combined inviscid and viscous jump term in \cite{HTCDG}.

The additional parabolic term introduces an additional time step restriction on the explicit cells. In addition to \eqref{eq:cfl}, we therefore require
\begin{eqnarray}
	\Delta t \leq \mathrm{CFL}_{\mathrm{visc}}
	\left( \frac{2\,\varepsilon_x}{\bigl(h_x/(2N+1)\bigr)^2} + \frac{2\,\varepsilon_y}{\bigl(h_y/(2N+1)\bigr)^2} \right)^{-1},
	\qquad \mathrm{CFL}_{\mathrm{visc}} = 0.2,
	\label{eq:cflvisc}
\end{eqnarray}
where $\varepsilon_x$ and $\varepsilon_y$ denote the viscosity coefficients \eqref{eq:epsilon} of the Cartesian cells without the sensor factor, i.e.\ with $\chi_E = 1$. Note that this restriction, too, depends only on the size of the background cells.

\paragraph{Smoothness sensor}
The shock sensor $\chi_E$ in \eqref{eq:epsilon} is based on an element-based smoothness indicator $\omega_E \in [0,1]$, which is adapted from the oscillation-eliminating (OE) parameter of Abgrall and Liu \cite{AbgrallLiu2026}.
 It is computed once per time step from $\mathbf{u}_h^n$ as
\begin{eqnarray}
  \omega_E & = & \exp \left( - \frac{1}{N_E} \sum_{e \subset \partial E} \max_{1 \leq i \leq m} \sigma_{e,E}\big( u_h^{(i)} \big) \right),
  \label{eq:sensor} \\
  \sigma_{e,E}(v) & = & \sum_{\kappa=0}^{N} \ell_{e,E}^{\kappa} \sum_{a+b=\kappa} \binom{\kappa}{a}
  \frac{ \frac{1}{|e|} \int_e \left| \partial_x^a \partial_y^b v^- - \partial_x^a \partial_y^b v^+ \right| \ds }
       { \norm{ v - \langle v \rangle_\Omega }_{L^\infty(\Omega)} + 10^{-10} },
  \label{eq:sigma}
\end{eqnarray}
where $u_h^{(i)}$ is the $i$-th component of $\mathbf{u}_h^n$, $N_E$ denotes the number of faces of $E$, $\langle v \rangle_\Omega$ the mean value of $v$ over $\Omega$, and the $L^\infty$ norm is evaluated at the nodes and quadrature points of all cells; the constant $10^{-10}$ in the denominator avoids a division by zero for components that are constant in $\Omega$. For the two momentum components, the maximum of their two normalization factors is used for both, so that a small variation of one momentum component is not amplified by a separate normalization. The local length scale is $\ell_{e,E} = h_{x,E}$ for faces whose normal vector satisfies $|n_x| \geq |n_y|$ and $\ell_{e,E} = h_{y,E}$ otherwise, with the local cell sizes of \eqref{eq:epsilon}. On the outer Dirichlet boundary only the jump of the values ($\kappa = 0$) with respect to the ghost state \eqref{eq:ghost} is taken into account, while faces on the embedded boundary and, for transition cells, the faces shared with a cut cell do not contribute to the sum. The sensor measures the jumps of the solution and of its derivatives across the cell faces: $\omega_E \approx 1$ in smooth regions, such that essentially no viscosity is added, and $\omega_E \to 0$ close to discontinuities. Compared to \cite{AbgrallLiu2026}, we additionally include the jumps of the solution values ($\kappa = 0$), which vanish for the globally continuous point values used there, we sum the derivative jumps up to the order $\kappa = N$ instead of using only the first and second derivatives, and we replace the CFL-type prefactor of \cite{AbgrallLiu2026} by one. Finally, the raw sensor value $1-\omega_E$ is spread over a layer of cells around a flagged cell, which yields the sensor value $\chi_E$ used in \eqref{eq:epsilon}: each vertex of the background mesh takes the maximum of the values $1-\omega_{E'}$ of its adjacent cells $E'$, and $\chi_E$ is the maximum of these vertex values over the vertices of $E$ (and at least $1-\omega_E$); for cut cells, only those vertices are used that coincide with vertices of the background mesh.

\subsection{Implementation aspects}\label{sec:implementation}

\subsubsection{Cut cell geometry and quadrature}

The cut cells are computed by intersecting each background cell with the polygonal representation of the embedded boundary, which results in general, possibly non-convex polygons. For the evaluation of the volume integrals, each cut cell is subdivided into subtriangles by ear clipping, and for the volume integrals on each subtriangle we use the conical product Gauss rule of Stroud \cite{Stroud1971} with $(N+2)^2$ points, built from Gauss--Jacobi points in the collapsed direction, which is exact for polynomials of degree $2N+3$. To avoid the loss of accuracy that is caused by a piecewise linear approximation of curved boundaries, the embedded boundary is represented by a cubic spline through the vertices of the polygon (a natural or clamped spline for open boundary polygons and a periodic spline for closed polygons), and all vertices of the cut cells that lie on a curved boundary are placed exactly on the spline. We assume that the boundary $\partial \Omega$ is sufficiently densely sampled by the points of the boundary polygon so that a cubic spline approximation of $\partial \Omega$ is accurate enough. The subtriangles of a cut cell that are adjacent to the curved boundary are treated as isoparametric curved triangles of degree $\max(N,1)$, following Bassi and Rebay \cite{BassiRebay1997}, and are integrated with $(N+4)^2$ quadrature points. The integrals on curved boundary faces are computed with $N+2$ Gauss points, using the normal vector and the arc length element of the isoparametric boundary map, while $N+1$ Gauss points are used on straight faces. If several consecutive segments of the curved boundary lie in the same cut cell, they are merged into one single curved face, whose isoparametric nodes lie on the spline; this avoids very small subtriangles and reduces the number of quadrature points.

In time, all cells use the $N+1$ Gauss--Legendre points of $\In$, so that the numerical flux on a face between an explicit and an implicit cell is evaluated at the same temporal quadrature points from both sides. On a face between a transition cell and a cut cell, the trace of the modal polynomial of the cut cell is evaluated directly at the quadrature points of the common face.
\medskip

\subsubsection{Basis functions}\label{sec:basis}

On the Cartesian cells, i.e.\ on the explicit and on the transition cells, we use a \textit{nodal} tensor-product Lagrange basis defined via the $N+1$ Gauss--Legendre points in each space direction and in time, so that the ADER predictor, the flux and the space-time DG polynomials are represented by their values at the space-time Gauss--Legendre nodes and all integrals are computed with the associated tensor-product Gauss--Legendre quadrature. This allows a simple dimension-by-dimension implementation, see also \cite{AMR3DCL,Zanotti2015}. On the cut cells we use a \textit{modal} basis of polynomials of degree up to $N$ in space and time, i.e.\ of $\mathbb{P}_N(E \times \In)$, that consists of the rescaled and shifted Taylor monomials
\begin{eqnarray}
  \theta_\ell = \tau^k \, \xi^i \, \eta^j, \quad i+j+k \leq N,
  \qquad
  \tau = \frac{t-t^n}{\Delta t}, \quad \xi = \frac{x - x_E^c}{h_{x,E}}, \quad \eta = \frac{y - y_E^c}{h_{y,E}},
  \label{eq:modalbasis}
\end{eqnarray}
where $\ell = (i,j,k)$ is a multi-index, $(x_E^c,y_E^c)$ is the center of the bounding box of $E$ and $h_{x,E}$ and $h_{y,E}$ are its extents in $x$- and $y$-direction, so that $\tau \in [0,1]$ and $\xi, \eta \in [-\frac{1}{2},\frac{1}{2}]$ on the bounding box. The basis is not orthonormalized. In our space-time DG scheme the element-local space-time mass and stiffness matrices are computed with the quadrature rules described above.

\subsubsection{Solving the implicit system}

The nonlinear system of Step~2, see \cref{sec:threestep}, has
\begin{eqnarray}
  n_{\rm dof} = m \left( (N+1)^3 \, |\Mtr| + \frac{(N+1)(N+2)(N+3)}{6} \, |\Mcut| \right)
  \label{eq:ndof}
\end{eqnarray}
unknowns, i.e.\ the space-time degrees of freedom of all transition cells and cut cells, which are solved together in one monolithic system. As $\Mimp$ only consists of the cut cells and of their face neighbors, the system is small compared to the total number of unknowns. We write the system as $\mathbf{R}(\hat{\mathbf{q}}) = 0$, where $\hat{\mathbf{q}}$ is the vector of all degrees of freedom of the implicit cells and $\mathbf{R}$ the residual of \eqref{eq:stdg} and \eqref{eq:stdgstrong}, tested with all basis functions, where the integrals of each cell are expressed in the reference coordinates of the cell, i.e.\ with respect to the space-time unit cube, which for cut cells corresponds to the space-time bounding box. The system is solved with an inexact Newton method \cite{DemboEisenstatSteihaug1982,BellaviaMorini2001},
\begin{eqnarray}
  \mathbf{J}(\hat{\mathbf{q}}^k) \, \delta\hat{\mathbf{q}}^k = \mathbf{R}(\hat{\mathbf{q}}^k), \qquad \hat{\mathbf{q}}^{k+1} = \hat{\mathbf{q}}^k - \delta\hat{\mathbf{q}}^k, \qquad k = 0, 1, 2, \ldots,
  \label{eq:newton}
\end{eqnarray}
which is stopped as soon as the Euclidean norm $\| \mathbf{R}(\hat{\mathbf{q}}^k) \|$ is smaller than a prescribed tolerance. The residual of each cell is evaluated in the reference coordinates of the cell, i.e.\ it is divided by the area of the Cartesian cell or, for a cut cell, by the area of its bounding box. As initial guess we simply use the spatial polynomial $\mathbf{u}_h^n$, extended constantly in time. The Jacobian $\mathbf{J} = \partial \mathbf{R} / \partial \hat{\mathbf{q}}$ is exact: it contains the Jacobians $\mathbf{A}(\mathbf{q}_h,\mathbf{n})$ of the physical flux at the current iterate, the frozen wave speeds $s_{\max}$ of \eqref{eq:smax}, the linear viscous terms \eqref{eq:nipg} and the linearization of the ghost states \eqref{eq:ghost}, including the full coupling between neighboring implicit cells. It is never assembled; instead, the Jacobian-vector products are computed in a matrix-free manner by the same loops over cells and faces that compute the residual. The linear systems in \eqref{eq:newton} are solved with the restarted GMRES method \cite{SaadSchultz1986} with a restart after 20 iterations and with left preconditioning, until the norm of the preconditioned linear residual is smaller than $\textnormal{tol}=\min( 10^{-2} \|\mathbf{R}(\hat{\mathbf{q}}^k)\|, \max(10^{-16}, 10^{-6-k}) )$, i.e.\ the preconditioned linear residual of each Newton step must be at least two orders of magnitude smaller than the current nonlinear residual, and below a tolerance that decreases exponentially with the Newton iterations.
Note that it is a typical feature of inexact Newton methods to solve the linear systems of the first Newton iterations with a larger GMRES tolerance than those of the later ones, see \cite{DemboEisenstatSteihaug1982,BellaviaMorini2001} for details. As preconditioner we use a simple block Jacobi preconditioner with one block per implicit cell, i.e.\ with blocks of size $m(N+1)^3$ on transition cells and $m(N+1)(N+2)(N+3)/6$ on cut cells, which contain the derivatives of the residual of the cell with respect to its own degrees of freedom and whose exact inverses are computed and stored. For efficiency, the preconditioner is computed at the beginning of the first time step and then recomputed only every $n_{\rm prec}$ time steps, with $n_{\rm prec} = 1000$ for smooth problems and $n_{\rm prec} = 5$ for problems with shock waves. Since $\mathbf{J}$ is exact, no approximation error is introduced in the Jacobian-vector products.

\medskip 

\subsubsection{Conditioning}\label{sec:conditioning}

Small cut cells can potentially lead to ill-conditioned systems, since the entries of their mass and stiffness matrices can be much smaller than those of their neighbors, and since the local basis can become ill-conditioned on cells of degenerate shape. In our experience, the conditioning problem is less severe than expected, which we attribute to the following properties of the scheme:
\begin{itemize}
  \item \textbf{Scaling with the bounding box.} The modal basis \eqref{eq:modalbasis} and the residual of each cut cell are defined with respect to the bounding box of the cut cell rather than with respect to the background cell. This removes the overall area factor of the bounding box from the local equations, so that the local mass matrices depend on the ratio of the area of the cut cell to the area of its bounding box instead of on $\alpha_E$, while the derivative terms contain the corresponding inverse extents of the bounding box. For small cut cells at a corner of a background cell and for thin cut cells that are parallel to a face of the background cell, this ratio stays bounded away from zero, even if $\alpha_E \ll 1$; only thin cut cells that cross a background cell diagonally lead to small ratios.
  \item \textbf{Block Jacobi preconditioning.} As each diagonal block contains the complete local space-time operator of a cell, the left preconditioning multiplies the equations of each cell with the inverse of its own diagonal block, which removes the remaining scaling of the local equations with the size of the cell. The preconditioner matrix is computed by an exact inversion of the diagonal block of each implicit space-time element. 
  \item \textbf{Elimination of vanishing cut cells.} Cut cells with $\alpha_E < 10^{-9}$ are removed from the mesh, see \eqref{eq:cutmesh}. Apart from this, the polynomial degree is \emph{not} reduced on small cut cells.
  \item \textbf{Artificial viscosity.} Since the viscosity \eqref{eq:epsilon} of a cut cell scales with the extent of its bounding box and the minimum of the two adjacent cells is used on the faces, the penalty coefficient satisfies the bound \eqref{eq:penaltybound} independently of the size of the cut cells, so that the penalty coefficient itself does not grow when the volume of a cut cell tends to zero. The conditioning of the complete discrete operator, however, also depends on the shape of the cells, on the scaling of the residual and on the preconditioning.
\end{itemize}
A detailed analysis of the conditioning is beyond the scope of this paper and will be the subject of future work.

\section{Numerical Results}\label{sec:numerics}

In this section, we present a set of numerical results obtained with ADER-ZIST, i.e.\ the mixed explicit ADER-DG and zonally implicit space-time DG scheme for Cartesian cut cells presented in this paper. The overall aim of this section is to benchmark the accuracy, robustness and overall performance of the scheme on a set of well-established test cases. 

We begin with a simple one-dimensional accuracy study.  
Then, we examine the accuracy of the method in two space dimensions for a set of smooth problems using Cartesian cut cell meshes with straight and curved embedded boundaries. To illustrate the ability of the scheme to treat more complex geometries and multiple embedded objects, we also show a smooth acoustic wave propagation problem in a domain containing five circular cylinders. Furthermore, we consider problems involving shocks, for which the artificial viscosity is switched on, and finally we compare the computational cost of ADER-ZIST with the one of a globally implicit space-time DG scheme and, for linear scalar advection, also with a classical explicit Runge-Kutta DG scheme with DoD stabilization. 

\subsection{Setup}\label{sec:num-setup}

\paragraph{Software}
The ADER-ZIST scheme has been implemented in three independent research codes: two Fortran codes \texttt{CC1D} and \texttt{CC2D} written from scratch, and an independent implementation based on the well-known finite element library \texttt{deal.II} \cite{dealII98,dealIIdesign}. The different implementations have been used to cross-check the results. 
All results presented in this paper have been computed with the Fortran codes \texttt{CC1D} and \texttt{CC2D}, apart from the Runge-Kutta DG scheme with DoD stabilization.

\paragraph{Parameters}
Unless stated otherwise, we solve the compressible Euler equations of gasdynamics with a ratio of specific heats of $\gamma=1.4$. The time step size $\Delta t$ is computed according to \eqref{eq:cfl} with $\mathrm{CFL}=0.4$ and, in the test problems with artificial viscosity, additionally according to \eqref{eq:cflvisc}; both restrictions only depend on the size of the background cells. This is indeed the key feature of the present approach, since it resolves the small cell problem appearing in Cartesian cut cell methods with the aid of a locally implicit treatment of the cut cells and their direct neighbors via a space-time DG scheme, while the explicit ADER-DG method is used everywhere else.  
The nonlinear system of Step~2 is solved with the inexact Newton--Krylov method described in \cref{sec:implementation}, where the linear systems are solved with the aid of a GMRES solver with a simple block Jacobi preconditioner. Throughout all test problems we use a Newton tolerance of $10^{-8}$, if not stated otherwise, i.e.\ the Newton iteration is stopped as soon as the Euclidean norm of the residual of the nonlinear system is smaller than $10^{-8}$. 
For each mesh, we report the minimal volume fraction $\alpha_{\min}=\min_{E\in\Mcut}\alpha_E$ appearing in the computational domain. As in \cref{sec:setting}, $\widehat\Omega$ denotes the rectangle covered by the Cartesian background mesh and $\Omega \subset \widehat\Omega$ the actual computational domain.

\subsection{Smooth problem in one dimension}\label{sec:num-1d}

\subsubsection{Convergence study}\label{sec:num-1d-conv}

We first examine the accuracy of the mixed scheme in one space dimension for polynomial degrees $N=1,2,3$ in space and time, using the linear advection equation
\begin{eqnarray}
  \partial_t u + a \, \partial_x u = 0, \qquad x \in [-1,1], \quad t \in (0,T], \qquad u(x,0) = u_0(x) = \sin(2 \pi x),
  \label{eq:advection1d}
\end{eqnarray}
with $a = 1$ and periodic boundary conditions, whose exact solution is $u(x,t) = \sin(2 \pi (x - a t))$, up to the final time $T = 2$. Following \cite{MayLaakmann2024}, the cut cell mesh is constructed as follows (see \cref{fig:mesh1d}): the background mesh consists of $L$ blocks of $K = 20$ cells each, i.e.\ of $I_{\max} = K L$ cells of size $h = 2/I_{\max}$, and in every block the 10th cell is split into a small cut cell of size $\alpha h$ and a neighbor of size $(1-\alpha) h$, with $\alpha = 10^{-4}$. The implicit region consists of the small cut cell and of its two neighbors, namely the background cell on its left and the cell of size $(1-\alpha) h$ on its right, which are treated as transition cells. Under refinement, $h$ is halved and the number of blocks $L$, and therefore the number of cut cells, is doubled; we use $L = 2, 4, 8, 16$, i.e.\ $I_{\max} = 40, 80, 160, 320$ background cells. In one space dimension, all implicit cells, including the cut cells, use the nodal tensor-product space-time basis and the formulation \eqref{eq:stdg}.

We compare three schemes: (i) the ADER-DG scheme on the equidistant mesh with $I_{\max}$ cells of size $h$, (ii) the space-time DG scheme, applied in all cells of the same equidistant mesh, and (iii) ADER-ZIST on the cut cell mesh with $I_{\max} + L$ cells. All three schemes use the same time step $\Delta t = \mathrm{CFL}/(2N+1) \, h/a$ with $\mathrm{CFL} = 0.4$, see \eqref{eq:cfl}, which does not depend on $\alpha$, and the $L^2$ projection $u_h^0 = P_h u_0$ of the initial condition onto $V_h$ as initial data. The implicit residual is evaluated in reference coordinates and contains terms scaled by $\Delta t/h_E$, where $h_E$ is the cell size. To account for round-off amplification on small cells, we stop the Newton iteration when the Euclidean residual norm is below $10^{-13}\max(1,\Delta t/h_{\min})$, where $h_{\min}$ is the smallest implicit cell size. This gives a tolerance of $10^{-13}$ for the fully implicit scheme on the equidistant mesh and approximately $5.7\cdot10^{-11}$ to $1.3\cdot10^{-10}$ for ADER-ZIST. We measure the one-step error after a single time step, $u_h^1 - u(\cdot,\Delta t)$, and the error at the final time, $u_h^n - u(\cdot,T)$ with $t^n = T$, both with respect to the exact solution, in the $L^1$ and in the $L^\infty$ norm.

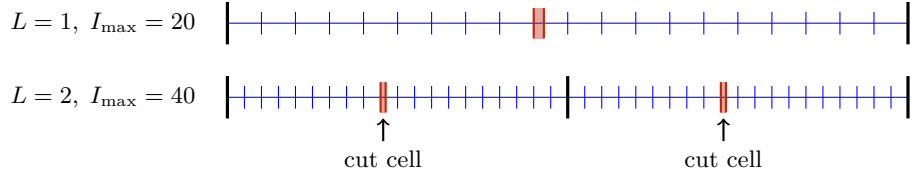
\begin{figure}[tb]
  \centering
  \input{mesh_1d.tex}
  \caption{One-dimensional cut cell meshes used in \cref{sec:num-1d}: each
    block of 20 background cells contains one small cut cell (red). Under
    refinement, the number of blocks $L$ is doubled.}
  \label{fig:mesh1d}
\end{figure}

\begin{table}[tbp]
  \centering
  \caption{One-dimensional linear advection \eqref{eq:advection1d}: $L^1$ and $L^\infty$ errors with respect to the exact solution after one time step, $u_h^1 - u(\cdot,\Delta t)$, and at the final time $T=2$, $u_h^n - u(\cdot,T)$, and observed orders of convergence; $I_{\max}$ = number of background cells.}
  \label{tab:1d-errors}
  \begin{footnotesize}
  \input{onedim_errors.tex}
  \end{footnotesize}
\end{table}

\Cref{tab:1d-errors} shows convergence consistent with the nominal order $N+1$ in both norms, after one time step and at the final time. ADER-ZIST on the cut cell mesh has errors comparable to those of the two unmixed schemes on the equidistant mesh. Its $L^\infty$ convergence orders are less regular, since the maximum error occurs in an explicit cell or in a transition cell, depending on the mesh. No additional loss of convergence order is observed in this test. In the following, we will examine the question of a potential transition error in more detail.

\subsubsection{Coupling of the explicit and the implicit scheme}\label{sec:num-1d-coupling}
For the MUSCL--Trap FV scheme studied in \cite{MayLaakmann2024}, the explicit--implicit transition reduces the local one-step order, although second-order convergence at fixed final time is retained in one dimension. 
Here, the \(L^\infty\) one-step errors in \cref{tab:1d-errors} show no indication of a coupling-induced reduction in convergence order. To investigate whether the coupling nevertheless contributes to the leading local error, we advance the ADER-DG scheme (ADER) and ADER-ZIST by one step from identical initial data and measure their solution difference directly, as well as their individual errors with respect to the semi-discrete reference solution.

We use the advection equation and periodic boundary conditions of \eqref{eq:advection1d}, with initial data specified below.
We employ an equidistant mesh with cells $E_j=[-1+jh,-1+(j+1)h]$, $j=0,\ldots,I_{\max}-1$, where $h=2/I_{\max}$ and $I_{\max}=40,80,160,320,640$. 
In ADER-ZIST, only $E_{m-3},E_{m-2},E_{m-1}$, with $m=I_{\max}/2$, are implicit. As in \cref{sec:num-1d-conv}, all three implicit cells use the nodal tensor-product space-time basis and formulation \eqref{eq:stdg}. 
The implicit region ends at $x=0$, and $E_{\mathrm{out}}=E_m=[0,h]$ is its first explicit downstream neighbor. No small cell is present. 
For $N=1,2,3$, both schemes use degree $N$ in space and time and the time step $\Delta t=\mathrm{CFL}/(2N+1)\,h/a$ with the same $\mathrm{CFL}=0.4$ as in \cref{sec:num-1d-conv}, where $h$ is the size of the equidistant cells. The Newton tolerance for ADER-ZIST is $10^{-14}$. 
We perform separate runs with initial data
$u_{0,c}(x)=\cos(2\pi x)$ and $u_{0,s}(x)=\sin(2\pi x)$.
Before comparing the schemes, we evolve the $L^2$ projections of the initial data
with the semi-discrete upwind DG operator $\mathcal{L}_h$,
allowing rapidly decaying initial transients to subside.
For each run, both schemes then start from the same state
\begin{eqnarray}
  w_{h,\sigma} = \exp(t_0 \mathcal{L}_h)\,P_h u_{0,\sigma},
  \qquad t_0=0.25, \qquad \sigma\in\{c,s\},
  \label{eq:coupling-data}
\end{eqnarray}
which, together with the semi-discrete reference
$w_{h,\sigma}^{\mathrm{next}}
=\exp(\Delta t\,\mathcal{L}_h)\,w_{h,\sigma}$,
is computed using the matrix exponential of the Fourier symbol
of $\mathcal{L}_h$.

Let $U_{S,\sigma}$ denote the solution of scheme
$S\in\{\mathrm{ADER},\mathrm{ADER\text{-}ZIST}\}$
after one step from $t_0$ to $t_0+\Delta t$, and define
$d_\sigma=U_{\mathrm{ADER\text{-}ZIST},\sigma}
-U_{\mathrm{ADER},\sigma}$.
In the first explicit cell $E_{\text{out}}$ downstream of the implicit region, we measure
\begin{eqnarray}
  D &=& \max_{x\in E_{\rm out}}
    \sqrt{d_c(x)^2+d_s(x)^2},
  \label{eq:coupling-D} \\
  e_S &=& \max_{x\in E_{\rm out}}
    \sqrt{
      \left(U_{S,c}(x)-w_{h,c}^{\mathrm{next}}(x)\right)^2
      +
      \left(U_{S,s}(x)-w_{h,s}^{\mathrm{next}}(x)\right)^2
    }.
  \label{eq:coupling-e}
\end{eqnarray}
Combining the two runs avoids masking the leading error through
cancellation for a particular phase at the interface.
The maximum is computed exactly, up to round-off, by evaluating the squared modulus at the two endpoints of $E_{\mathrm{out}}$ (one-sided values) and at the real roots of its derivative inside $E_{\mathrm{out}}$.
Thus, $D$ measures the solution difference directly, while
$e_S$ measures each scheme's error with respect to the
same semi-discrete reference.
Since both schemes start from identical data and \(a>0\), their updates in \(E_{\mathrm{out}}\) differ only in the inflow flux at \(x=0\). The solution difference in this cell therefore directly measures the local effect of the explicit--implicit coupling.

\Cref{tab:1d-coupling} shows that, under mesh refinement with this time step, i.e.\ $\mathrm{CFL}=0.4$, $D$, $e_{\mathrm{ADER}}$ and $e_{\mathrm{ADER\text{-}ZIST}}$ converge with order $N+1$. For $N=1,2,3$, respectively, $D/h^{N+1}$ approaches approximately $0.41$, $0.28$ and $0.14$, while $D/e_{\mathrm{ADER}}$ approaches $0.39$, $0.48$ and $0.22$. 
The coupling therefore changes the leading local error, but does
not reduce the one-step convergence order in this test.
Depending on $N$, it can either increase or decrease the error.
Together with the final-time results in \cref{tab:1d-errors},
these observations show that the nominal order $N+1$ is retained
in the present one-dimensional tests despite a nonzero local
coupling contribution.
\begin{remark}[Comparison with mixed explicit-implicit FV schemes]
Among the improved mixed FV schemes proposed in \cite{MayLaakmann2024} to reduce the transition error of the original MUSCL--Trap scheme of \cite{MayBerger2017}, MPRKC--Trap is closely related to ADER-ZIST. It combines the implicit trapezoidal rule with the explicit trapezoidal rule, whose first stage is computed with the MUSCL scheme. Both the explicit and the implicit scheme thus share the trapezoidal rule as a common outer shell and differ only in the state at $t^{n+1}$: the unknown new solution in the implicit scheme, a sufficiently accurate predictor in the explicit one. As a result, the switch between the schemes no longer reduced the order of the one-step error. In ADER-ZIST, the weak formulation of the PDE in space-time plays the role of this outer shell, and the two schemes differ only in the polynomial that enters the flux, $\mathbf{p}_h$ or $\mathbf{q}_h$. Unlike MPRKC--Trap, whose two-stage explicit scheme required a larger implicit zone, the element-local ADER predictor does not enlarge the implicit region. 
\end{remark}

\begin{table}[tbp]
  \centering
  \caption{One-dimensional linear advection \eqref{eq:advection1d}: contribution $D$ of the coupling and one-step errors $e_{\rm ADER}$ and $e_{\rm ADER\text{-}ZIST}$ with respect to the semi-discrete solution in the first explicit cell $E_{\rm out}$ downstream of the implicit region, see \eqref{eq:coupling-D} and \eqref{eq:coupling-e}, and observed orders of convergence; $I_{\max}$ = number of cells.}
  \label{tab:1d-coupling}
  \begin{footnotesize}
  \input{onedim_coupling.tex}
  \end{footnotesize}
\end{table}

\FloatBarrier %
\subsection{Smooth problems in two dimensions}\label{sec:num-smooth}

We now check the accuracy of ADER-ZIST in two space dimensions. We start with a single and straight embedded boundary, to investigate the coupling between the explicit ADER-DG scheme and the implicit space-time DG cells. Subsequently, we consider more complex geometries including curved and even non-convex boundaries.

\subsubsection{Advection of a smooth density profile along a ramp}\label{sec:num-ramp}

In this first two-dimensional test problem we consider the advection of a smooth density profile along a ramp that is not aligned with the Cartesian background mesh. The background mesh covers $\widehat\Omega = [-1,1]^2$, and the computational domain $\Omega$ is the part of $\widehat\Omega$ above a straight ramp that passes through the
points $\mathbf{x}_1 = (-1,-0.56403)$ and $\mathbf{x}_2 = (1,0.59597)$, i.e.\ along the line $y = 0.58\,x + 0.01597$, so that it forms an angle of roughly $30^\circ$ with respect to the $x$-axis, see  \cref{fig:ramp}. On the ramp, slip wall boundary conditions have been imposed, i.e.\ for the unknown outer state that is needed for the Riemann solver at the boundary we set the same density, pressure and tangential velocity as inside the domain, while the sign of the normal velocity is flipped. We choose a constant velocity parallel to the ramp and a constant pressure. In this case, with suitable boundary conditions, the exact solution of the compressible Euler equations reduces to that of a linear scalar advection equation for the fluid density, and any density profile that is transported parallel to the ramp is an
exact solution. In particular, we choose the following time-dependent solution, which serves both as initial condition and as Dirichlet boundary condition on all other boundaries of the domain: 
\begin{eqnarray}
	\rho(\mathbf{x},t) &=& 1 + \frac{1}{10} \sin^2 \left( 2 \pi (x - u \, t) \right) + \frac{1}{10} \cos^2 \left( 2 \pi (y - v \, t) \right), \label{eqn.ramp.rho} \\
	\mathbf{v}(\mathbf{x},t) &=& \left( \begin{array}{c} u \\ v \end{array} \right) = \frac{\mathbf{x}_2 - \mathbf{x}_1}{\| \mathbf{x}_2 - \mathbf{x}_1 \|} = \frac{1}{\sqrt{1+0.58^2}} \left( \begin{array}{c} 1 \\ 0.58 \end{array} \right), 
	\qquad 
	p(\mathbf{x},t) = 1. \label{eqn.ramp.vp}
\end{eqnarray}
Simulations are run until a final time of $T=0.02$. 

\begin{figure}[tb]
  \centering
  \begin{subfigure}[t]{0.48\textwidth}\centering
    \includegraphics[width=\textwidth]{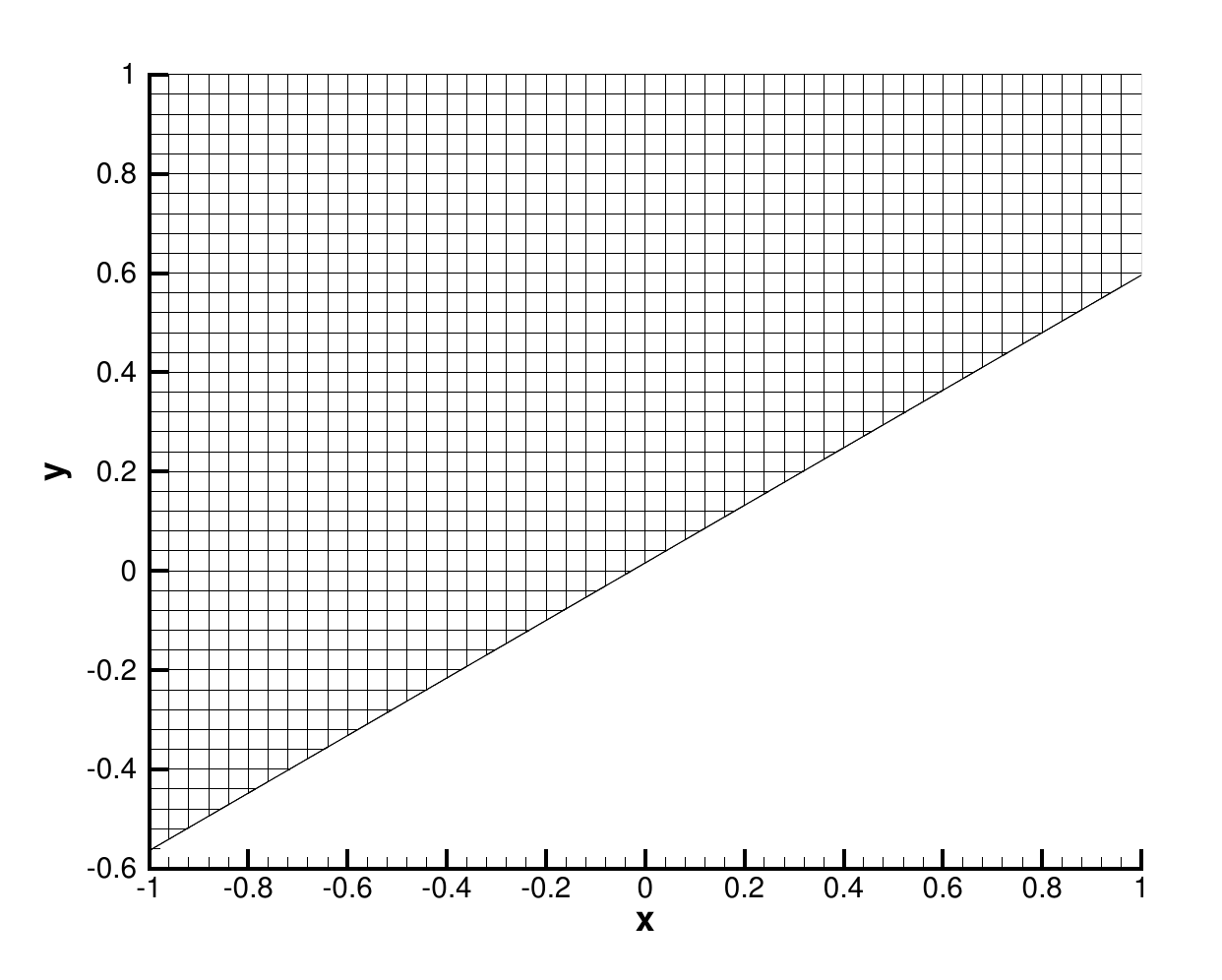}
    \par\vspace{11.01pt}%
    \caption{Mesh, $N=3$, $50\times 50$ cells.}
  \end{subfigure}\hfill
  \begin{subfigure}[t]{0.48\textwidth}\centering
    \raisebox{-11.01pt}{\includegraphics[trim=10 10 10 10,clip,width=0.96064\textwidth]{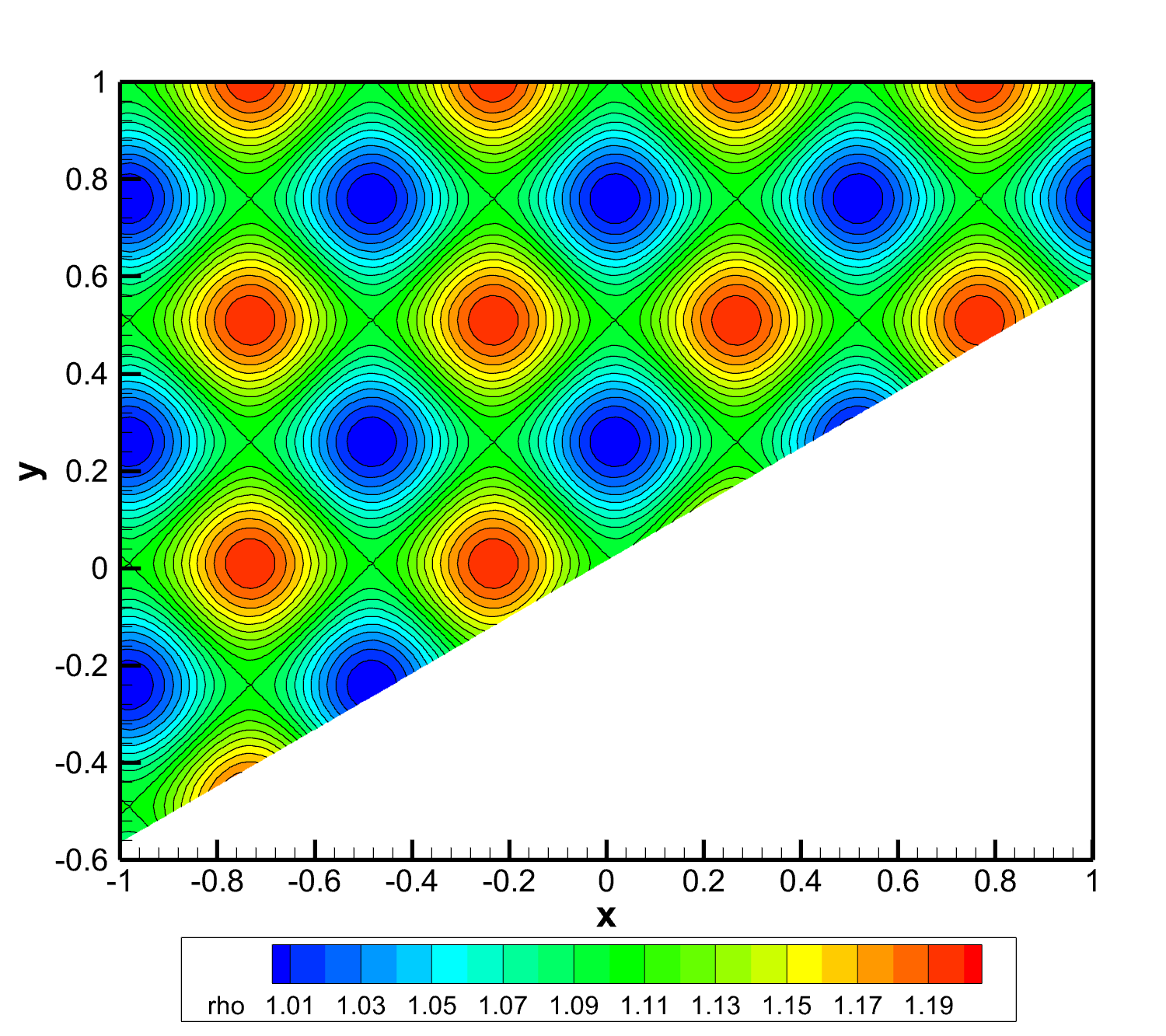}}
    \par
    \caption{Density $\rho$.}
  \end{subfigure}
  \caption{Advection along a ramp: cut cell mesh (left) and density 
  	contour colors (right).}
  \label{fig:ramp}
\end{figure}

The results of the convergence study are reported in \cref{tab:ramp}. Since, in the context of Cartesian cut cells, mesh refinement does not maintain the same mesh regularity / mesh quality over all meshes, the observed orders between two successive meshes are rather irregular, in particular in the $L^\infty$ norm. This can also be seen in the smallest volume fraction $\alpha_{\min}$, which varies between $7.76 \cdot 10^{-8}$ and $5.54 \cdot 10^{-3}$ on the meshes under consideration. We have therefore carried out the computations also on a denser sequence of meshes, namely $I_{\max} = 20, 30, \ldots, 150$ for $N=2,3$, $I_{\max} = 10, 20, \ldots, 100$ together with $I_{\max} = 25, 75$ for $N=4$, and $I_{\max} = 10, 15, \ldots, 60$ for $N=5$, and we have computed the convergence rates by a least-squares fit of the logarithm of the error against the logarithm of the mesh size $h = 2/I_{\max}$ through all meshes of each polynomial degree. The errors, the fitted lines and the fitted convergence rates are shown in \cref{fig:ramp-conv}, together with the slopes of the nominal convergence order $N+1$. The number of time steps grows linearly with $I_{\max}$ and is not affected by $\alpha_{\min}$, as expected.

\begin{table}[tb]
  \centering
  \caption{Advection along a ramp: density errors in $L^1$, $L^2$ and $L^\infty$ norm, observed orders, minimal volume fraction and number of time steps.}
  \label{tab:ramp}
  \begin{footnotesize}
  \input{ramp.tex}
  \end{footnotesize}
\end{table}

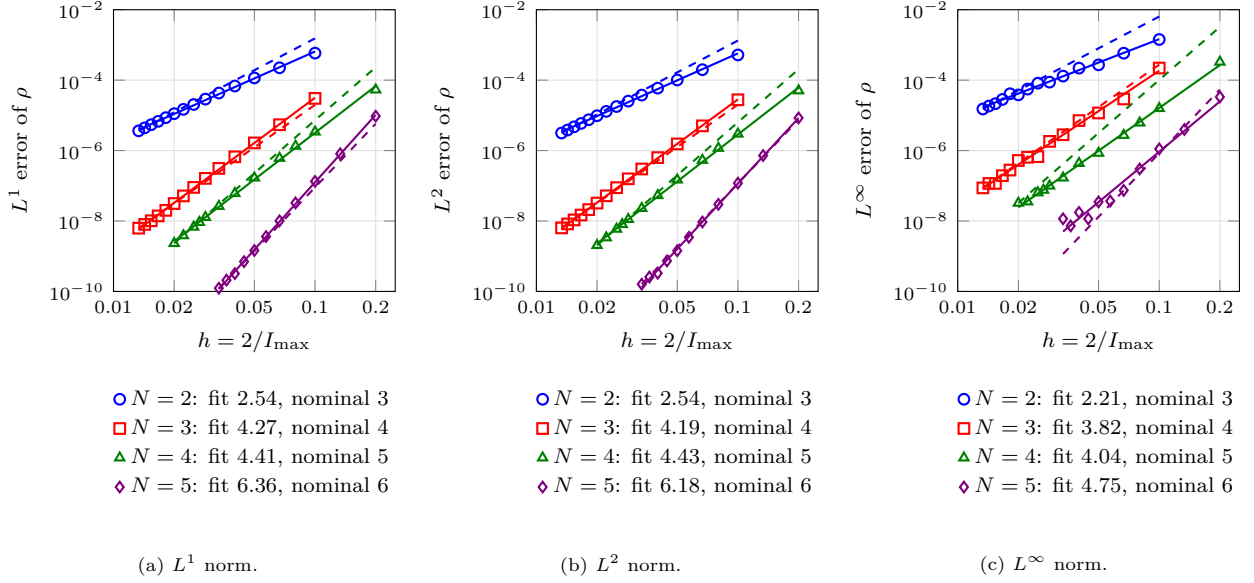
\begin{figure}[tb]
  \centering
  \begin{subfigure}[t]{0.32\textwidth}\centering
    \input{ramp_conv_L1.tex}
    \caption{$L^1$ norm.}
  \end{subfigure}\hfill
  \begin{subfigure}[t]{0.32\textwidth}\centering
    \input{ramp_conv_L2.tex}
    \caption{$L^2$ norm.}
  \end{subfigure}\hfill
  \begin{subfigure}[t]{0.32\textwidth}\centering
    \input{ramp_conv_Linf.tex}
    \caption{$L^\infty$ norm.}
  \end{subfigure}
  \caption{Advection of a density profile along a ramp: errors of the density $\rho$ at $T=0.02$ as a function of the mesh size $h = 2/I_{\max}$ (markers), least-squares fits through all meshes of each polynomial degree $N$ (solid lines, fitted rate in the legend) and reference lines of slope $N+1$ (dashed lines).}
  \label{fig:ramp-conv}
\end{figure}

\FloatBarrier %
\subsubsection{Isentropic vortex around a cylinder}\label{sec:num-vortex}

In this example we consider an isentropic vortex centered at the origin of the domain $\widehat\Omega = [-2,2]^2$, from which a circular cylinder of radius $R=0.5$ has been cut out in order to obtain the computational domain $\Omega$. 
The setup follows \cite{HuShuTri}, but with the center of the vortex now at the origin and with no uniform background velocity field
superimposed. The initial condition, which is also the exact solution of this problem, reads
\begin{eqnarray}
	\rho(\mathbf{x},t) &=& \left( 1 - \frac{(\gamma-1) \, \Sigma^2}{2 \gamma} \, e^{1-r^2} \right)^{\frac{1}{\gamma-1}}, \qquad p(\mathbf{x},t) = \rho^\gamma(\mathbf{x},t), \label{eqn.vortex.rhop} \\
	\mathbf{v}(\mathbf{x},t) &=& \left( \begin{array}{c} u \\ v \end{array} \right) = \Sigma \, e^{\frac{1}{2}(1-r^2)} \left( \begin{array}{c} -y \\ x \end{array} \right), \label{eqn.vortex.v}
\end{eqnarray}
with the square of the radial coordinate $r^2 = x^2 + y^2$ and the vortex strength $\Sigma = 5/(2\pi)$. 
Since in this test the streamlines are circular, the flow is an exact solution of the compressible Euler equations also in the presence of the circular wall boundary. We run the simulation up to a final time of $T=0.1$ on a sequence of successively refined meshes. \Cref{fig:vortex} shows the cut cell mesh with $80\times80$ background cells and the density at the final time obtained with $N=3$. The curved geometry of the cylinder is approximated by a cubic spline using 1024 points on the circle. The space-time DG scheme then treats the curved boundary via standard isoparametric elements on the triangular subgrid of the cut cells, exactly as described in \cref{sec:implementation}.

\begin{figure}[tb]
  \centering
  \begin{subfigure}[t]{0.48\textwidth}\centering
    \includegraphics[width=0.95371\textwidth]{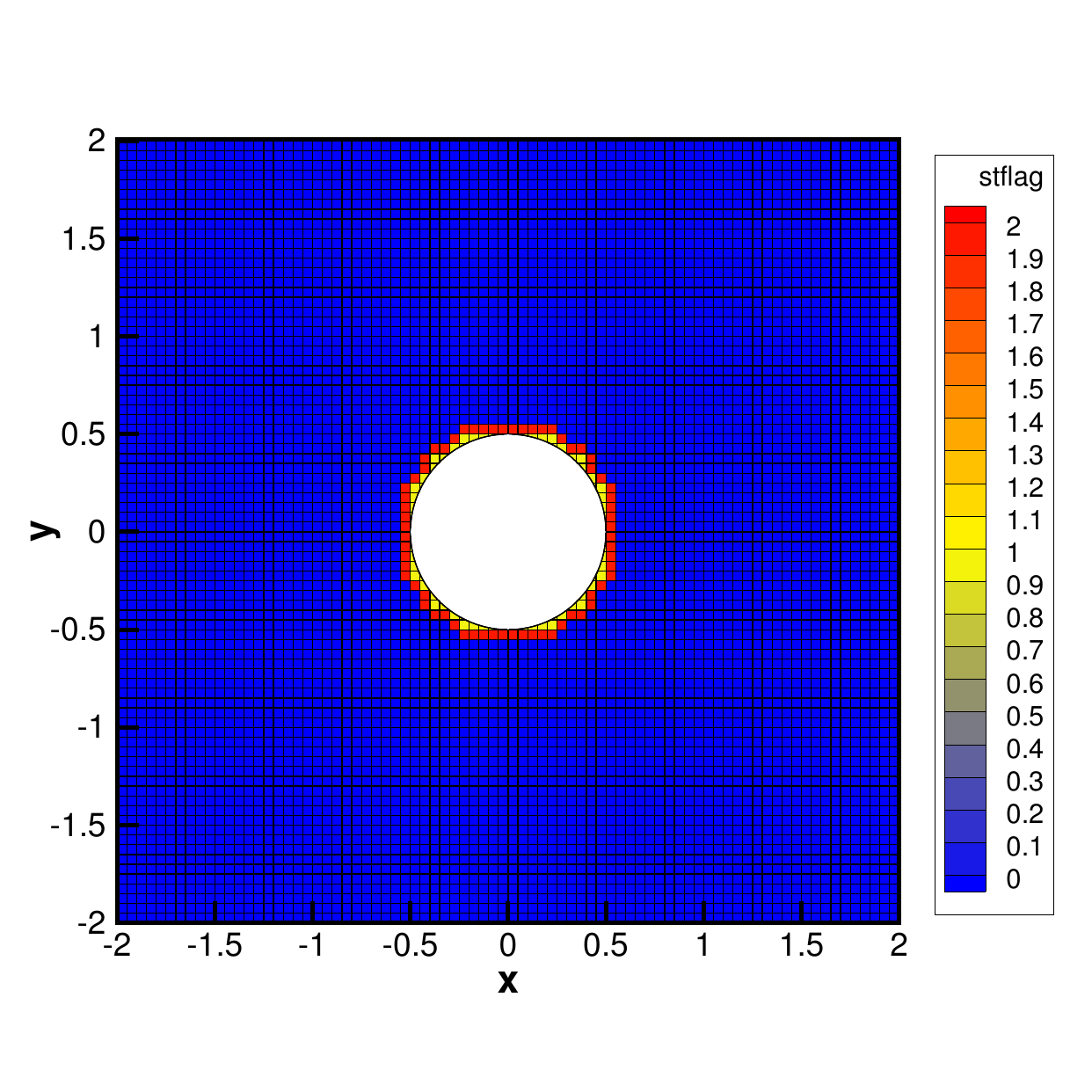}
    \par
    \caption{Cut cell mesh with $80\times80$ background cells and color codes of the cell types (flag array, see \cref{sec:classification}): 0 explicit cell, 1 cut cell, 2 transition cell.}
  \end{subfigure}\hfill
  \begin{subfigure}[t]{0.48\textwidth}\centering
    \raisebox{10.54pt}{\includegraphics[trim=10 10 10 10,clip,width=0.98667\textwidth]{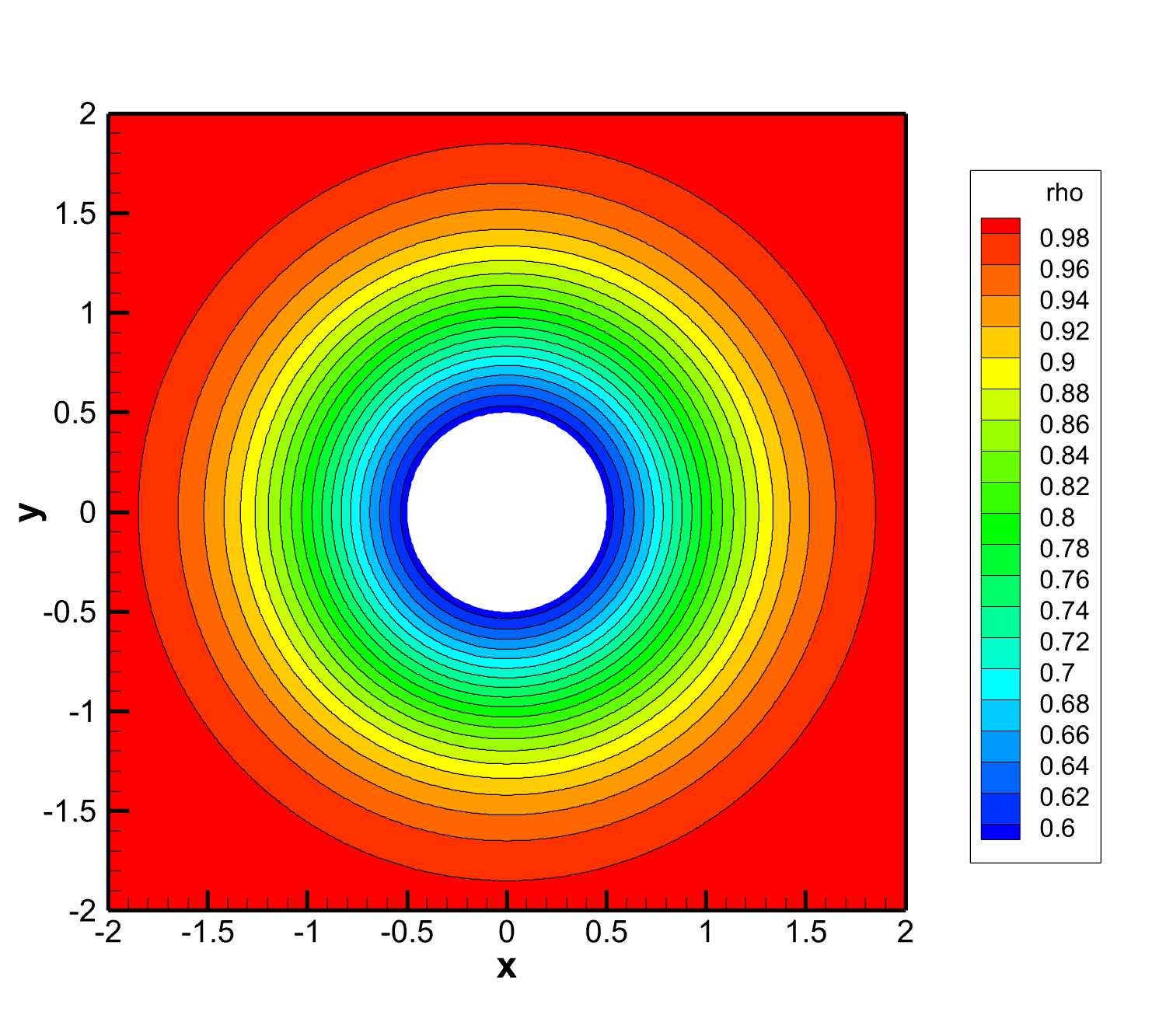}}
    \par
    \caption{Density $\rho$ at the final time obtained with $N=3$.}
  \end{subfigure}
  \caption{Isentropic vortex around a circular cylinder.}
  \label{fig:vortex}
\end{figure}

The obtained numerical results are shown in \cref{tab:vortex}. In the $L^1$ norm, the observed convergence rates are close to the nominal order $N+1$, while the $L^2$ and $L^\infty$ errors show a more irregular behavior. The minimum volume fraction lies in the range $1.07 \cdot 10^{-3} \leq \alpha_{\min} \leq 2.55 \cdot 10^{-1}$.

\begin{table}[tb]
  \centering
  \caption{Isentropic vortex around a cylinder: density errors, observed orders,
    minimal volume fraction and number of time steps.}
  \label{tab:vortex}
  \begin{footnotesize}
  \input{vortex.tex}
  \end{footnotesize}
\end{table}

\FloatBarrier %
\subsubsection{Ringleb flow}\label{sec:num-ringleb}

The Ringleb flow \cite{Ringleb1940} is one of the rare examples of a smooth, steady, transonic solution of the Euler equations for which an exact solution in closed form is available. Here we use the setup as described in 
\cite{DumbserKaeserTitarevToro2007}. To make the paper self-contained, we provide a summary of the details of the flow. 
 
The flow is constructed in the hodograph plane $(V,\theta)$, where $V = \| \mathbf{v} \|$ is the magnitude of the velocity and $\theta$ is the angle of the velocity vector with respect to the $x$-axis. In this plane, the stream function is simply given by $\Psi = \sin\theta / V$, hence the streamlines are the level sets $\Psi = 1/k$ of the stream function, and along the streamline with parameter $k$ one has $\sin\theta = V/k$ with $0 < V \leq k$. With the stagnation values $\rho_0 = 1$, $a_0 = 1$ and $p_0 = 1/\gamma$ used for the non-dimensionalization, the sound speed $a$, the density and the pressure depend only on $V$,
\begin{equation}
	a = \sqrt{1 - \frac{\gamma-1}{2} V^2}, \qquad \rho = a^{\frac{2}{\gamma-1}}, \qquad p = \frac{\rho^\gamma}{\gamma}, \label{eqn.ringleb.arhop}
\end{equation}
and the point in the physical $x$--$y$ plane that belongs to the streamline $k$ and to the speed $V$ is given by
\begin{eqnarray}
	x &=& \frac{1}{2 \rho} \left( \frac{1}{V^2} - \frac{2}{k^2} \right) + \frac{J}{2}, \qquad y = \pm \frac{1}{k \rho V} \sqrt{1 - \frac{V^2}{k^2}}, \label{eqn.ringleb.xy} \\
	J &=& \frac{1}{a} + \frac{1}{3a^3} + \frac{1}{5a^5} - \frac{1}{2} \ln \frac{1+a}{1-a}. \label{eqn.ringleb.J}
\end{eqnarray}
The velocity vector at this point is given by 
\begin{equation}
	\mathbf{v} = \left( \begin{array}{c} u \\ v \end{array} \right) = \left( \begin{array}{c} \operatorname{sgn}(y) \, V \sqrt{1 - V^2/k^2} \\ V^2/k \end{array} \right), \label{eqn.ringleb.v}
\end{equation}
i.e.\ the flow enters the channel through its lower boundary and leaves it through its upper boundary. Since \eqref{eqn.ringleb.xy} provides the physical coordinates as a function of $V$ and $k$, and not \textit{vice versa}, the exact solution at a given point $(x,y)$ has to be computed by an inversion. For this purpose we use the fact that the lines of constant speed (isotachs) are the circles
\begin{equation}
	\left( x - \frac{J}{2} \right)^2 + y^2 = \frac{1}{4 \rho^2 V^4}, \label{eqn.ringleb.isotach}
\end{equation}
so that first the scalar nonlinear equation \eqref{eqn.ringleb.isotach}, in which $\rho$ and $J$ depend only on $V$, is solved numerically for $V$, then $k = \sqrt{2} \left( V^{-2} - 2 \rho \, (x - J/2) \right)^{-1/2}$ follows from the first relation in \eqref{eqn.ringleb.xy}, and the state is finally obtained from \eqref{eqn.ringleb.arhop} and \eqref{eqn.ringleb.v}.

As in \cite{DumbserKaeserTitarevToro2007}, we set $\gamma = 1.4$, and the two solid walls of the channel are the streamlines $k_1 = 0.7$ and $k_2 = 1.2$, each of which is taken from the point where $V = 0.5$ on the lower branch ($y<0$) through the nose ($V = k$, $y = 0$) up to the point where $V = 0.5$ on the upper branch ($y>0$). Different from \cite{DumbserKaeserTitarevToro2007}, where the channel is closed by the corresponding arcs of the isotach $V = 0.5$, here the channel is closed at the bottom and at the top by the straight segments that connect the end points of the two walls with each other, see \cref{fig:ringleb}. On the two curved walls we impose slip wall boundary conditions, while on the two straight segments on the top and on the bottom the exact solution is prescribed. Each of the two curved walls is approximated by a cubic spline through 801 points that are placed in an equidistant manner in terms of the flow angle $\theta \in [\theta_{\min}, \pi - \theta_{\min}]$, with $\sin \theta_{\min} = 0.5/k$. The Cartesian background mesh covers the rectangle $\widehat\Omega = [-1.5636, 1.4364] \times [-2.3522, 2.6478]$ with $I_{\max} \times J_{\max}$ cells, where $I_{\max} = \operatorname{round}(3/h)$ and $J_{\max} = \operatorname{round}(5/h)$ for the mesh sizes $h$ reported in \cref{tab:ringleb}. The computational domain $\Omega$, i.e.\ the channel between the two walls, is non-convex, which makes this a demanding test for the cut cell geometry and for the accuracy of the curved boundary representation. The exact solution serves as initial condition, and we run the simulations up to a final time of $T=0.01$.

\begin{figure}[tb]
  \centering
  \begin{subfigure}[t]{0.48\textwidth}\centering
    \includegraphics[width=0.86761\textwidth]{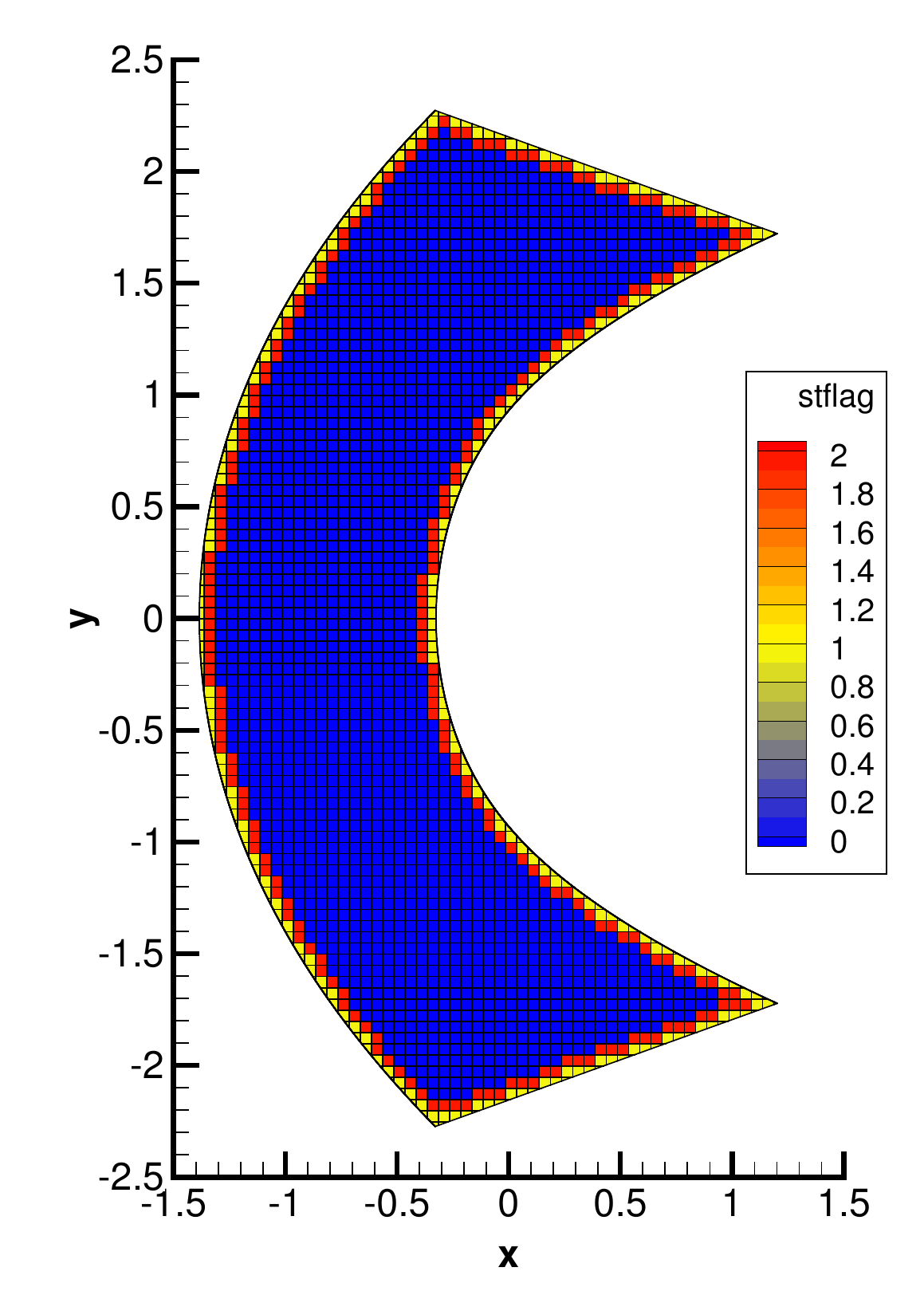}
    \par
    \caption{Cartesian cut cell mesh for the Ringleb flow and color codes of the cell types (flag array, see \cref{sec:classification}): 0 explicit cell, 1 cut cell, 2 transition cell.}
  \end{subfigure}\hfill
  \begin{subfigure}[t]{0.48\textwidth}\centering
    \raisebox{1.72pt}{\includegraphics[trim=10 10 10 10,clip,width=0.85604\textwidth]{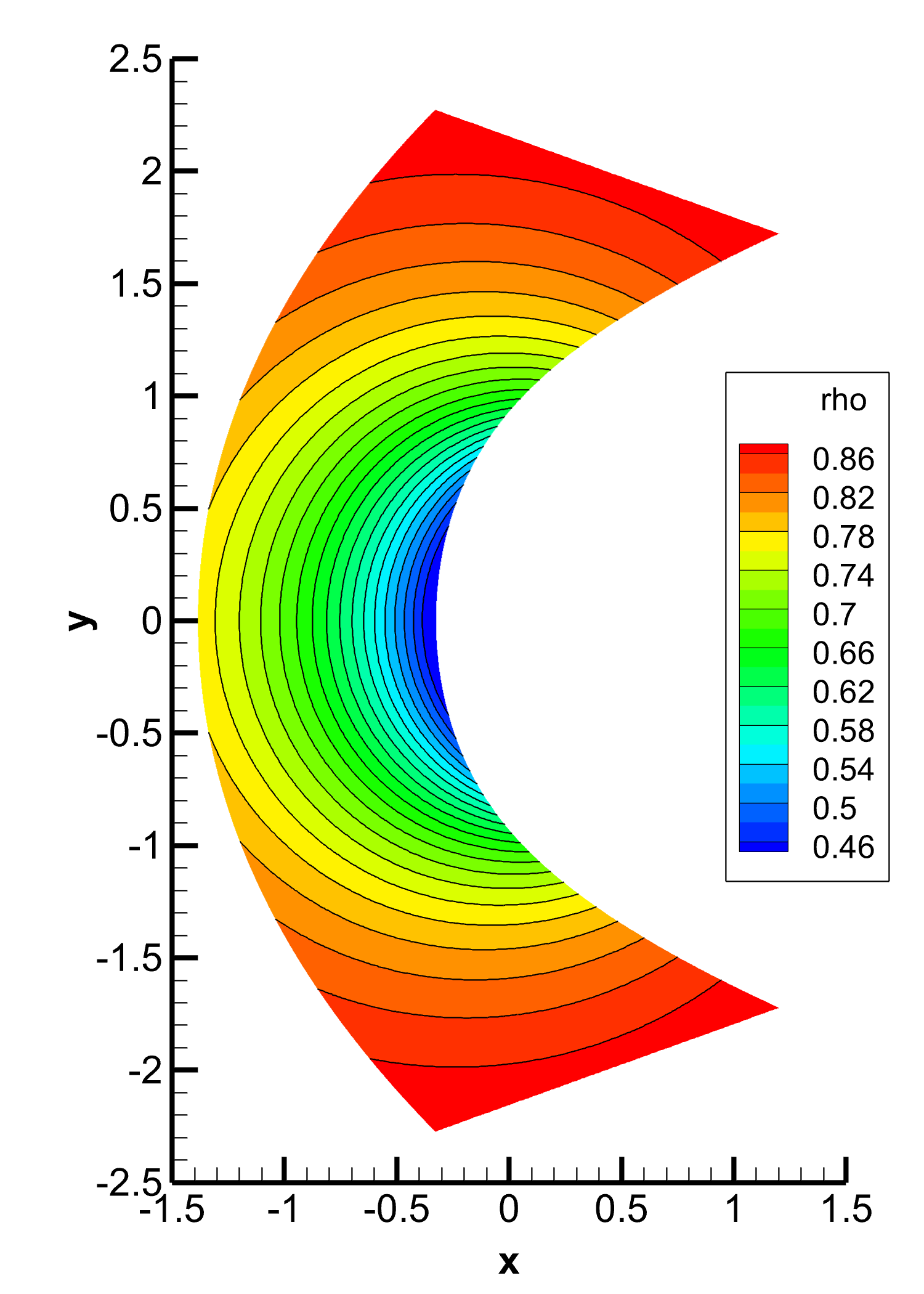}}
    \par
    \caption{Density contour colors at the final time.}
  \end{subfigure}
  \caption{Ringleb flow: non-convex computational domain and density.}
  \label{fig:ringleb}
\end{figure}

For this challenging test case, the observed convergence rates are close to the expected order $N+1$ in all norms, apart from the $L^\infty$ norm in the case $N=5$. These results confirm that the curved boundary is represented properly and accurately and that ADER-ZIST is also able to handle this challenging non-convex domain. We note that the Ringleb flow problem has already been solved on Cartesian cut cells with high order DG schemes, see \cite{Giuliani2022}, where a state redistribution method has been proposed. 

\begin{table}[tb]
  \centering
  \caption{Ringleb flow: density errors, observed orders, minimal volume
    fraction and number of time steps.}
  \label{tab:ringleb}
  \begin{footnotesize}
  \input{ringleb.tex}
  \end{footnotesize}
\end{table}

\FloatBarrier %
\subsubsection{Supersonic vortex}\label{sec:num-supersonic}

Here we consider the supersonic vortex test, see e.g. 
\cite{BergerGiuliani2021}, which consists of a supersonic flow between two concentric circular arcs with radii $r_i=1$ and $r_o=1.384$ in the first quadrant. The exact solution of this stationary problem reads 
\begin{equation}\label{eq:supersonic}
  \rho = \Bigl(1 + \frac{\gamma-1}{2}M_i^2\Bigl[1-\Bigl(\frac{r_i}{r}\Bigr)^2\Bigr]\Bigr)^{\frac{1}{\gamma-1}},
  \qquad
  u = M_i\frac{r_i}{r}\sin\theta, \qquad v = -M_i\frac{r_i}{r}\cos\theta, \qquad
  p = \frac{\rho^\gamma}{\gamma},
\end{equation}
with the Mach number at the inner wall set to $M_i=2.25$. The Cartesian background mesh covers $\widehat\Omega = [0,1.43]\times[0,1.4301]$, and the computational domain is $\Omega = \{ \mathbf{x} \in \widehat\Omega : r_i < r < r_o \}$.  On both arcs we impose reflective wall boundary conditions, while the exact solution is prescribed on the straight boundaries. We run the simulations on a sequence of successively refined meshes up to a final time of $T=0.02$.  

\begin{figure}[tb]
  \centering
  \begin{subfigure}[t]{0.48\textwidth}\centering
    \includegraphics[trim=23 141 1 207,clip,width=\textwidth]{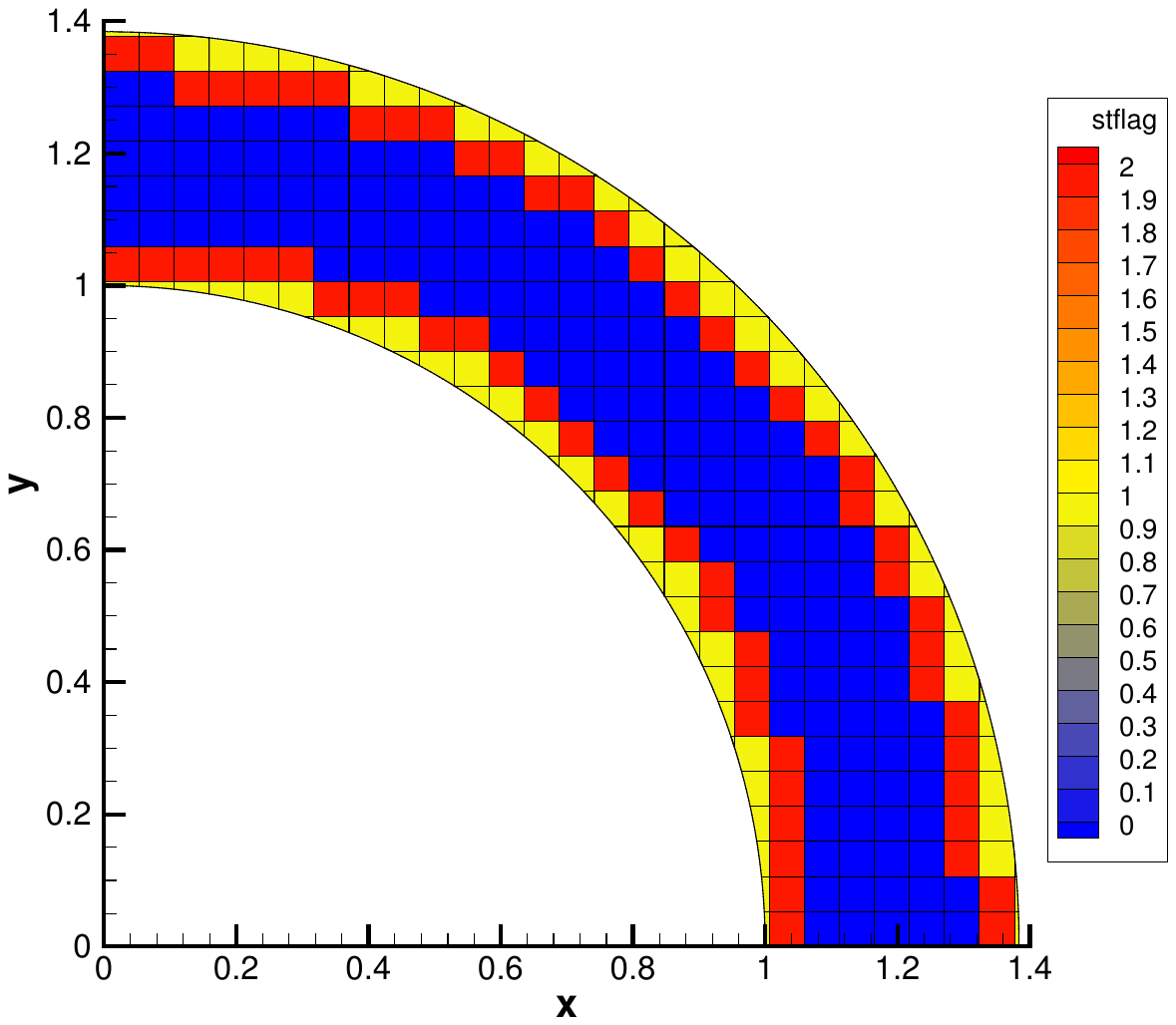}
    \caption{Cut cell mesh with $27\times27$ background cells and color codes of the cell types (flag array, see \cref{sec:classification}): 0 explicit cell, 1 cut cell, 2 transition cell.}
  \end{subfigure}\hfill
  \begin{subfigure}[t]{0.48\textwidth}\centering
    \includegraphics[trim=37 24 27 20,clip,width=0.99970\textwidth]{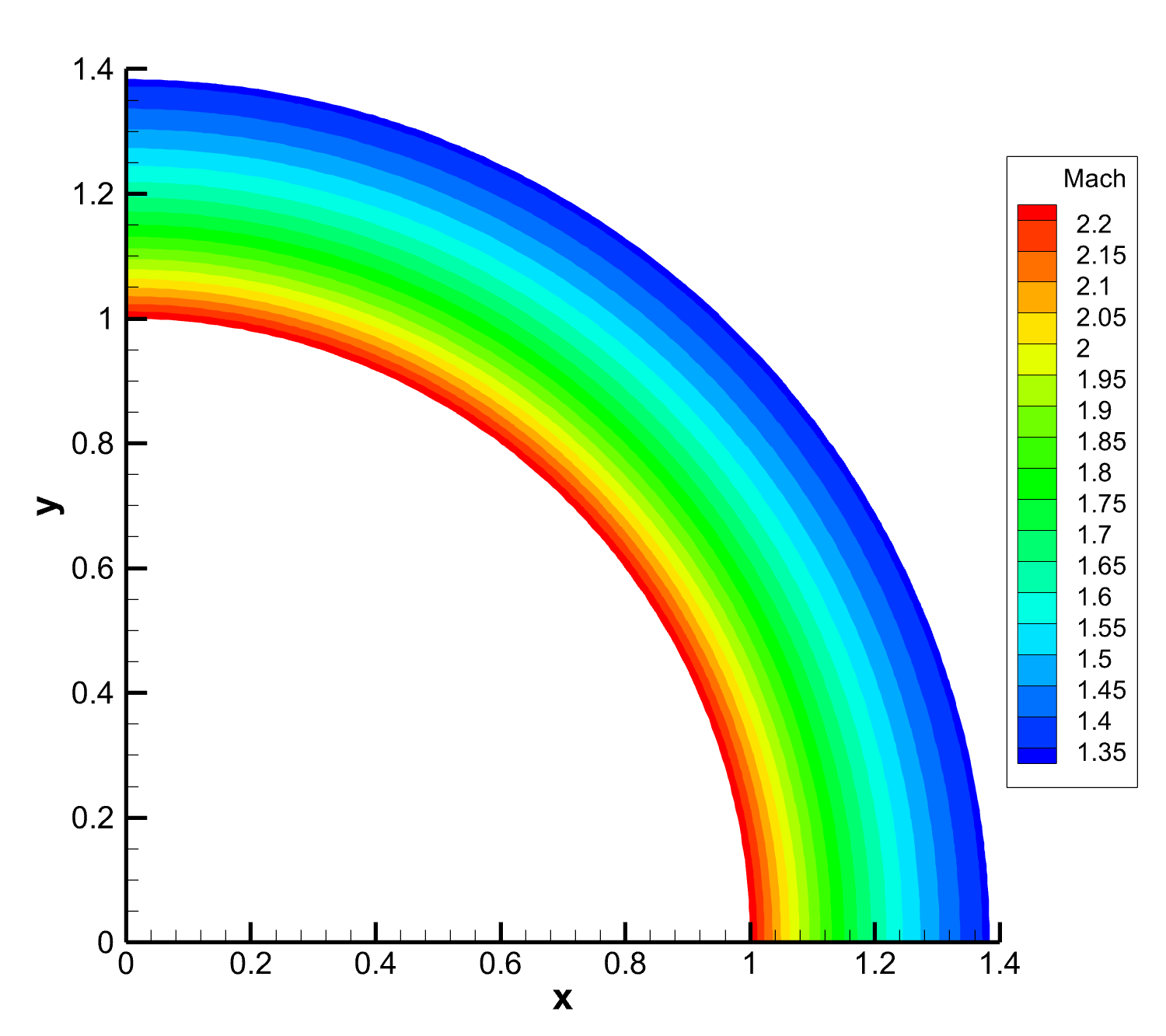}
    \caption{Mach number at the final time obtained with $N=2$.}
  \end{subfigure}
  \caption{Supersonic vortex.}
  \label{fig:supersonic}
\end{figure}

To illustrate the setup of the test case, a coarse cut cell mesh with only $27\times27$ background cells and the Mach number obtained with $N=2$ on this mesh at the final time are shown in \cref{fig:supersonic}. The detailed numerical convergence results are presented in \cref{tab:supersonic}. For $N=3$ the observed convergence rates in the $L^1$ and $L^2$ norm lie between $3.88$ and $4.51$, i.e.\ the nominal order $N+1$ is reached, while for $N=2$ they lie between $2.52$ and $2.62$, below the nominal order but slowly increasing under mesh refinement. For $N=4$ and $N=5$ the observed rates are irregular. The least-squares fits through all meshes give about $5.51$ for $N=4$ and $6.16$ for $N=5$ in the $L^2$ norm, i.e.\ on average the nominal order $N+1$ is reached also here. The $L^\infty$ errors show a more irregular behavior. The minimum volume fraction lies in the range $3.72 \cdot 10^{-6} \leq \alpha_{\min} \leq 9.65 \cdot 10^{-3}$. 

\begin{table}[tb]
  \centering
  \caption{Supersonic vortex: density errors, observed orders, minimal volume
    fraction and number of time steps.}
  \label{tab:supersonic}
  \begin{footnotesize}
  \input{quarterannulus.tex}
  \end{footnotesize}
\end{table}

\FloatBarrier %
\subsubsection{Low Mach number flow around a cylinder}\label{sec:num-potflow}

In this test we use the same geometry as in \cref{sec:num-vortex}, i.e.\ the domain $\widehat\Omega = [-2,2]^2$ from which a circular cylinder of radius $R=0.5$ centered at the origin has been cut out in order to obtain the computational domain $\Omega$. We consider the steady flow around the cylinder at a low free stream Mach number of $M_\infty = 10^{-2}$, similar to the test considered in \cite{DumbserKaeserTitarevToro2007}. As reference solution we use the classical incompressible potential flow around a circular cylinder, which in polar coordinates $(r,\theta)$ and for $\rho = 1$ reads  
\begin{equation}
	v_r = U_\infty \left( 1 - \frac{R^2}{r^2} \right) \cos\theta, \qquad v_\theta = -U_\infty \left( 1 + \frac{R^2}{r^2} \right) \sin\theta,  \qquad p = p_0 + \frac{1}{2} \left( U_\infty^2 - \| \mathbf{v} \|^2 \right), \label{eqn.potflow}
\end{equation}
with the free stream velocity $U_\infty = M_\infty = 10^{-2}$ and the background pressure $p_0 = 1/\gamma$, so that the sound speed of the free stream is unity. The velocity field in \eqref{eqn.potflow} satisfies the no-penetration condition $v_r = 0$ on the wall, and the pressure is given by the Bernoulli equation for steady incompressible flows. Since the velocity perturbations are of the order $M_\infty$ and the pressure perturbations of the order $M_\infty^2$, \eqref{eqn.potflow} is a consistent low Mach number state, but it is not an exact solution of the compressible Euler equations. However, it differs from the exact solution only by terms that vanish in the low Mach number limit. The potential flow solution serves both as initial condition and as reference solution and is prescribed on the outer boundary of the domain, while slip wall boundary conditions are imposed on the cylinder. The curved cylinder is again approximated by a cubic spline through 1024 points on the cylinder wall.

\Cref{fig:potflow} shows the velocity component $u$ and the pressure perturbation $p' = p - p_0$ obtained with $N=3$ on a mesh with $80 \times 80$ cells at a final time of $T=0.5$, together with their distribution along the cylinder wall. The flow is essentially symmetric, it stagnates at the front and at the back of the cylinder and accelerates to $2 U_\infty$ at its top and bottom. The pressure perturbation shows the corresponding pattern, with $p' = \frac{1}{2} U_\infty^2$ at the stagnation points and $p' = -\frac{3}{2} U_\infty^2$ at the top and bottom of the cylinder. The numerical solution on the wall agrees very well with the reference solution. It is well-known that for the solution of this test problem with the aid of high order DG schemes, the use of curved boundaries is mandatory, see \cite{BassiRebay1997,WangLiu2006}.

We now perform a numerical convergence study on a sequence of successively refined meshes solving the problem up to $T=0.1$ with $N=2$. The errors are measured with respect to the reference solution and are reported in \cref{tab:potflow}. For $\rho$, $u$ and $v$ we observe the nominal order $N+1=3$ on all meshes. For the pressure, the order $3$ is obtained only on the first refinement. This can be explained by the fact that the potential flow solution is an exact solution of the incompressible Euler equations and not of the compressible ones. On finer meshes the discretization error reaches the size of the difference between the incompressible reference solution and the exact solution of the compressible Euler equations.

\begin{figure}[tb]
  \centering
  \begin{subfigure}[t]{0.48\textwidth}\centering
    \makebox[\textwidth][l]{\hspace*{5.04pt}\includegraphics[trim=10 10 10 10,clip,width=\textwidth]{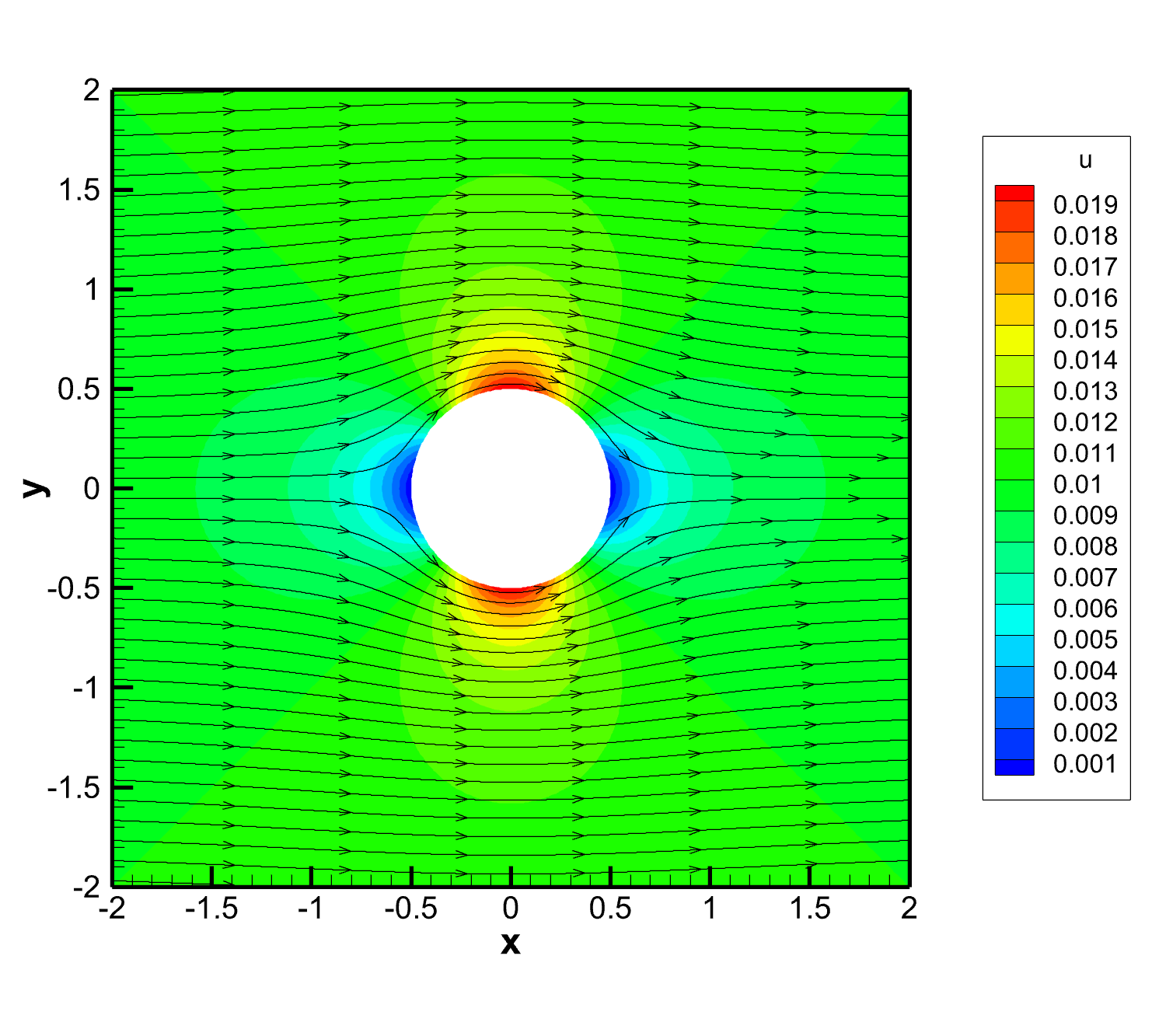}}
    \par
    \caption{Velocity component $u\vphantom{p'}$ and streamlines.}%
  \end{subfigure}\hfill
  \begin{subfigure}[t]{0.48\textwidth}\centering
    \makebox[\textwidth][l]{\hspace*{5.04pt}\includegraphics[trim=10 10 10 10,clip,width=\textwidth]{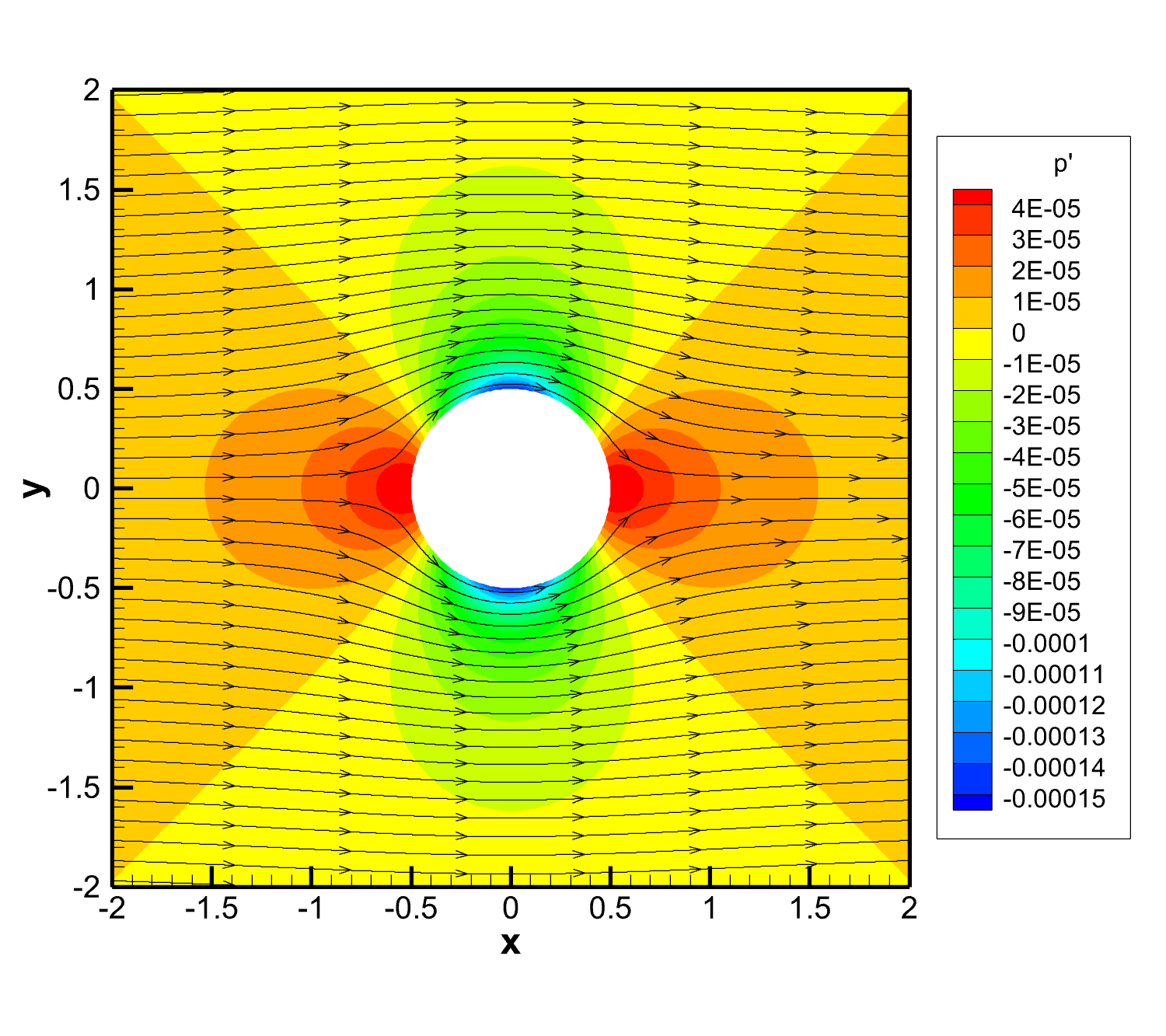}}
    \par
    \caption{Pressure perturbation $p'$ and streamlines.}
  \end{subfigure}\\[0.3cm]
  \begin{subfigure}[t]{0.48\textwidth}\centering
    \makebox[\textwidth][l]{\hspace*{-5.28pt}\includegraphics[width=0.88992\textwidth]{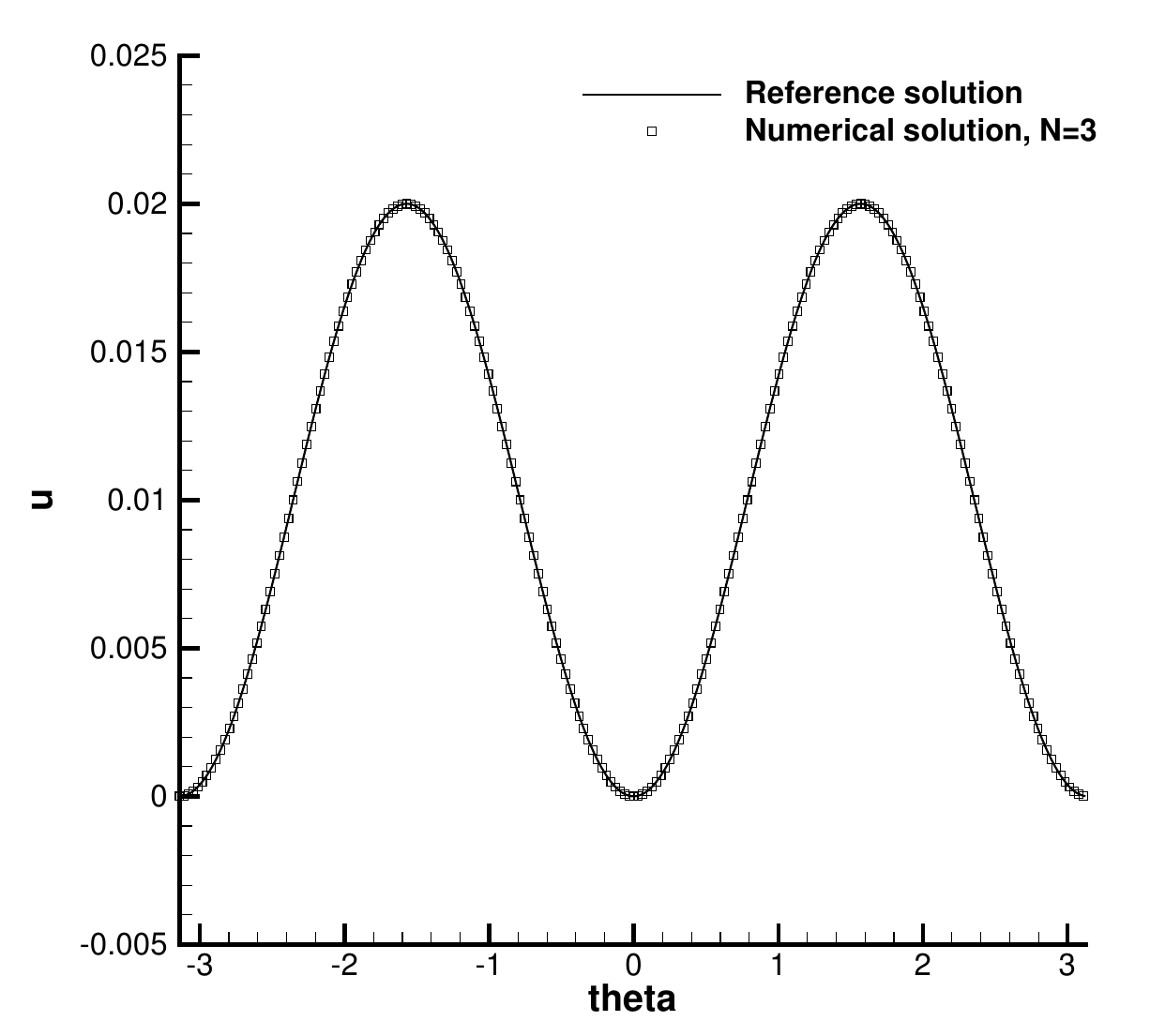}}
    \par
    \caption{Velocity component $u\vphantom{p'}$ along the cylinder wall.}%
  \end{subfigure}\hfill
  \begin{subfigure}[t]{0.48\textwidth}\centering
    \makebox[\textwidth][l]{\hspace*{-5.28pt}\includegraphics[width=0.88992\textwidth]{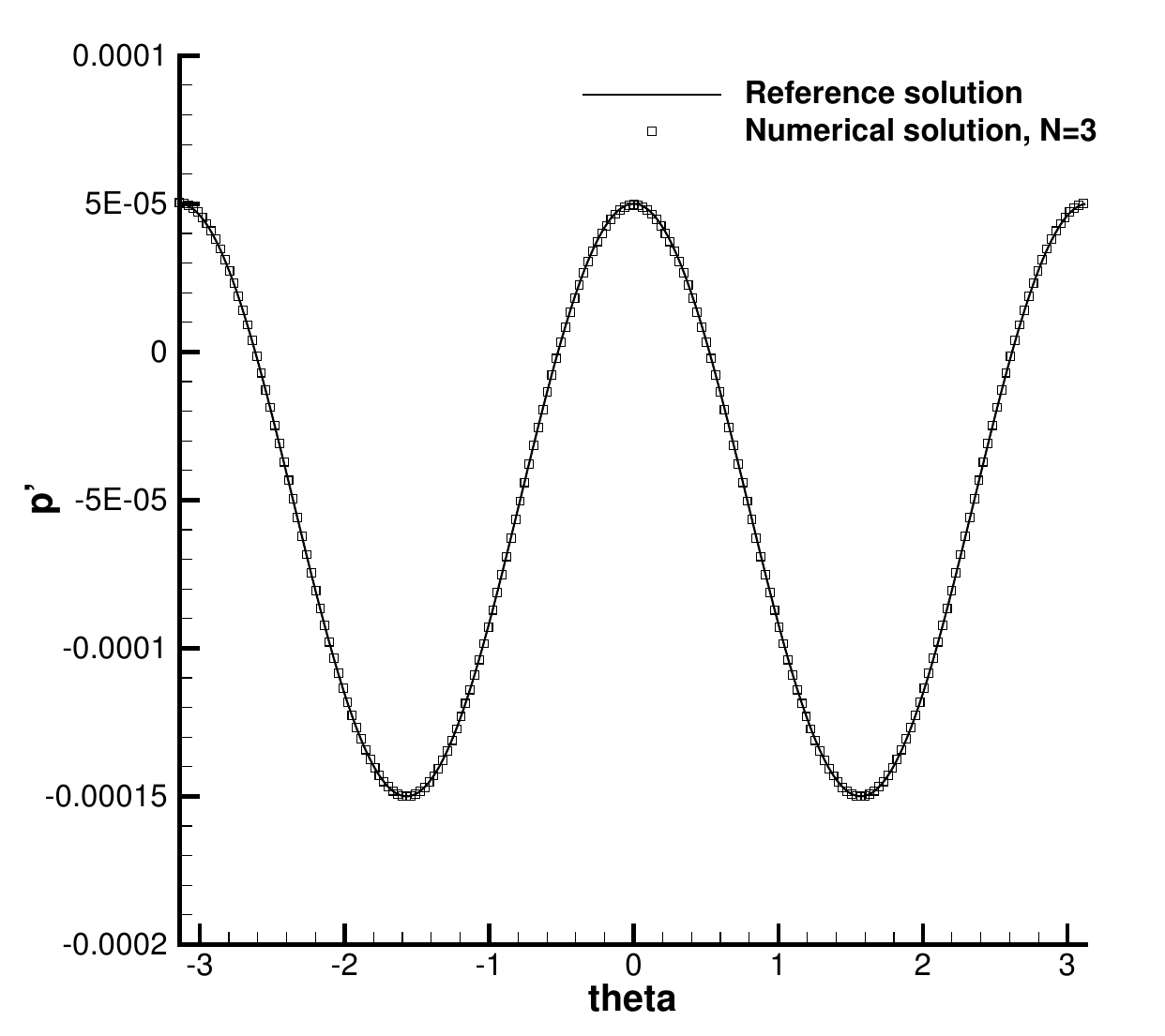}}
    \par
    \caption{Pressure perturbation $p'$ along the cylinder wall.}
  \end{subfigure}
  \caption{Low Mach number potential flow around a cylinder, $N=3$, $80\times80$ cells: numerical solution at $t=0.5$ (top) and comparison with the reference solution along the cylinder wall (bottom).}
  \label{fig:potflow}
\end{figure}

\begin{table}[tb]
  \centering
  \caption{Low Mach number flow around a cylinder, $N=2$: $L^1$ errors of $\rho$, $u$, $v$ and $p$ with respect to the incompressible reference solution, observed orders, minimal volume fraction and number of time steps.}
  \label{tab:potflow}
  \begin{footnotesize}
  \input{potentialflow.tex}
  \end{footnotesize}
\end{table}

\FloatBarrier %
\subsection{Acoustic wave propagation in a more complex geometry: pressure pulse scattered by five cylinders}\label{sec:num-waves}

The aim of the test case presented in this section is to test the ability of the code to handle non-convex domains with multiple obstacles.  	
Following \cite{Giuliani2022}, we consider the scattering of a smooth pressure pulse with small amplitude at five circular obstacles. The Cartesian background mesh covers $\widehat\Omega = [0,20]^2$, from which five circular cylinders of radius $R=2$ have been cut out in order to obtain the computational domain $\Omega$. The centers of the cylinders \begin{equation}
	\mathbf{x}_k = \left( \begin{array}{c} 10 + R_0 \cos\varphi_k \\ 10 + R_0 \sin\varphi_k \end{array} \right), \qquad \varphi_k = \frac{\pi}{2} + \frac{2 \pi k}{5}, \qquad k = 0, \ldots, 4, \label{eqn.scatter.centers}
\end{equation}
are located on a circle of radius $R_0=5$ around the center of the domain, as in \cite[Fig.~6.5(a)]{Giuliani2022}, see \cref{fig:scatter}. The initial condition is a quiescent gas with a Gaussian perturbation of amplitude $A = 10^{-4}$ in pressure and density, located in the center of the domain,
\begin{equation}
	\rho(\mathbf{x},0) = 1 - \frac{1}{\gamma} + p(\mathbf{x},0), \qquad \mathbf{v}(\mathbf{x},0) = 0, \qquad p(\mathbf{x},0) = \frac{1}{\gamma} + A \, e^{-b \left( (x-10)^2 + (y-10)^2 \right)}, \label{eqn.scatter.ic}
\end{equation}
with $b = \ln 2 / 0.2^2$ and $\gamma = 1.4$, so that the sound speed of the background state is unity. Slip wall boundary conditions are imposed on the five obstacles, the geometry of each of which is approximated by a cubic spline using 1024 points, while on the outer boundary of the domain the quiescent background state $\rho = 1$, $\mathbf{v} = 0$, $p = 1/\gamma$ is prescribed. The simulation is run until a final time of $T=6$.

\begin{figure}[tb]
  \centering
  \begin{subfigure}[t]{0.48\textwidth}\centering
    \begin{tikzpicture}
      \node[anchor=south west,inner sep=0] (img) at (0,0) {\includegraphics[width=0.86000\textwidth]{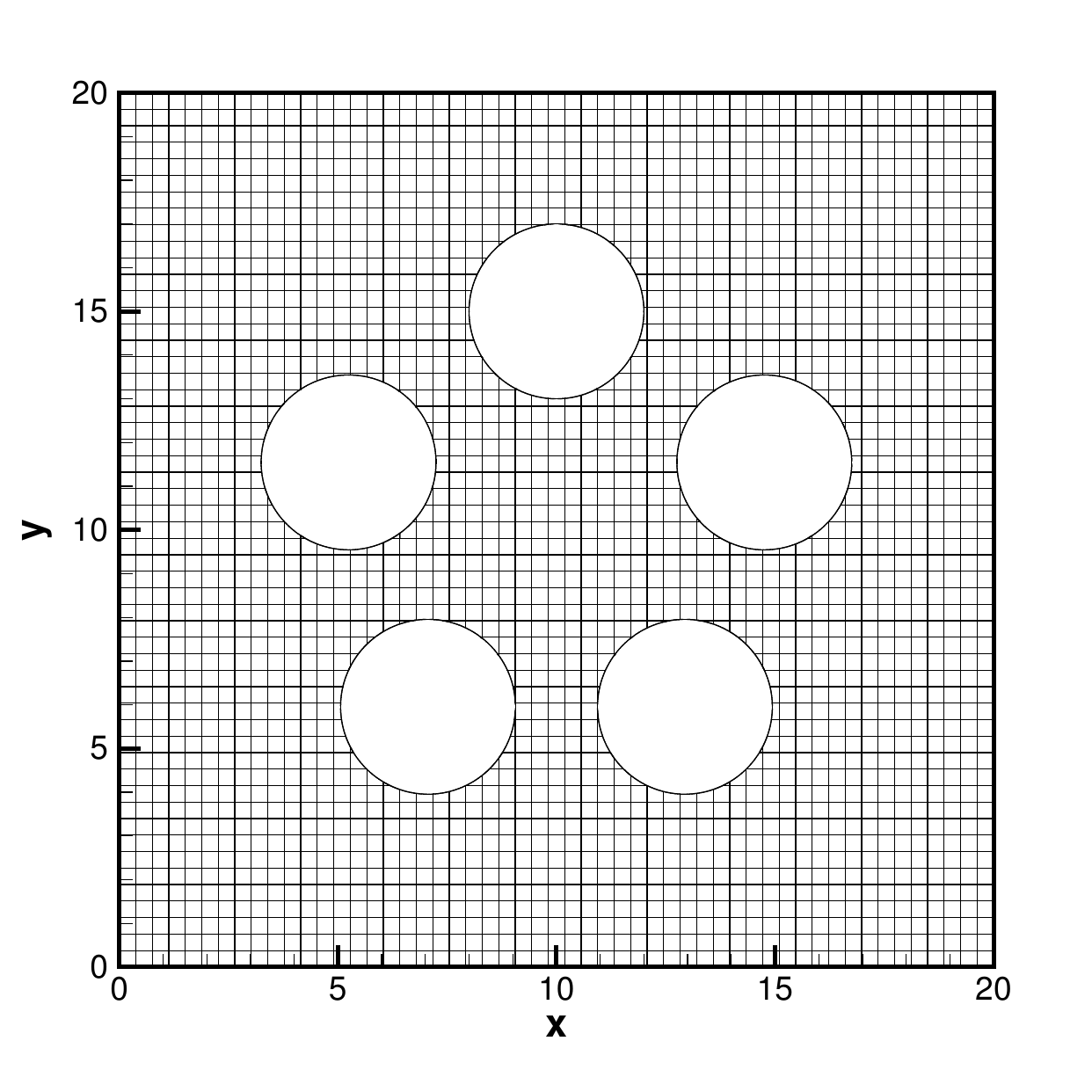}};
      \begin{scope}[x={(img.south east)},y={(img.north west)}]
        \draw[line width=0.8pt] (0.31918,0.57657) ++(-5pt,-5pt) -- ++(10pt,10pt) (0.31918,0.57657) ++(-5pt,5pt) -- ++(10pt,-10pt);
      \end{scope}
    \end{tikzpicture}
    \par
    \caption{Mesh, $53\times53$ cells.}
  \end{subfigure}\hfill
  \begin{subfigure}[t]{0.48\textwidth}\centering
    \raisebox{7.47pt}{\begin{tikzpicture}
      \node[anchor=south west,inner sep=0] (img) at (0,0) {\includegraphics[trim=10 10 10 10,clip,width=0.98037\textwidth]{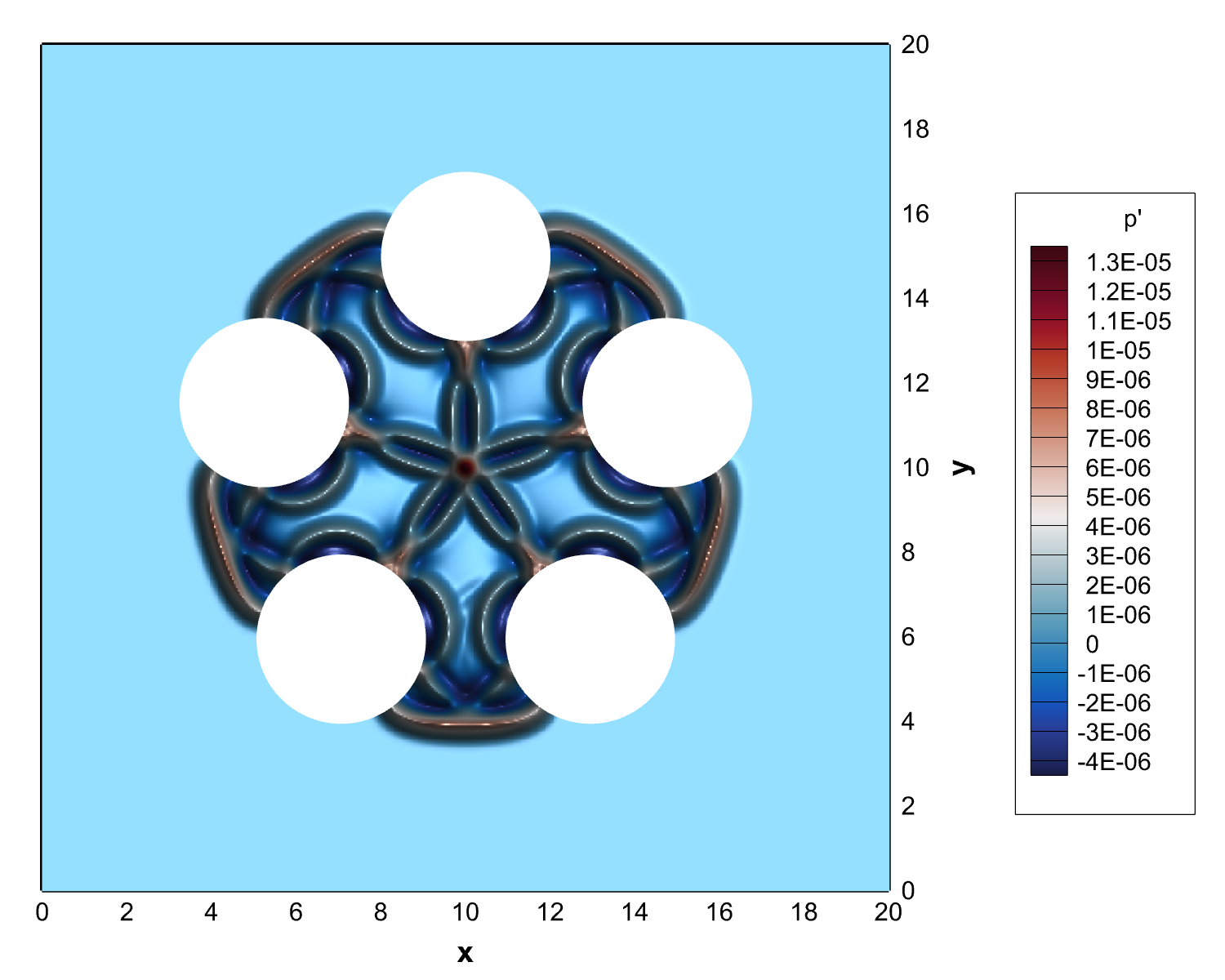}};
      \begin{scope}[x={(img.south east)},y={(img.north west)}]
        \draw[line width=0.8pt] (0.21095,0.59153) ++(-5pt,-5pt) -- ++(10pt,10pt) (0.21095,0.59153) ++(-5pt,5pt) -- ++(10pt,-10pt);
      \end{scope}
    \end{tikzpicture}}
    \par
    \caption{Pressure perturbation at $t=6$, $N=3$, $106\times106$ cells.}
  \end{subfigure}
  \caption{Pressure pulse scattered by five cylinders.}
  \label{fig:scatter}
\end{figure}

\Cref{fig:scatter} shows the pressure perturbation $p' = p - 1/\gamma$ at $T=6$, obtained with $N=3$ on a background mesh of $106 \times 106$ cells, on which the minimum volume fraction is $\alpha_{\min} = 8.01 \cdot 10^{-6}$. The solution presents the expected five-fold symmetric scattering pattern and agrees qualitatively with the results in \cite{Giuliani2022}. No spurious oscillations or asymmetries are visible close to the embedded boundaries.
For a quantitative comparison, \cref{fig:scatter-wall} shows the pressure deviation along the obstacle marked by $\times$ in \cref{fig:scatter}, which is the same obstacle as the one marked by $\times$ in \cite[Fig.~6.5(a)]{Giuliani2022}, together with the data of \cite[Fig.~6.5(b)]{Giuliani2022}, which was extracted directly from the vector graphics of that figure. The two numerical solutions agree very well with each other. The root mean square of the difference, evaluated at the points of the reference curve with the present solution interpolated linearly in $\theta$, is below $1\%$ of the maximum absolute value of the scaled pressure deviation of the reference solution, and the amplitudes of all local extrema agree to within $2\%$ of this maximum. 

\begin{figure}[tb]
  \centering
  \includegraphics[width=0.75\textwidth]{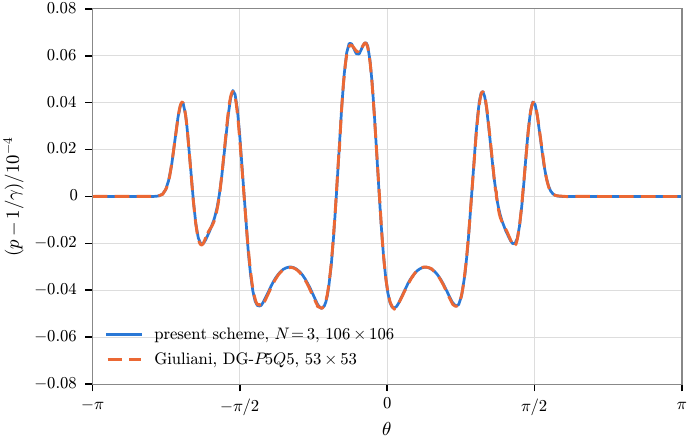}
  \caption{Smooth acoustic pressure pulse with small amplitude, scattered by five cylinders: scaled pressure deviation $(p - 1/\gamma)/10^{-4}$ at $T=6$ along the obstacle marked by $\times$ in \cref{fig:scatter} and in \cite[Fig.~6.5(a)]{Giuliani2022}, as a function of the polar angle $\theta$ about the center of that obstacle, i.e.\ of the obstacle with $k=1$ in \eqref{eqn.scatter.centers}, measured counterclockwise from the $x$-axis. ADER-ZIST with $N=3$ on $106 \times 106$ cells compared with the DG-$P5Q5$ solution of \cite{Giuliani2022} on $53 \times 53$ cells. }
  \label{fig:scatter-wall}
\end{figure}

\FloatBarrier %
\subsection{Test problems with shocks}\label{sec:num-shocks}

For the following tests, we switch on the artificial viscosity (AV) described in \cref{sec:av}, which is based on the shock sensor of Abgrall and Liu \cite{AbgrallLiu2026}. 

\subsubsection{Sod shock tube with walls}\label{sec:num-sod}

We first consider the classical Sod shock tube problem in $x$-direction, but in a two-dimensional strip. The Cartesian background mesh covers the square $\widehat\Omega = [-1,1]^2$ with $100 \times 10$ cells, i.e.\ $h_x = 0.02$ and $h_y = 0.2$. Two horizontal walls are placed at $y = \pm y_0$ with $y_0 = 0.82$, so that the computational domain is the strip $\Omega = [-1,1] \times [-y_0,y_0]$ and the cells along each wall are cut cells with volume fraction $\alpha=0.1$. As a result, the mesh contains
all three types of cells -- explicit Cartesian cells, which are updated by the classical ADER-DG scheme with a nodal tensor-product basis, transition cells, which use the implicit space-time DG method but with the same nodal tensor-product basis, and finally cut
cells directly adjacent to the walls, which use the space-time DG method on a modal basis of approximation degree $N$. For a sketch of the mesh and the flag array, see \cref{fig:sod}. The exact solution is computed with the exact Riemann solver of Toro \cite{toro-book} and is
one-dimensional. Hence, ideally all cells should produce (almost) the same solution values. However, this is not trivial since three different numerical schemes are employed and the basis in the space-time DG cells changes between transition cells and cut cells. 
The initial data are $(\rho,u,v,p) = (1,0,0,1)$ for $x<0$ and $(\rho,u,v,p) = (0.125,0,0,0.1)$ for $x>0$, and these two states are prescribed on the left and on the right boundary, respectively. We use $N=2$ and compute up to $T=0.4$.

\begin{figure}[!htbp]
  \centering
  \begin{subfigure}[t]{0.48\textwidth}\centering
    \includegraphics[trim=15 20 12 20,clip,width=\textwidth]{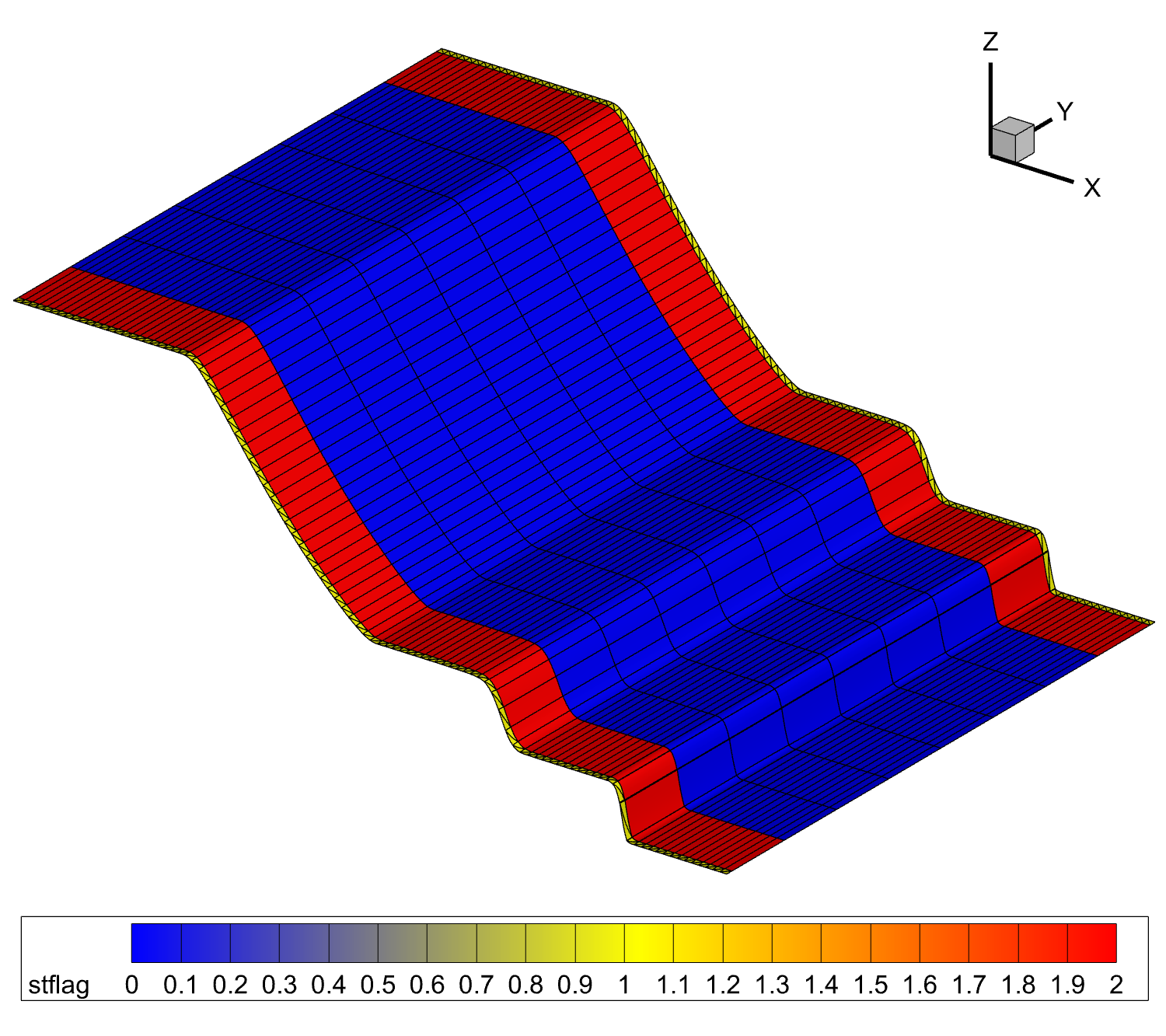}
    \caption{Color codes for the different cell types (flag array, see \cref{sec:classification}: 0 explicit cell, 1 cut cell, 2 transition cell.}
  \end{subfigure}\hfill
  \begin{subfigure}[t]{0.48\textwidth}\centering
    \includegraphics[trim=15 20 12 20,clip,width=\textwidth]{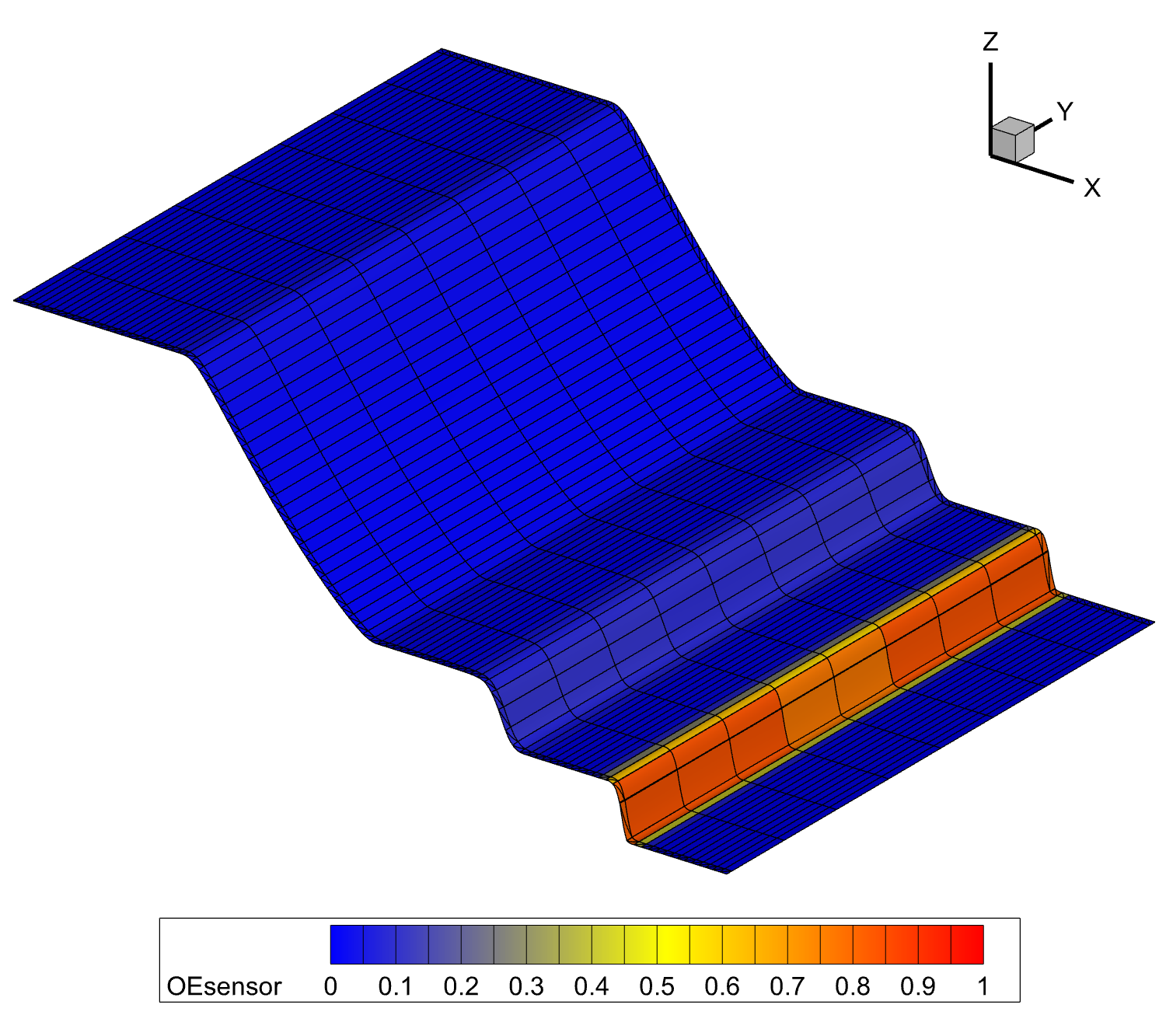}
    \caption{OE sensor value $\chi_E$.}
  \end{subfigure}\\[0.3cm]
  \begin{subfigure}[t]{0.32\textwidth}\centering
    \includegraphics[width=\textwidth]{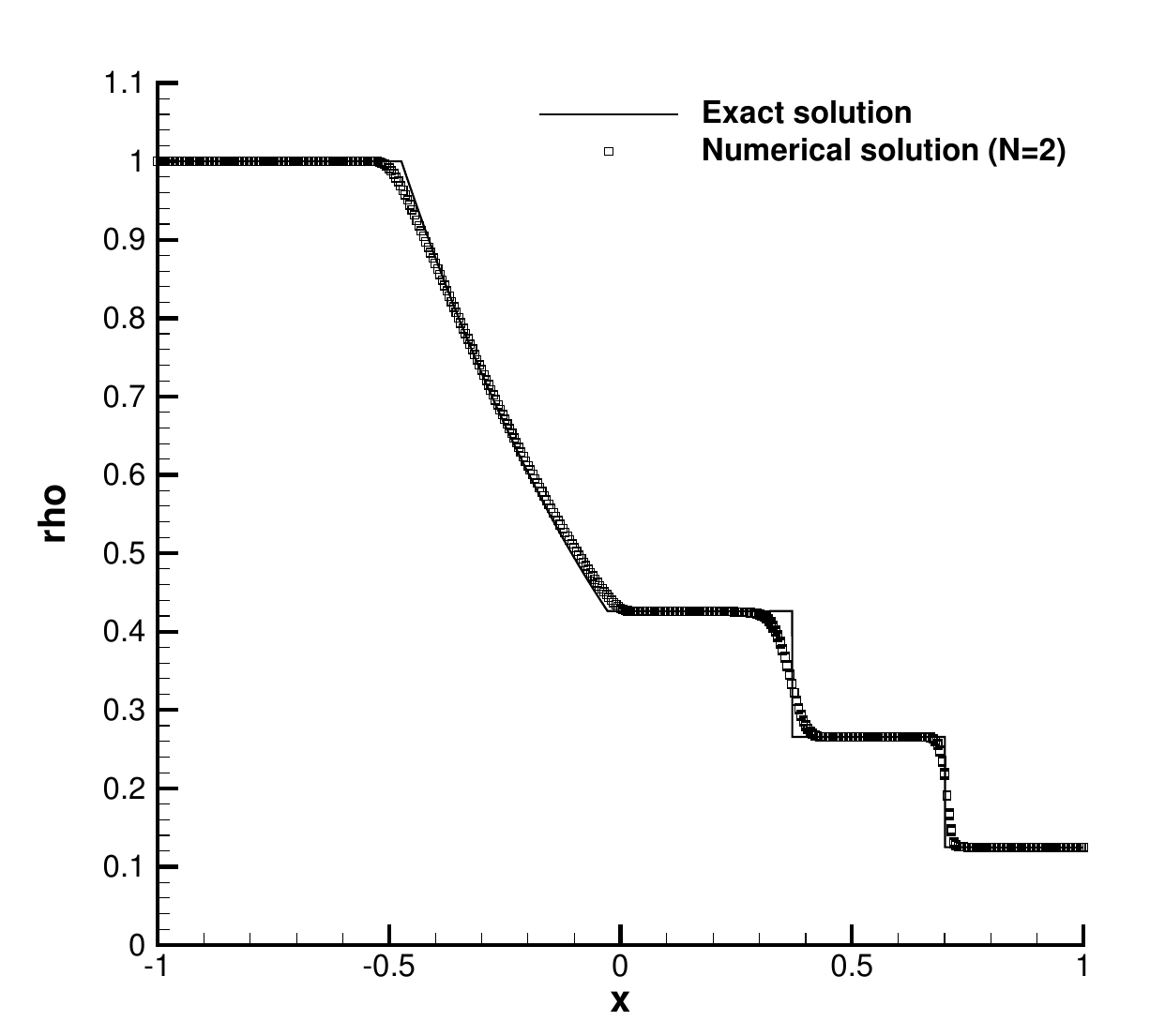}
    \caption{Density $\rho$.}
  \end{subfigure}\hfill
  \begin{subfigure}[t]{0.32\textwidth}\centering
    \includegraphics[width=\textwidth]{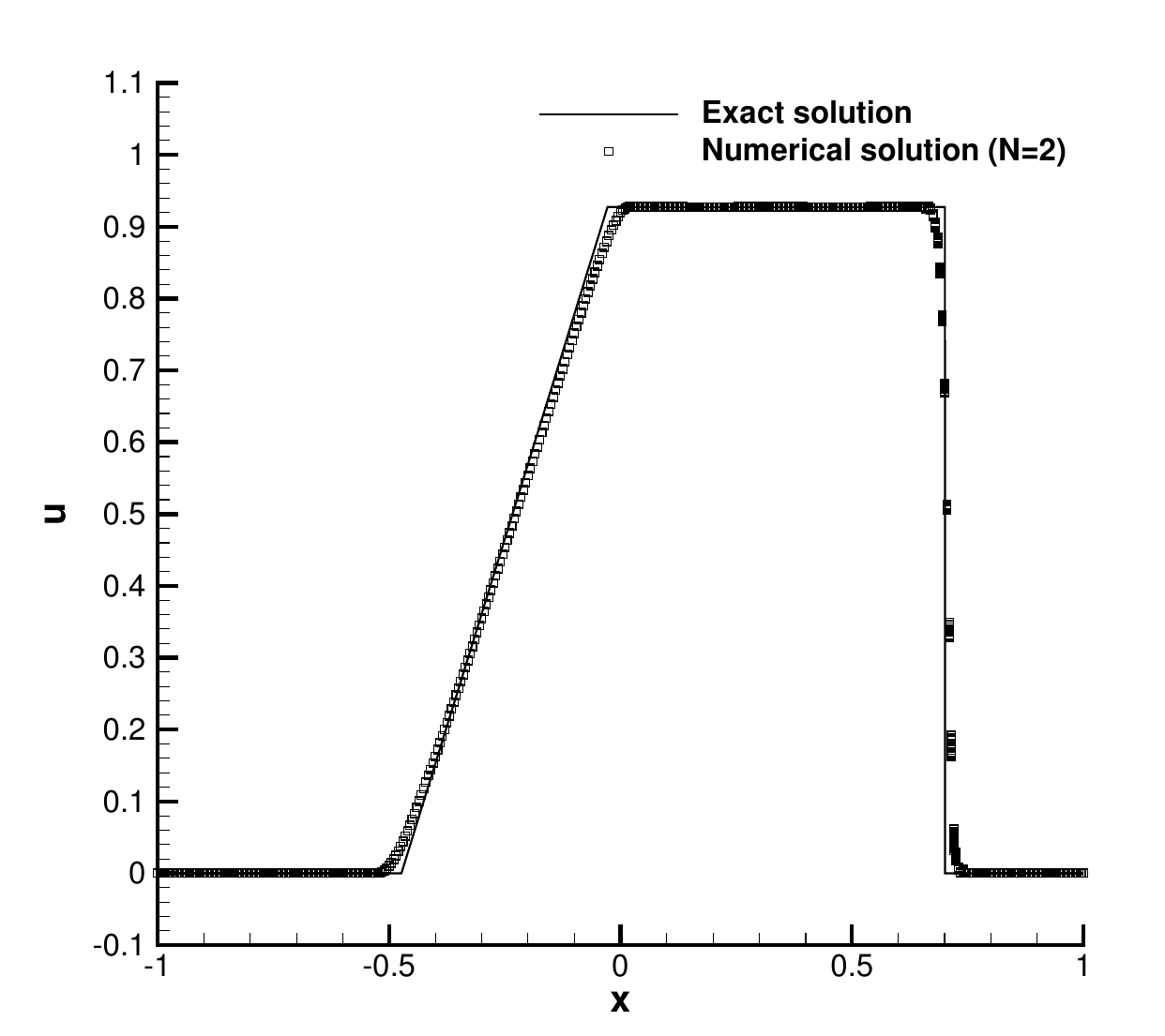}
    \caption{Velocity $u$.}
  \end{subfigure}\hfill
  \begin{subfigure}[t]{0.32\textwidth}\centering
    \includegraphics[width=\textwidth]{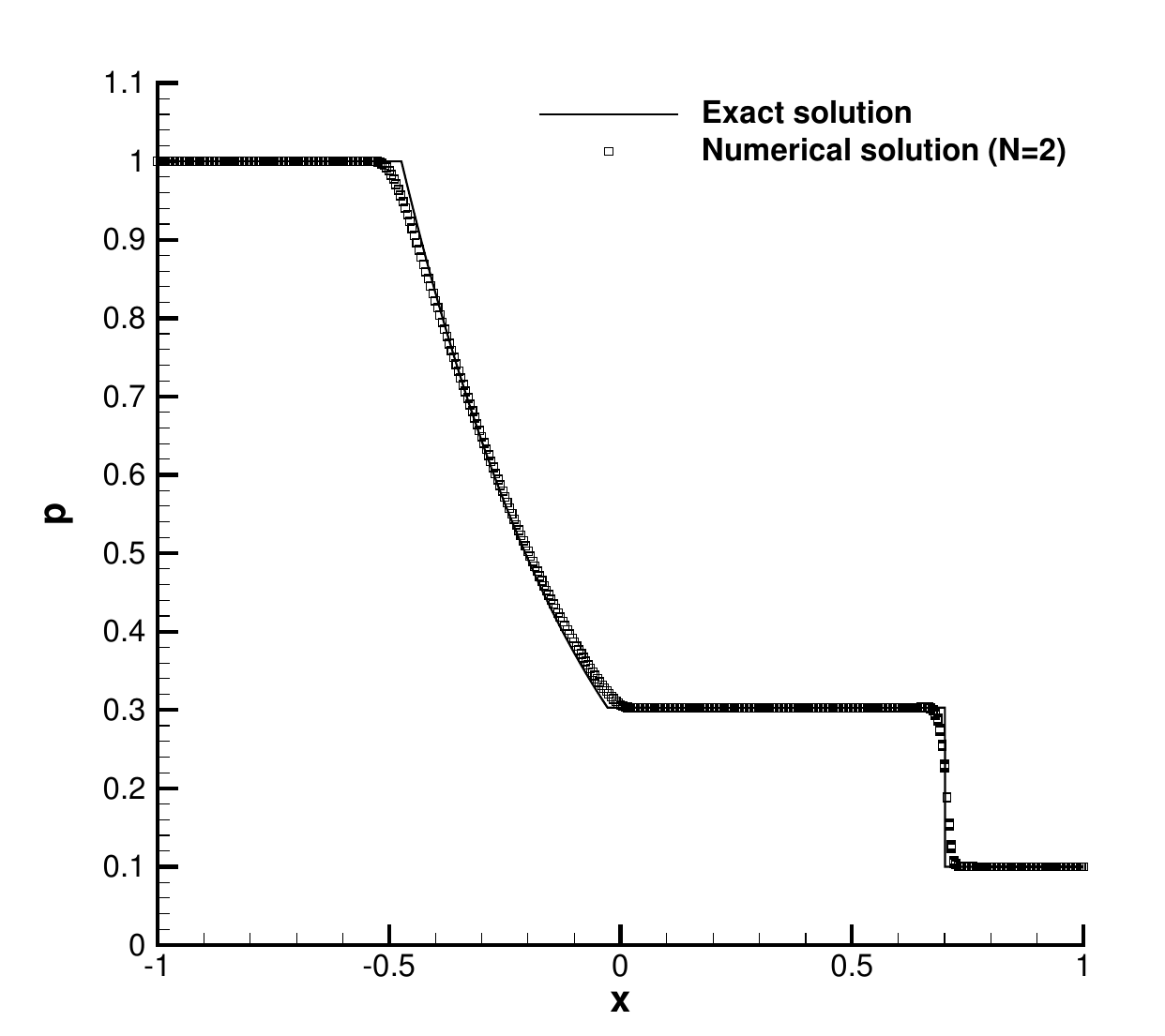}
    \caption{Pressure $p$.}
  \end{subfigure}
  \caption{Sod shock tube problem computed in 2D with cut cells at the walls: cell types and artificial viscosity sensor
    (top), and scatter plots of the solution values in all cells against the exact solution at $T=0.4$ (bottom).}
  \label{fig:sod}
\end{figure}

The scatter plots in \cref{fig:sod} show the solution values of all cells. From the figure one can observe that the values in the cut cells and transition cells agree with those in the Cartesian
cells, and in all cases the solution agrees well with the exact solution. Furthermore, the shock sensor of Abgrall and Liu \cite{AbgrallLiu2026} activates the artificial viscosity only in the vicinity of the shock and the contact discontinuity.

\FloatBarrier %
\subsubsection{Shock diffraction around a cylinder}\label{sec:num-cylshock}

As a last test case we consider the shock diffraction problem presented in \cite{MayBerger2013} and \cite{BergerGiuliani2021}. A right-moving Mach 2 shock, initially located at $x=0.2$, with pre-shock state $\mathbf{V}_R = (\rho_R,u_R,v_R,p_R)=(1.4,0,0,1)$ and post-shock state
$\mathbf{V}_L = (3.7333,1.25,0,4.5)$, hits a cylinder of radius $R = 0.15$ centered at $\mathbf{x}_c = (0.5,0.5)$. The Cartesian background domain is $\widehat\Omega = [0,1] \times [-0.05,1.05]$. We employ a polynomial approximation degree of $N=2$ on a Cartesian background mesh composed of $150\times165$ cells, i.e.\ $h_x = h_y = 1/150$, and run the simulation up to times $T_1=0.25$ and $T_2=0.3$, respectively. Slip wall boundary conditions are imposed on the cylinder and on the two horizontal walls, while the post-shock state $\mathbf{V}_L$ is prescribed on the left and the pre-shock state $\mathbf{V}_R$ on the right boundary. \Cref{fig:cylshock} depicts 
the density contours and the density along the cylinder surface compared with the reference solutions of Berger and Giuliani \cite{BergerGiuliani2021} (method of lines,
$T_1=0.25$) and May and Berger \cite{MayBerger2013} ($T_2=0.3$), which have been digitized from the figures of these two papers.

\begin{figure}[!htbp]
  \centering
  \begin{subfigure}[t]{0.48\textwidth}\centering
    \includegraphics[trim=40 20 30 20,clip,width=\textwidth]{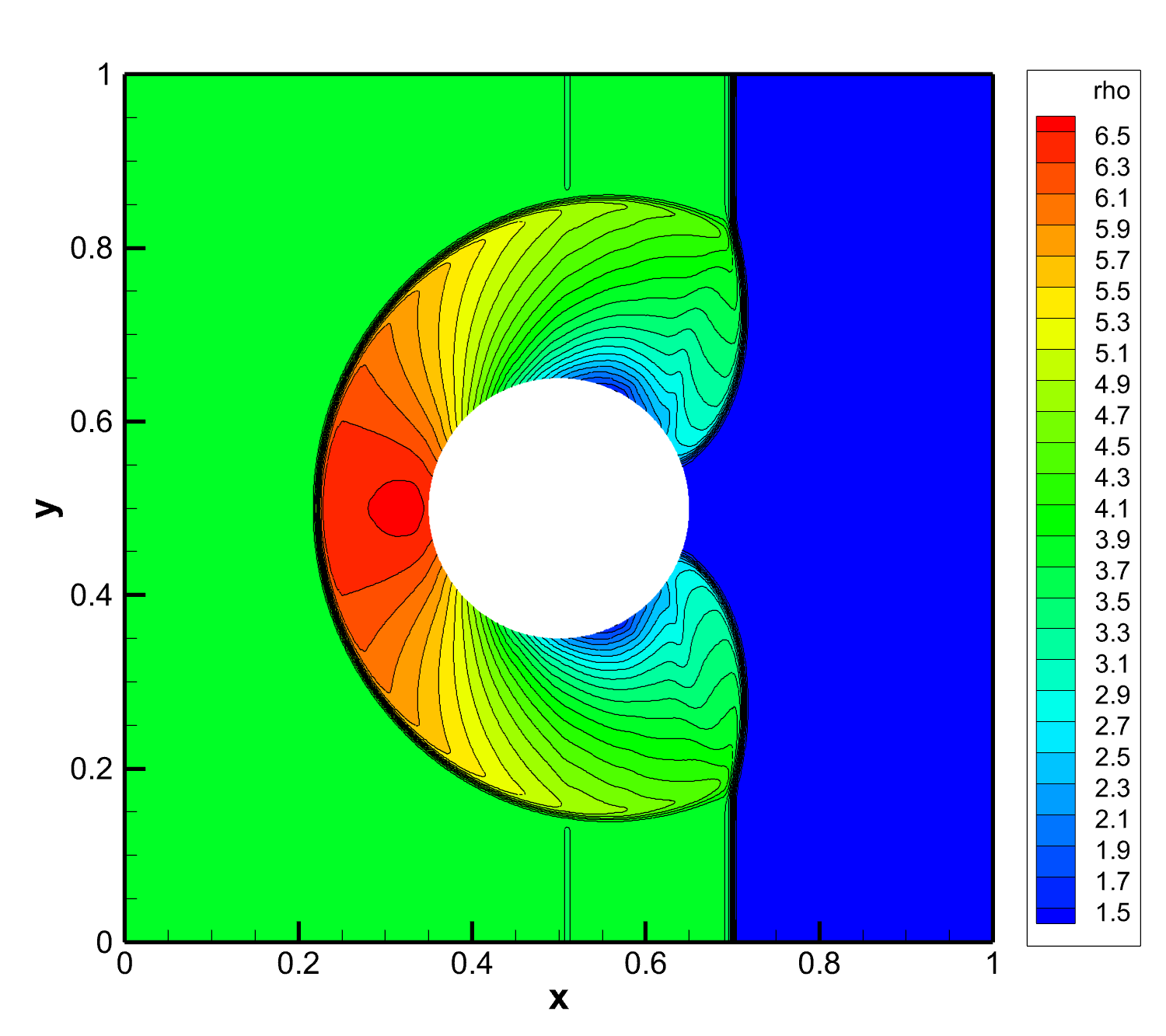}
    \par
    \caption{Density contour colors at time $T_1=0.25$.}
  \end{subfigure}\hfill
  \begin{subfigure}[t]{0.48\textwidth}\centering
    \raisebox{2.52pt}{\includegraphics[trim=50 20 50 20,clip,width=0.95413\textwidth]{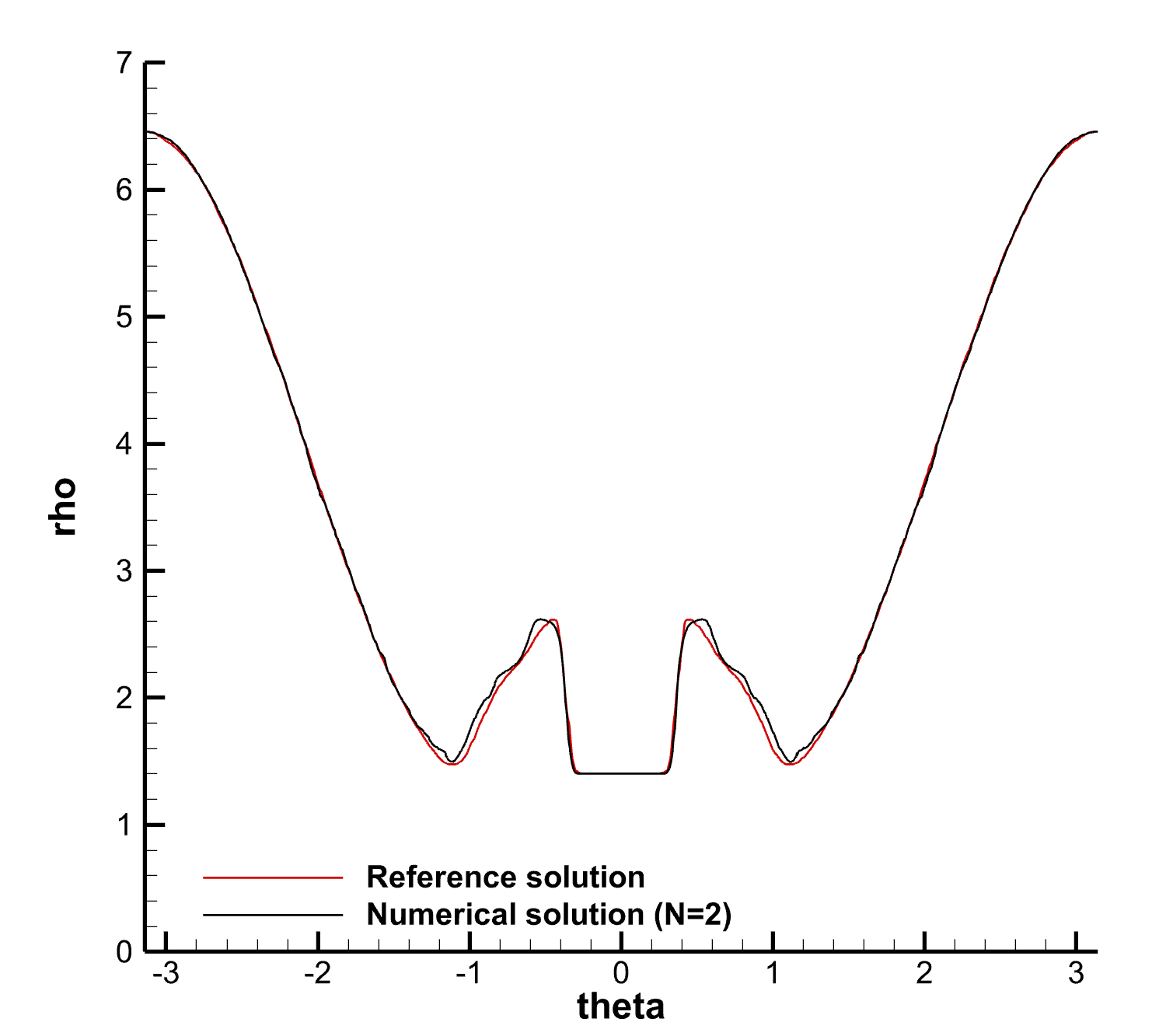}}
    \par
    \caption{Density along the cylinder surface and comparison with the reference solution at time $T_1=0.25$.}
  \end{subfigure}\\[0.3cm]
  \begin{subfigure}[t]{0.48\textwidth}\centering
    \includegraphics[trim=40 20 30 20,clip,width=\textwidth]{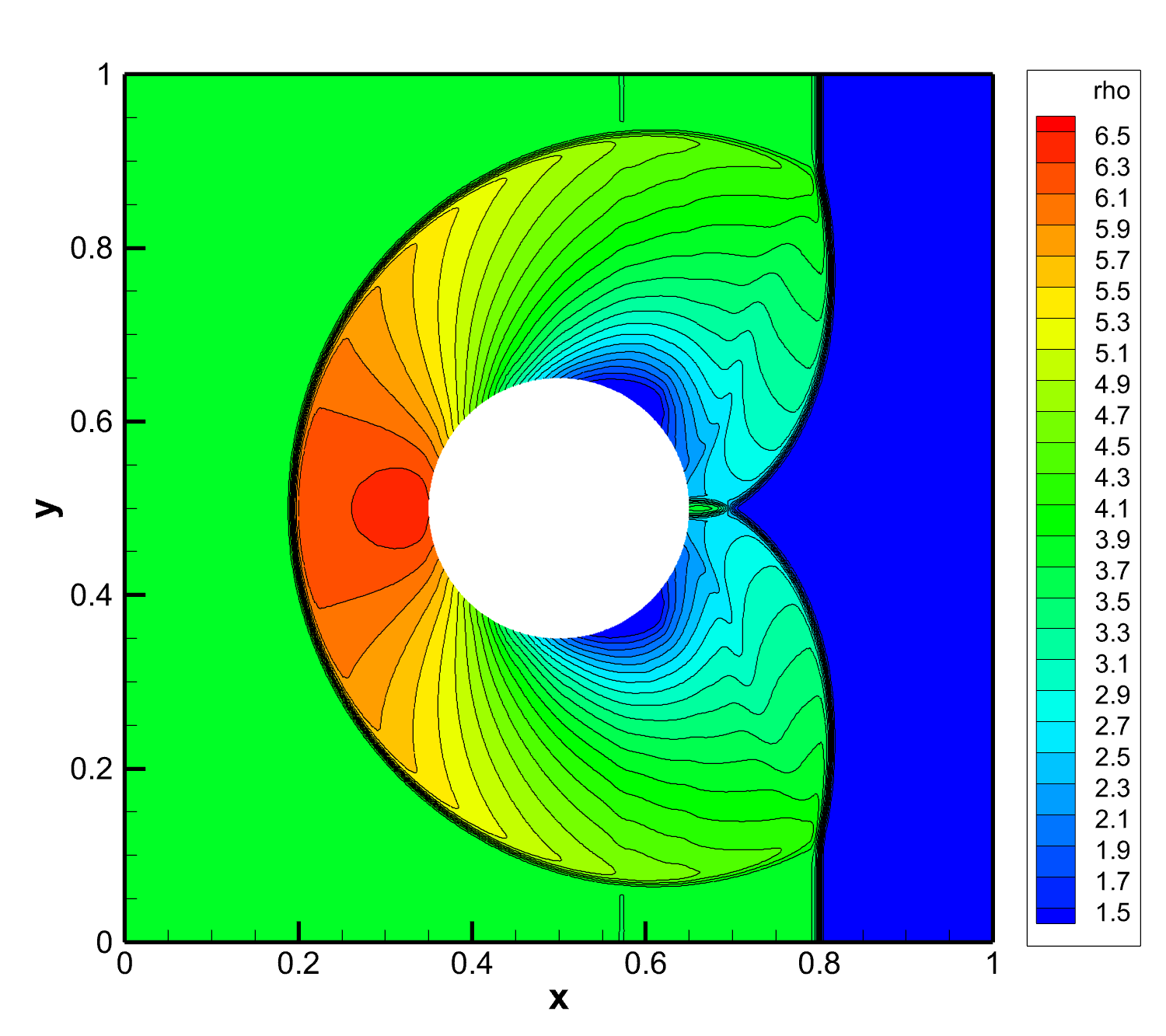}
    \par
    \caption{Density contour colors at time $T_2=0.3$.}
  \end{subfigure}\hfill
  \begin{subfigure}[t]{0.48\textwidth}\centering
    \raisebox{2.52pt}{\includegraphics[trim=50 20 50 20,clip,width=0.95413\textwidth]{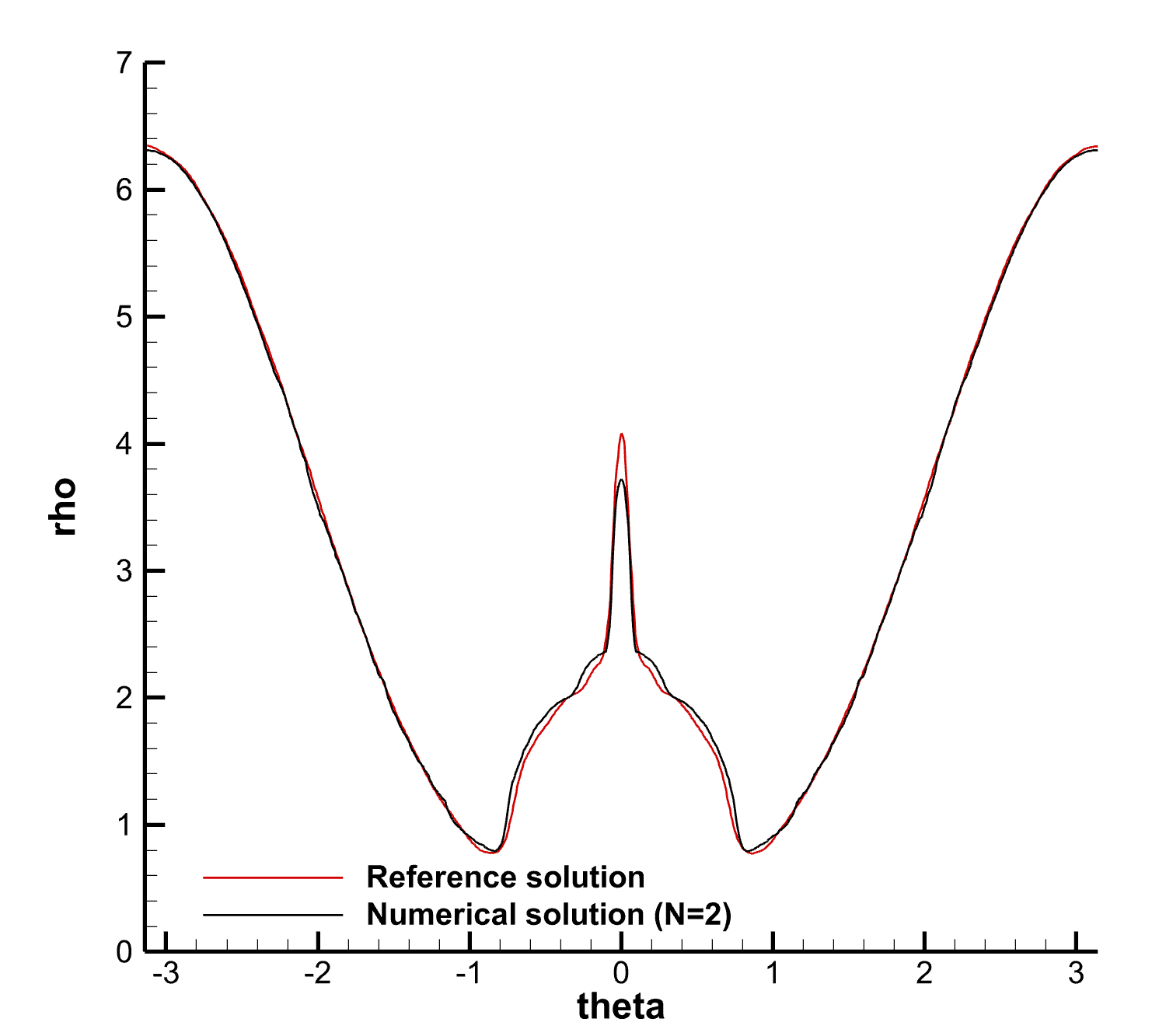}}
    \par
    \caption{Density along the cylinder surface at time $T_2=0.3$ and comparison with the reference solution.}
  \end{subfigure}
  \caption{Shock diffraction around a cylinder, $N=2$, $150\times165$ cells.}
  \label{fig:cylshock}
\end{figure}

In both cases the density profile along the cylinder agrees well with the two reference solutions, which have been obtained on Cartesian cut cells, but with different numerical schemes. The mean absolute difference is $0.029$ ($0.6\%$ of the range of the reference density along the cylinder) compared to the method of lines solution at $T_1=0.25$ and $0.043$
($0.8\%$) compared to the solution of May and Berger at $T_2=0.3$. The maximum density of the central density spike at time $t=0.3$ is lower in our solution than in the reference solution. This is probably due to our first order artificial viscosity, while the reference solution has been computed with a second order FV scheme that uses, on the cut cells, the LP limiter of \cite{MayBerger2013}, which limits the $x$- and $y$-components of the gradient separately and has been specifically developed for cut cell meshes.

\FloatBarrier %
\subsection{Computational cost}\label{sec:num-cost}

The main motivation for ADER-ZIST is that the expensive implicit space-time DG scheme is needed only in a thin layer of cells along the embedded boundary, while all remaining cells can be advanced in time with the computationally efficient explicit ADER-DG scheme on the regular Cartesian background mesh, whose element-local predictor and fully discrete corrector are considerably cheaper than the solution of the monolithic nonlinear system of Step~2 in \cref{sec:threestep}. In order to quantify this gain, we compare ADER-ZIST with a globally implicit space-time DG scheme used on the entire domain. This globally implicit scheme is directly obtained from the mixed scheme simply by setting the flag array of all explicit cells to the value 2, i.e.\ by treating all Cartesian cells as transition cells, so that the space-time DG polynomials of all cells of the mesh are computed together with the Newton--Krylov method described in \cref{sec:implementation}. Both variants are run on the same meshes, with the same parameters and with the same time step, namely the one of ADER-ZIST given by \eqref{eq:cfl}, supplemented by \eqref{eq:cflvisc} for the Sod problem. The number of time steps executed by ADER-ZIST and the globally implicit space-time DG method was the same in all cases and is reported in \cref{tab:cpu}; for the scattering problem it corresponds to the final time $T=6$ of \cref{sec:num-waves}, while for the Ringleb flow and the Sod shock tube a fixed number of time steps has been carried out. Since both schemes use the same code base and the same implementation of the implicit solver, the comparison between the two methods is as fair as possible. It therefore quantifies how much computing time is saved when the coupled nonlinear system is restricted to the cut cells and their Cartesian face neighbors, a saving that of course also depends on how the Newton--Krylov solver behaves for the smaller and for the larger implicit system. The CPU times refer to the execution of the main time loop only, i.e.\ without the computation of the cut cell geometry and without any output. If not stated otherwise, numerical experiments have been carried out one after the other pinned to one single CPU core, fixed to 4.4~GHz clock frequency, of an Intel Core Ultra 7 165H processor with 32~GB of RAM, and with no other jobs running on the same machine.

For the performance comparison we consider three test problems of rather different character:
\begin{itemize}
  \item the Ringleb flow of \cref{sec:num-ringleb} with $N=2$ and $N=3$ on the four meshes with $h=0.2$, $0.1$, $0.05$ and $0.025$ of \cref{tab:ringleb}, which is stationary, so that the Newton method of the implicit solver needs only very few iterations per time step, and which therefore represents the most favorable situation for the globally implicit scheme;
  \item the acoustic pulse scattered by five cylinders of \cref{sec:num-waves} with $N=3$ on the mesh of $106 \times 106$ background cells, which is smooth and genuinely time-dependent;
  \item the Sod shock tube with walls of \cref{sec:num-sod} with $N=2$, in which the solution is discontinuous, hence the artificial viscosity is active and the preconditioner is recomputed every five time steps; due to the time dependence and the non-smooth solution of this nonlinear problem, several Newton iterations per time step are needed. 
\end{itemize}

The obtained results are summarized in \cref{tab:cpu}. The two schemes (mixed explicit-implicit, globally implicit) produce essentially the same solution up to small differences: for the Ringleb flow, the $L^1$ errors of the density obtained with the two schemes differ by less than $6\%$ relative to the error of ADER-ZIST, and for the scattering problem and the Sod shock tube, the relative $L^1$ difference $\delta_z = \| z_{\mathrm{mixed}} - z_{\mathrm{impl}} \|_{L^1(\Omega)} / \| z_{\mathrm{mixed}} \|_{L^1(\Omega)}$ between the two numerical solutions after the last time step amounts to $4.8 \cdot 10^{-3}$ for the pressure perturbation $z = p - 1/\gamma$ and to $1.9 \cdot 10^{-3}$ for the density $z = \rho$, respectively, while the globally implicit space-time DG scheme is between $1.0$ and $5.2$ times as expensive as ADER-ZIST. As expected, the ratio of the CPU times grows with mesh refinement, since the number of implicit cells of ADER-ZIST, which are located along the embedded boundary, grows only like $h^{-1}$, while the number of explicit cells grows like $h^{-2}$. For the Ringleb flow with $N=3$, for instance, the share of the explicit cells increases from $28\%$ on the coarsest to $88\%$ on the finest mesh, and the ratio of the CPU times increases accordingly from $1.1$ to $4.8$, while the largest ratio of $5.2$ is obtained for the scattering problem, in which $92\%$ of the cells are explicit.

A more direct measure of the cost of the two schemes is an effective CPU time per cell and time step. We estimate it from the CPU time spent outside the implicit solver in ADER-ZIST and from the increase of the CPU time of the implicit solver in the globally implicit scheme, both divided by the number of explicit cells of the mixed scheme and by the number of time steps. This effective cost per cell also contains the additional work caused by the larger coupled nonlinear system, i.e.\ by changes in the number of Newton and GMRES iterations and by the computation of the preconditioner. In this sense, a cell treated with the implicit space-time DG scheme instead of the explicit ADER-DG scheme costs about $12$ to $16$ times as much for the Ringleb flow with $N=2$ and about $30$ times as much with $N=3$, both on the two finest meshes, $15$ times as much for the scattering problem and $28$ times as much for the Sod shock tube. Consequently, even in ADER-ZIST most of the total CPU time is spent in the implicit cells, namely between $88\%$ and $98\%$ for the Ringleb flow and still $70\%$ for the scattering problem, in which the implicit cells are only $8\%$ of all cells, so that the overall cost of the method is essentially determined by the number of cut cells and transition cells.

Finally, we compare the mixed scheme also with the domain of dependence (DoD) stabilization \cite{EngwerMayNuessingStreitbuerger2020,StreitbuergerBirkeEngwerMay2023}, which stabilizes an explicit DG scheme on cut cell meshes via suitable additional terms that restore the correct domain of dependence of the neighbors of small cut cells, and of which an independent implementation in the \texttt{deal.II} library \cite{dealII98,dealIIdesign} is available.
As a benchmark, we solve the linear scalar advection equation $\partial_t q + \mathbf{v} \cdot \nabla q = 0$ on the ramp geometry of \cref{sec:num-ramp}, with the constant velocity $\mathbf{v}$ of \eqref{eqn.ramp.vp} parallel to the ramp and with the exact solution given by the right-hand side of \eqref{eqn.ramp.rho}, which also serves as initial condition and as Dirichlet boundary condition, on the same meshes and with the same polynomial degrees $N=1,2,3$ and the same final time $T=0.02$ in both codes.
This comparison study was carried out on an Intel Core Ultra 7 255H processor with 32~GB of RAM, running one simulation at a time, pinned to a single CPU core at 5.1~GHz.
For both codes, i.e.\ for the \texttt{deal.II} implementation of the DoD stabilization as well as for the two CC2D variants, the reported CPU timings contain only the time loop, excluding any setup time. The DoD stabilization is combined with the explicit Heun method for $N=1$ and with the third order SSP Runge-Kutta scheme for $N=2$, both with $\mathrm{CFL}=0.4$, as well as with the classical fourth order Runge-Kutta scheme with $\mathrm{CFL}=0.3$ for $N=3$, while CC2D uses $\mathrm{CFL}=0.4$ in \eqref{eq:cfl} for all $N$, and both CC2D variants execute the same number of time steps.

The results are summarized in \cref{tab:dod}. In the $L^1$ norm, all three schemes converge with the nominal order $N+1$ and their errors are very close to each other: the errors of ADER-ZIST are between $3\%$ and $18\%$ smaller than those of the DoD stabilization, while those of the globally implicit space-time DG scheme are at most $4\%$ larger. In the $L^\infty$ norm, the errors of the two CC2D variants, which are practically identical, are larger than those of the DoD stabilization by a factor between $1.3$ and $2.4$. As expected, the fully explicit DoD stabilization is faster than ADER-ZIST, which has to solve an implicit system on the cut cells and on the transition cells in each time step: the CPU time of the mixed scheme is between $1.4$ and $2.7$ times the one of the DoD stabilization, while the globally implicit space-time DG scheme is again $2.2$ to $4.0$ times as expensive as ADER-ZIST. On the other hand, ADER-ZIST applies without any modification to nonlinear hyperbolic systems in multiple space dimensions, such as the Euler equations of all other test cases of this section, while the DoD stabilization has so far been developed for linear conservation laws in two space dimensions \cite{EngwerMayNuessingStreitbuerger2020,StreitbuergerBirkeEngwerMay2023,BirkeEngwerMayPetriRanocha2026} and for nonlinear conservation laws in one space dimension \cite{MayStreitbuerger2022}, and its extension to nonlinear systems in multiple space dimensions is the subject of ongoing research.

\begin{table}[!htbp]
  \centering
  \caption{Computational cost of ADER-ZIST (mixed) and of the globally implicit space-time DG scheme (impl.) on the same meshes, with the same time step and the same number of time steps: number of time steps, number of explicit and implicit cells of ADER-ZIST, percentage of explicit cells, CPU times and their ratio.}
  \label{tab:cpu}
  \begin{footnotesize}
  \input{cpu.tex}
  \end{footnotesize}
\end{table}

\begin{table}[!htbp]
  \centering
  \caption{Linear scalar advection along the ramp of \cref{sec:num-ramp}: $L^1$ and $L^\infty$ errors and CPU times of the DG scheme with DoD stabilization (DoD, \texttt{deal.II}), of ADER-ZIST (mixed) and of the globally implicit space-time DG scheme (impl.), the latter two computed with CC2D, all on the same meshes and on the same machine.}
  \label{tab:dod}
  \begin{footnotesize}
  \input{dod.tex}
  \end{footnotesize}
\end{table}

\clearpage

\section{Conclusion}\label{sec:conclusion}

In this paper we have presented a new mixed explicit-implicit ADER-space-time DG scheme, called ADER-ZIST, for the solution of nonlinear hyperbolic systems of conservation laws on Cartesian cut cell meshes. Our method solves the well-known small cell problem via the use of a zonally implicit space-time DG scheme in the cut cells and in their direct neighbors. 

More precisely, away from the embedded boundaries, the discrete solution is evolved in time with a classical \textit{explicit} Cartesian one-step ADER-DG scheme, which consists of an element-local space-time predictor and of a fully discrete corrector step, while on the cut cells and on their Cartesian face neighbors, the so-called transition cells, a zonally \textit{implicit} space-time DG scheme is used. Both schemes employ space-time polynomials of degree $N$ on each space-time cell and the same temporal quadrature points, so that the explicit and the implicit region are coupled with each other simply via the arguments of the numerical flux, using the ADER predictor on the explicit cells and the space-time DG polynomial on the implicit cells. The resulting scheme is conservative, and its time step only depends on the size of the background cells and not on the size of the cut cells. The small cell problem is therefore overcome without cell merging, state redistribution (SRD) or domain of dependence (DoD) stabilization, and without any reduction of the polynomial degree on small cut cells. The nonlinear system of the implicit cells, which are only a thin layer of cells along the embedded boundary, is solved with an inexact matrix-free Newton--Krylov method with a block Jacobi preconditioner. Curved embedded boundaries are taken into account with the aid of isoparametric sub-triangles along a cubic spline representation of the boundary, and for problems with shock waves an artificial viscosity has been added, which is discretized with a non-symmetric interior penalty Galerkin method and which is controlled by an oscillation-eliminating smoothness sensor adapted from the one of Abgrall and Liu \cite{AbgrallLiu2026}.

The accuracy and the robustness of the new scheme have been assessed on a set of test problems. In one space dimension, ADER-ZIST attains the nominal order $N+1$ on the tested cut cell meshes, with errors comparable to those of the unmixed schemes on equidistant meshes. A separate one-step comparison identifies a local coupling contribution of order $N+1$, which changes the leading local error without reducing its order in that test. In two space dimensions, the convergence studies carried out for the compressible Euler equations, namely the advection of a density profile along a ramp, an isentropic vortex around a cylinder, the Ringleb flow in a non-convex channel, a supersonic vortex between two circular arcs and the low Mach number flow around a cylinder, have shown that the scheme essentially reaches its nominal order of accuracy $N+1$ for polynomial degrees up to $N=5$ in most but not all cases. The Cartesian cut cell meshes under consideration in this paper have minimal volume fractions down to the order of $10^{-7}$, but the time step size has always been chosen according to the uniform background mesh and has never been affected by the size of the small cut cells. The results obtained for the scattering of an acoustic pulse at five cylinders and for the diffraction of a Mach 2 shock at a cylinder agree very well with the reference solutions available in the literature \cite{MayBerger2013,BergerGiuliani2021,Giuliani2022}, and the Sod shock tube with walls shows that the coupling of the explicit cells, the transition cells and the cut cells produces consistent solutions also in the presence of discontinuities.
A comparison with a globally implicit space-time DG scheme on the same meshes, with the same time step and the same number of time steps, has shown a speedup of ADER-ZIST of up to $5.2$ with comparable accuracy, and that its advantage grows under mesh refinement, as observed for the Ringleb flow, since on the finer meshes considered the effective cost of an implicit space-time DG cell per time step is about $12$ to $30$ times the one of an explicit ADER-DG cell. Compared to the fully explicit DoD stabilization, ADER-ZIST was between $1.4$ and $2.7$ times as expensive for linear advection, with comparable $L^1$ errors, but unlike DoD it directly applies also to nonlinear systems of hyperbolic conservation laws in multiple space dimensions, such as the compressible Euler equations of gasdynamics.

Future work will concern the extension of the method to three space dimensions, to moving embedded boundaries and to more general hyperbolic systems, including non-conservative ones, the combination with space-time AMR, as well as a theoretical analysis of the stability and of the accuracy of ADER-ZIST, at least for linear problems in one space dimension.

\section*{Acknowledgments}
M.D. was funded by the Fondazione Caritro via the project SOPHOS and by the European Research Council (ERC) under the European Union's Horizon Europe research and innovation programme via the project SOPHOS, grant agreement no.  ERC-ADG-2025-101265878-SOPHOS.
Views and opinions expressed are however those of the authors only and do not necessarily reflect those of the European Union or the European Research Council Executive Agency. Neither the European Union nor the granting authority can be held responsible for them.
M.D. was also financially supported by the Italian Ministry of University
and Research (MUR) via the  Departments of Excellence  Initiative 2018--2027 attributed to DICAM of the University of Trento (grant L. 232/2016).
M.D. is a member of the Gruppo Nazionale per il Calcolo Scientifico dell'Istituto Nazionale di Alta Matematica (GNCS-INdAM).

The research of C.E. was supported by the Deutsche Forschungsgemeinschaft (DFG, German Research Foundation) -- SPP 2410 Hyperbolic Balance Laws in Fluid Mechanics: Complexity, Scales, Randomness (CoScaRa), specifically within the project 526031774 (EsCUT: Entropy-stable high-order CUT-cell discontinuous Galerkin methods). C.E. further acknowledges support by the Deutsche Forschungsgemeinschaft (DFG, German Research Foundation) under Germany's Excellence Strategy EXC 2044 -- 390685587 and EXC 2044/2 -- 390685587, Mathematics M{\"u}nster: Dynamics--Geometry--Structure.

The research of S.M. was supported by the Deutsche Forschungsgemeinschaft (DFG, German Research Foundation) -- SPP 2410 Hyperbolic Balance Laws in Fluid Mechanics: Complexity, Scales, Randomness (CoScaRa), project number 586619126. S.M. was also funded by the Deutsche Forschungsgemeinschaft (DFG, German Research Foundation) under Germany's Excellence Strategy -- The Berlin Mathematics Research Center MATH+ (EXC-2046/2, project ID: 390685689).

\medskip

The authors thank Max Schr\"oder for his support with the numerical comparison of ADER-ZIST and the DoD stabilization.

\medskip

\section*{Declaration of generative AI and AI-assisted technologies}

During the preparation of this work, the authors used the AI assistants Claude (Anthropic), with the model Claude Opus 5.5, and ChatGPT/Codex (OpenAI), with the model Astra 6, in order to draft parts of the text, to search and verify the literature, and to improve language, notation and consistency. After using these tools, the authors reviewed and edited the text as needed and take full responsibility for the content of the published article.

The Fortran codes \texttt{CC1D} and \texttt{CC2D}, which produced the  computational results shown in this paper, were written from scratch by M.D. and S.M. Parts of the 2D code were also written and verified with the help of the AI system Claude Code 2.1.286 (Anthropic), using the model Claude Opus 5.5. The authors reviewed and edited the code as needed, verified it by extensive testing as well as by direct comparisons with an independent implementation in \texttt{deal.II}. The authors take full responsibility for the codes and for the numerical results obtained with them.

\medskip

\medskip

\textbf{This paper is dedicated to Prof. Eleuterio F. Toro on the occasion of his $\textbf{80}^{\textnormal{\textbf{th}}}$ birthday}.

\bibliographystyle{elsarticle-num}
\bibliography{references}

\typeout{get arXiv to do 4 passes: Label(s) may have changed. Rerun}
\end{document}

%% file: cutcell_mesh.tex
\begin{tikzpicture}[scale = 3.0]
  \begin{scope}
  \clip (0,0) rectangle (1.5,1);
  \draw[thick, blue] (0,0) rectangle (1.5,1);
  \draw[thick,step=0.1, blue] (0,0) grid (1.5,1);
    \fill[rust, opacity=0.2] (0.5,0.3) rectangle (0.6,0.4);
    \fill[rust, opacity=0.2] (0.5,0.4) rectangle (0.6,0.5);
    \fill[rust, opacity=0.2] (0.5,0.5) rectangle (0.6,0.6);
    \fill[rust, opacity=0.2] (0.5,0.6) rectangle (0.6,0.7);
    \fill[rust, opacity=0.2] (0.6,0.2) rectangle (0.7,0.3);
    \fill[rust, opacity=0.2] (0.6,0.3) rectangle (0.7,0.4);
    \fill[rust, opacity=0.2] (0.6,0.6) rectangle (0.7,0.7);
    \fill[rust, opacity=0.2] (0.6,0.7) rectangle (0.7,0.8);
    \fill[rust, opacity=0.2] (0.7,0.2) rectangle (0.8,0.3);
    \fill[rust, opacity=0.2] (0.7,0.7) rectangle (0.8,0.8);
    \fill[rust, opacity=0.2] (0.8,0.2) rectangle (0.9,0.3);
    \fill[rust, opacity=0.2] (0.8,0.3) rectangle (0.9,0.4);
    \fill[rust, opacity=0.2] (0.8,0.6) rectangle (0.9,0.7);
    \fill[rust, opacity=0.2] (0.8,0.7) rectangle (0.9,0.8);
    \fill[rust, opacity=0.2] (0.9,0.3) rectangle (1.0,0.4);
    \fill[rust, opacity=0.2] (0.9,0.4) rectangle (1.0,0.5);
    \fill[rust, opacity=0.2] (0.9,0.5) rectangle (1.0,0.6);
    \fill[rust, opacity=0.2] (0.9,0.6) rectangle (1.0,0.7);
    \draw[very thick, rust] (0.5,0.3) rectangle (0.6,0.4);
    \draw[very thick, rust] (0.5,0.4) rectangle (0.6,0.5);
    \draw[very thick, rust] (0.5,0.5) rectangle (0.6,0.6);
    \draw[very thick, rust] (0.5,0.6) rectangle (0.6,0.7);
    \draw[very thick, rust] (0.6,0.2) rectangle (0.7,0.3);
    \draw[very thick, rust] (0.6,0.3) rectangle (0.7,0.4);
    \draw[very thick, rust] (0.6,0.6) rectangle (0.7,0.7);
    \draw[very thick, rust] (0.6,0.7) rectangle (0.7,0.8);
    \draw[very thick, rust] (0.7,0.2) rectangle (0.8,0.3);
    \draw[very thick, rust] (0.7,0.7) rectangle (0.8,0.8);
    \draw[very thick, rust] (0.8,0.2) rectangle (0.9,0.3);
    \draw[very thick, rust] (0.8,0.3) rectangle (0.9,0.4);
    \draw[very thick, rust] (0.8,0.6) rectangle (0.9,0.7);
    \draw[very thick, rust] (0.8,0.7) rectangle (0.9,0.8);
    \draw[very thick, rust] (0.9,0.3) rectangle (1.0,0.4);
    \draw[very thick, rust] (0.9,0.4) rectangle (1.0,0.5);
    \draw[very thick, rust] (0.9,0.5) rectangle (1.0,0.6);
    \draw[very thick, rust] (0.9,0.6) rectangle (1.0,0.7);
  \draw[thick, rust, fill=objgray] (0.75,0.5) circle (0.225);
  \end{scope}
    \draw[dashed, rust, thick] (0.8,0.3) rectangle (0.9,0.4);
    \draw[dashed, gray] (0.9,0.4) -- (1.65,0.7);
    \draw[dashed, gray] (0.9,0.3) -- (1.65,0.3);
    \draw[thick, rust] (1.65,0.3) rectangle (2.05,0.7);
    \begin{scope}
      \clip (1.65,0.3) rectangle (2.05,0.7);
      \begin{scope}[shift={(-1.55,-0.9)}, scale=4]
        \fill[rust, opacity=0.3] (0.8,0.3) rectangle (0.9,0.4);
        \fill[objgray] (0.75,0.5) circle (0.225);
      \end{scope}
    \end{scope}
    \node[rust, anchor=south] at (1.85,0.72) {\footnotesize $\alpha \ll 1$};
\end{tikzpicture}

%% file: implicit_region.tex
\begin{tikzpicture}[scale = 4.5]
  \clip (0,0) rectangle (1.5,1);
  \draw[blue] (0,0) rectangle (1.5,1);
  \draw[step=0.05, blue] (0,0) grid (1.5,1);
  \fill[orange, opacity=0.45] (0.5,0.4) rectangle (0.55,0.45);
  \fill[orange, opacity=0.45] (0.5,0.45) rectangle (0.55,0.5);
  \fill[orange, opacity=0.45] (0.5,0.5) rectangle (0.55,0.55);
  \fill[orange, opacity=0.45] (0.5,0.55) rectangle (0.55,0.6);
  \fill[orange, opacity=0.45] (0.55,0.35) rectangle (0.6,0.4);
  \fill[orange, opacity=0.45] (0.55,0.4) rectangle (0.6,0.45);
  \fill[orange, opacity=0.45] (0.55,0.45) rectangle (0.6,0.5);
  \fill[orange, opacity=0.45] (0.55,0.5) rectangle (0.6,0.55);
  \fill[orange, opacity=0.45] (0.55,0.55) rectangle (0.6,0.6);
  \fill[orange, opacity=0.45] (0.55,0.6) rectangle (0.6,0.65);
  \fill[orange, opacity=0.45] (0.6,0.3) rectangle (0.65,0.35);
  \fill[orange, opacity=0.45] (0.6,0.35) rectangle (0.65,0.4);
  \fill[orange, opacity=0.45] (0.6,0.4) rectangle (0.65,0.45);
  \fill[orange, opacity=0.45] (0.6,0.45) rectangle (0.65,0.5);
  \fill[orange, opacity=0.45] (0.6,0.5) rectangle (0.65,0.55);
  \fill[orange, opacity=0.45] (0.6,0.55) rectangle (0.65,0.6);
  \fill[orange, opacity=0.45] (0.6,0.6) rectangle (0.65,0.65);
  \fill[orange, opacity=0.45] (0.6,0.65) rectangle (0.65,0.7);
  \fill[orange, opacity=0.45] (0.65,0.25) rectangle (0.7,0.3);
  \fill[orange, opacity=0.45] (0.65,0.3) rectangle (0.7,0.35);
  \fill[orange, opacity=0.45] (0.65,0.35) rectangle (0.7,0.4);
  \fill[orange, opacity=0.45] (0.65,0.4) rectangle (0.7,0.45);
  \fill[orange, opacity=0.45] (0.65,0.55) rectangle (0.7,0.6);
  \fill[orange, opacity=0.45] (0.65,0.6) rectangle (0.7,0.65);
  \fill[orange, opacity=0.45] (0.65,0.65) rectangle (0.7,0.7);
  \fill[orange, opacity=0.45] (0.65,0.7) rectangle (0.7,0.75);
  \fill[orange, opacity=0.45] (0.7,0.25) rectangle (0.75,0.3);
  \fill[orange, opacity=0.45] (0.7,0.3) rectangle (0.75,0.35);
  \fill[orange, opacity=0.45] (0.7,0.35) rectangle (0.75,0.4);
  \fill[orange, opacity=0.45] (0.7,0.6) rectangle (0.75,0.65);
  \fill[orange, opacity=0.45] (0.7,0.65) rectangle (0.75,0.7);
  \fill[orange, opacity=0.45] (0.7,0.7) rectangle (0.75,0.75);
  \fill[orange, opacity=0.45] (0.75,0.25) rectangle (0.8,0.3);
  \fill[orange, opacity=0.45] (0.75,0.3) rectangle (0.8,0.35);
  \fill[orange, opacity=0.45] (0.75,0.35) rectangle (0.8,0.4);
  \fill[orange, opacity=0.45] (0.75,0.6) rectangle (0.8,0.65);
  \fill[orange, opacity=0.45] (0.75,0.65) rectangle (0.8,0.7);
  \fill[orange, opacity=0.45] (0.75,0.7) rectangle (0.8,0.75);
  \fill[orange, opacity=0.45] (0.8,0.25) rectangle (0.85,0.3);
  \fill[orange, opacity=0.45] (0.8,0.3) rectangle (0.85,0.35);
  \fill[orange, opacity=0.45] (0.8,0.35) rectangle (0.85,0.4);
  \fill[orange, opacity=0.45] (0.8,0.4) rectangle (0.85,0.45);
  \fill[orange, opacity=0.45] (0.8,0.55) rectangle (0.85,0.6);
  \fill[orange, opacity=0.45] (0.8,0.6) rectangle (0.85,0.65);
  \fill[orange, opacity=0.45] (0.8,0.65) rectangle (0.85,0.7);
  \fill[orange, opacity=0.45] (0.8,0.7) rectangle (0.85,0.75);
  \fill[orange, opacity=0.45] (0.85,0.3) rectangle (0.9,0.35);
  \fill[orange, opacity=0.45] (0.85,0.35) rectangle (0.9,0.4);
  \fill[orange, opacity=0.45] (0.85,0.4) rectangle (0.9,0.45);
  \fill[orange, opacity=0.45] (0.85,0.45) rectangle (0.9,0.5);
  \fill[orange, opacity=0.45] (0.85,0.5) rectangle (0.9,0.55);
  \fill[orange, opacity=0.45] (0.85,0.55) rectangle (0.9,0.6);
  \fill[orange, opacity=0.45] (0.85,0.6) rectangle (0.9,0.65);
  \fill[orange, opacity=0.45] (0.85,0.65) rectangle (0.9,0.7);
  \fill[orange, opacity=0.45] (0.9,0.35) rectangle (0.95,0.4);
  \fill[orange, opacity=0.45] (0.9,0.4) rectangle (0.95,0.45);
  \fill[orange, opacity=0.45] (0.9,0.45) rectangle (0.95,0.5);
  \fill[orange, opacity=0.45] (0.9,0.5) rectangle (0.95,0.55);
  \fill[orange, opacity=0.45] (0.9,0.55) rectangle (0.95,0.6);
  \fill[orange, opacity=0.45] (0.9,0.6) rectangle (0.95,0.65);
  \fill[orange, opacity=0.45] (0.95,0.4) rectangle (1.0,0.45);
  \fill[orange, opacity=0.45] (0.95,0.45) rectangle (1.0,0.5);
  \fill[orange, opacity=0.45] (0.95,0.5) rectangle (1.0,0.55);
  \fill[orange, opacity=0.45] (0.95,0.55) rectangle (1.0,0.6);
  \draw[orange] (0.5,0.4) rectangle (0.55,0.45);
  \draw[orange] (0.5,0.45) rectangle (0.55,0.5);
  \draw[orange] (0.5,0.5) rectangle (0.55,0.55);
  \draw[orange] (0.5,0.55) rectangle (0.55,0.6);
  \draw[orange] (0.55,0.35) rectangle (0.6,0.4);
  \draw[orange] (0.55,0.4) rectangle (0.6,0.45);
  \draw[orange] (0.55,0.45) rectangle (0.6,0.5);
  \draw[orange] (0.55,0.5) rectangle (0.6,0.55);
  \draw[orange] (0.55,0.55) rectangle (0.6,0.6);
  \draw[orange] (0.55,0.6) rectangle (0.6,0.65);
  \draw[orange] (0.6,0.3) rectangle (0.65,0.35);
  \draw[orange] (0.6,0.35) rectangle (0.65,0.4);
  \draw[orange] (0.6,0.4) rectangle (0.65,0.45);
  \draw[orange] (0.6,0.45) rectangle (0.65,0.5);
  \draw[orange] (0.6,0.5) rectangle (0.65,0.55);
  \draw[orange] (0.6,0.55) rectangle (0.65,0.6);
  \draw[orange] (0.6,0.6) rectangle (0.65,0.65);
  \draw[orange] (0.6,0.65) rectangle (0.65,0.7);
  \draw[orange] (0.65,0.25) rectangle (0.7,0.3);
  \draw[orange] (0.65,0.3) rectangle (0.7,0.35);
  \draw[orange] (0.65,0.35) rectangle (0.7,0.4);
  \draw[orange] (0.65,0.4) rectangle (0.7,0.45);
  \draw[orange] (0.65,0.55) rectangle (0.7,0.6);
  \draw[orange] (0.65,0.6) rectangle (0.7,0.65);
  \draw[orange] (0.65,0.65) rectangle (0.7,0.7);
  \draw[orange] (0.65,0.7) rectangle (0.7,0.75);
  \draw[orange] (0.7,0.25) rectangle (0.75,0.3);
  \draw[orange] (0.7,0.3) rectangle (0.75,0.35);
  \draw[orange] (0.7,0.35) rectangle (0.75,0.4);
  \draw[orange] (0.7,0.6) rectangle (0.75,0.65);
  \draw[orange] (0.7,0.65) rectangle (0.75,0.7);
  \draw[orange] (0.7,0.7) rectangle (0.75,0.75);
  \draw[orange] (0.75,0.25) rectangle (0.8,0.3);
  \draw[orange] (0.75,0.3) rectangle (0.8,0.35);
  \draw[orange] (0.75,0.35) rectangle (0.8,0.4);
  \draw[orange] (0.75,0.6) rectangle (0.8,0.65);
  \draw[orange] (0.75,0.65) rectangle (0.8,0.7);
  \draw[orange] (0.75,0.7) rectangle (0.8,0.75);
  \draw[orange] (0.8,0.25) rectangle (0.85,0.3);
  \draw[orange] (0.8,0.3) rectangle (0.85,0.35);
  \draw[orange] (0.8,0.35) rectangle (0.85,0.4);
  \draw[orange] (0.8,0.4) rectangle (0.85,0.45);
  \draw[orange] (0.8,0.55) rectangle (0.85,0.6);
  \draw[orange] (0.8,0.6) rectangle (0.85,0.65);
  \draw[orange] (0.8,0.65) rectangle (0.85,0.7);
  \draw[orange] (0.8,0.7) rectangle (0.85,0.75);
  \draw[orange] (0.85,0.3) rectangle (0.9,0.35);
  \draw[orange] (0.85,0.35) rectangle (0.9,0.4);
  \draw[orange] (0.85,0.4) rectangle (0.9,0.45);
  \draw[orange] (0.85,0.45) rectangle (0.9,0.5);
  \draw[orange] (0.85,0.5) rectangle (0.9,0.55);
  \draw[orange] (0.85,0.55) rectangle (0.9,0.6);
  \draw[orange] (0.85,0.6) rectangle (0.9,0.65);
  \draw[orange] (0.85,0.65) rectangle (0.9,0.7);
  \draw[orange] (0.9,0.35) rectangle (0.95,0.4);
  \draw[orange] (0.9,0.4) rectangle (0.95,0.45);
  \draw[orange] (0.9,0.45) rectangle (0.95,0.5);
  \draw[orange] (0.9,0.5) rectangle (0.95,0.55);
  \draw[orange] (0.9,0.55) rectangle (0.95,0.6);
  \draw[orange] (0.9,0.6) rectangle (0.95,0.65);
  \draw[orange] (0.95,0.4) rectangle (1.0,0.45);
  \draw[orange] (0.95,0.45) rectangle (1.0,0.5);
  \draw[orange] (0.95,0.5) rectangle (1.0,0.55);
  \draw[orange] (0.95,0.55) rectangle (1.0,0.6);
  \draw[objgray, fill=objgray] (0.75,0.5) circle (0.16875);
\end{tikzpicture}

%% file: mixed_idea.tex
\begin{tikzpicture}[every node/.style={font=\small}]
  \def\hbot{0}
  \def\htop{1.2}
  \def\xcut{4.356} %
  \def\over{0.1} %
  \def\edgew{1mm}
  \def\edgeeps{0.05} %
  \draw[dash pattern=on 2pt off 2pt] (-0.69,\hbot) -- (0,\hbot);
  \draw[dash pattern=on 2pt off 2pt] (-0.69,\htop) -- (0,\htop);
  \draw[dash pattern=on 2pt off 2pt] (7.92,\hbot) -- (8.61,\hbot);
  \draw[dash pattern=on 2pt off 2pt] (7.92,\htop) -- (8.61,\htop);
  \draw[thick] (0,\hbot) -- (7.92,\hbot);
  \draw[thick] (0,\htop) -- (7.92,\htop);
  \foreach \xa/\xb/\col in {%
    0/1.32/blue,%
    1.32/2.64/blue,%
    2.64/3.96/orange,%
    3.96/4.356/orange,%
    4.356/5.28/orange,%
    5.28/6.6/blue,%
    6.6/7.92/blue%
  }{
    \draw[\col,line width=\edgew,opacity=0.4] (\xa+\edgeeps,\hbot) -- (\xa+\edgeeps,\htop);
    \draw[\col,line width=\edgew,opacity=0.4] (\xb-\edgeeps,\hbot) -- (\xb-\edgeeps,\htop);
  }
  \foreach \x in {0,1.32,2.64,3.96,5.28,6.6,7.92}
    \draw (\x,\hbot-\over) -- (\x,\htop+\over);
  \draw (\xcut,\hbot-\over) -- (\xcut,\htop+\over);
  \node[blue] at (0.66,0.6) {$p_h$};
  \node[blue] at (1.98,0.6) {$p_h$};
  \node[orange] at (3.3,0.6) {$q_h$};
  \node[orange,font=\scriptsize] at (4.158,0.6) {$q_h$}; %
  \node[orange] at (4.818,0.6) {$q_h$};
  \node[blue] at (5.94,0.6) {$p_h$};
  \node[blue] at (7.26,0.6) {$p_h$};
  \draw[->,thick] (-1.05,-0.5) -- (-1.05,1.7) node[above] {$t$};
  \draw (-1.1,\hbot) -- (-1.0,\hbot);
  \node[left] at (-1.15,\hbot) {$t^n$};
  \draw (-1.1,\htop) -- (-1.0,\htop);
  \node[left] at (-1.15,\htop) {$t^{n+1}$};
  \draw[->,thick] (4.158,-0.975) -- (4.158,-0.05);
  \node[below] at (4.158,-0.975) {cut cell};
  \draw[->,thick] (2.64,-0.475) -- (2.64,-0.05);
  \node[below] at (2.64,-0.475) {$\mathcal{F}(\textcolor{blue}{p_h},\textcolor{orange}{q_h})$};
  \draw[->,thick] (5.28,-0.475) -- (5.28,-0.05);
  \node[below] at (5.28,-0.475) {$\mathcal{F}(\textcolor{orange}{q_h},\textcolor{blue}{p_h})$};
\end{tikzpicture}

%% file: step1.tex
\begin{tikzpicture}[scale=0.95,transform shape,every node/.style={font=\small}]
    \def\hbot{0}
    \def\htop{1.2}
    \def\xcut{4.356}
    \def\over{0.1}
    \draw[dash pattern=on 2pt off 2pt] (-0.69,\hbot) -- (0,\hbot);
    \draw[dash pattern=on 2pt off 2pt] (-0.69,\htop) -- (0,\htop);
    \draw[dash pattern=on 2pt off 2pt] (7.92,\hbot) -- (8.61,\hbot);
    \draw[dash pattern=on 2pt off 2pt] (7.92,\htop) -- (8.61,\htop);
    \draw[thick] (0,\hbot) -- (7.92,\hbot);
    \draw[thick] (0,\htop) -- (7.92,\htop);
    \foreach \x in {0,1.32,2.64,3.96,5.28,6.6,7.92}
      \draw (\x,\hbot-\over) -- (\x,\htop+\over);
    \draw (\xcut,\hbot-\over) -- (\xcut,\htop+\over);
    \node[blue] at (0.66,0.6) {$p_h$};
    \node[blue] at (1.98,0.6) {$p_h$};
    \node[blue] at (5.94,0.6) {$p_h$};
    \node[blue] at (7.26,0.6) {$p_h$};
    \draw[->,thick] (-1.05,-0.25) -- (-1.05,1.45) node[above] {$t$};
    \draw (-1.1,\hbot) -- (-1.0,\hbot);
    \node[left] at (-1.15,\hbot) {$t^n$};
    \draw (-1.1,\htop) -- (-1.0,\htop);
    \node[left] at (-1.15,\htop) {$t^{n+1}$};
  \end{tikzpicture}

%% file: step2.tex
\begin{tikzpicture}[scale=0.95,transform shape,every node/.style={font=\small}]
    \def\hbot{0}
    \def\htop{1.2}
    \def\xcut{4.356}
    \def\over{0.1}
    \def\edgew{1mm}
    \def\edgeeps{0.05}
    \draw[dash pattern=on 2pt off 2pt] (-0.69,\hbot) -- (0,\hbot);
    \draw[dash pattern=on 2pt off 2pt] (-0.69,\htop) -- (0,\htop);
    \draw[dash pattern=on 2pt off 2pt] (7.92,\hbot) -- (8.61,\hbot);
    \draw[dash pattern=on 2pt off 2pt] (7.92,\htop) -- (8.61,\htop);
    \draw[thick] (0,\hbot) -- (7.92,\hbot);
    \draw[thick] (0,\htop) -- (7.92,\htop);
    \foreach \xa/\xb in {2.64/3.96,3.96/4.356,4.356/5.28}{
      \draw[orange,line width=\edgew,opacity=0.4] (\xa+\edgeeps,\hbot) -- (\xa+\edgeeps,\htop);
      \draw[orange,line width=\edgew,opacity=0.4] (\xb-\edgeeps,\hbot) -- (\xb-\edgeeps,\htop);
    }
    \draw[blue,line width=\edgew,opacity=0.4] (2.64-\edgeeps,\hbot) -- (2.64-\edgeeps,\htop);
    \draw[blue,line width=\edgew,opacity=0.4] (5.28+\edgeeps,\hbot) -- (5.28+\edgeeps,\htop);
    \draw[orange,line width=0.6mm] (2.64,\htop) -- (5.28,\htop);
    \foreach \x in {0,1.32,2.64,3.96,5.28,6.6,7.92}
      \draw (\x,\hbot-\over) -- (\x,\htop+\over);
    \draw (\xcut,\hbot-\over) -- (\xcut,\htop+\over);
    \node[orange] at (3.3,0.6) {$q_h$};
    \node[orange,font=\scriptsize] at (4.158,0.6) {$q_h$}; %
    \node[orange] at (4.818,0.6) {$q_h$};
    \draw[->,thick] (-1.05,-0.25) -- (-1.05,1.45) node[above] {$t$};
    \draw (-1.1,\hbot) -- (-1.0,\hbot);
    \node[left] at (-1.15,\hbot) {$t^n$};
    \draw (-1.1,\htop) -- (-1.0,\htop);
    \node[left] at (-1.15,\htop) {$t^{n+1}$};
  \end{tikzpicture}

%% file: step3.tex
\begin{tikzpicture}[scale=0.95,transform shape,every node/.style={font=\small}]
    \def\hbot{0}
    \def\htop{1.2}
    \def\xcut{4.356}
    \def\over{0.1}
    \def\edgew{1mm}
    \def\edgeeps{0.05}
    \draw[dash pattern=on 2pt off 2pt] (-0.69,\hbot) -- (0,\hbot);
    \draw[dash pattern=on 2pt off 2pt] (-0.69,\htop) -- (0,\htop);
    \draw[dash pattern=on 2pt off 2pt] (7.92,\hbot) -- (8.61,\hbot);
    \draw[dash pattern=on 2pt off 2pt] (7.92,\htop) -- (8.61,\htop);
    \draw[thick] (0,\hbot) -- (7.92,\hbot);
    \draw[thick] (0,\htop) -- (7.92,\htop);
    \foreach \xa/\xb in {0/1.32,1.32/2.64,5.28/6.6,6.6/7.92}{
      \draw[blue,line width=\edgew,opacity=0.4] (\xa+\edgeeps,\hbot) -- (\xa+\edgeeps,\htop);
      \draw[blue,line width=\edgew,opacity=0.4] (\xb-\edgeeps,\hbot) -- (\xb-\edgeeps,\htop);
    }
    \draw[orange,line width=\edgew,opacity=0.4] (2.64+\edgeeps,\hbot) -- (2.64+\edgeeps,\htop);
    \draw[orange,line width=\edgew,opacity=0.4] (5.28-\edgeeps,\hbot) -- (5.28-\edgeeps,\htop);
    \draw[blue,line width=0.6mm] (0,\htop) -- (2.64,\htop);
    \draw[blue,line width=0.6mm] (5.28,\htop) -- (7.92,\htop);
    \foreach \x in {0,1.32,2.64,3.96,5.28,6.6,7.92}
      \draw (\x,\hbot-\over) -- (\x,\htop+\over);
    \draw (\xcut,\hbot-\over) -- (\xcut,\htop+\over);
    \draw[->,thick] (-1.05,-0.25) -- (-1.05,1.45) node[above] {$t$};
    \draw (-1.1,\hbot) -- (-1.0,\hbot);
    \node[left] at (-1.15,\hbot) {$t^n$};
    \draw (-1.1,\htop) -- (-1.0,\htop);
    \node[left] at (-1.15,\htop) {$t^{n+1}$};
  \end{tikzpicture}

%% file: mesh_1d.tex
\providecommand{\meshrow}[4]{%
  \draw[blue] (0,#1) -- (9,#1);
  \foreach \bx in {#3} {
    \foreach \i in {1,...,19} {
      \draw[blue] ({\bx+\i*#2},{#1-0.15}) -- ({\bx+\i*#2},{#1+0.15});
    }
    \fill[rust, opacity=0.4] ({\bx+9*#2},{#1-0.2}) rectangle ({\bx+9.3*#2},{#1+0.2});
    \draw[rust, thick] ({\bx+9*#2},{#1-0.2}) -- ({\bx+9*#2},{#1+0.2});
    \draw[rust, thick] ({\bx+9.3*#2},{#1-0.2}) -- ({\bx+9.3*#2},{#1+0.2});
  }
  \foreach \bx in {#4} {
    \draw[very thick] (\bx,{#1-0.28}) -- (\bx,{#1+0.28});
  }
}
\begin{tikzpicture}[every node/.style={font=\small}]
  \meshrow{0.98}{0.45}{0}{0,9}
  \meshrow{0}{0.225}{0,4.5}{0,4.5,9}
  \node[anchor=east] at (-0.3,0.98) {$L=1,\; I_{\max}=20$};
  \node[anchor=east] at (-0.3,0)   {$L=2,\; I_{\max}=40$};
  \foreach \cx in {2.06, 6.56} {
    \draw[->, thick] (\cx,-0.58) -- (\cx,-0.28);
    \node[below] at (\cx,-0.58) {cut cell};
  }
\end{tikzpicture}

%% file: onedim_errors.tex
\begin{tabular}{c c rc rc rc}
\toprule
 & & \multicolumn{2}{c}{ADER-DG} & \multicolumn{2}{c}{space-time DG} & \multicolumn{2}{c}{ADER-ZIST} \\
\cmidrule(lr){3-4}\cmidrule(lr){5-6}\cmidrule(lr){7-8}
$N$ & $I_{\max}$ & error & order & error & order & error & order \\
\midrule
\multicolumn{8}{l}{\textit{$L^1$ error after one time step}} \\
1 & 40 & $4.51\cdot 10^{-3\phantom{0}}$ & -- & $4.67\cdot 10^{-3\phantom{0}}$ & -- & $4.51\cdot 10^{-3\phantom{0}}$ & -- \\
 & 80 & $1.13\cdot 10^{-3\phantom{0}}$ & 2.00 & $1.17\cdot 10^{-3\phantom{0}}$ & 2.00 & $1.13\cdot 10^{-3\phantom{0}}$ & 1.99 \\
 & 160 & $2.82\cdot 10^{-4\phantom{0}}$ & 2.00 & $2.91\cdot 10^{-4\phantom{0}}$ & 2.00 & $2.83\cdot 10^{-4\phantom{0}}$ & 2.00 \\
 & 320 & $7.04\cdot 10^{-5\phantom{0}}$ & 2.00 & $7.28\cdot 10^{-5\phantom{0}}$ & 2.00 & $7.06\cdot 10^{-5\phantom{0}}$ & 2.00 \\
\addlinespace
2 & 40 & $1.25\cdot 10^{-4\phantom{0}}$ & -- & $1.30\cdot 10^{-4\phantom{0}}$ & -- & $1.26\cdot 10^{-4\phantom{0}}$ & -- \\
 & 80 & $1.56\cdot 10^{-5\phantom{0}}$ & 3.01 & $1.61\cdot 10^{-5\phantom{0}}$ & 3.01 & $1.56\cdot 10^{-5\phantom{0}}$ & 3.02 \\
 & 160 & $1.94\cdot 10^{-6\phantom{0}}$ & 3.00 & $2.01\cdot 10^{-6\phantom{0}}$ & 3.00 & $1.95\cdot 10^{-6\phantom{0}}$ & 3.00 \\
 & 320 & $2.43\cdot 10^{-7\phantom{0}}$ & 3.00 & $2.51\cdot 10^{-7\phantom{0}}$ & 3.00 & $2.44\cdot 10^{-7\phantom{0}}$ & 3.00 \\
\addlinespace
3 & 40 & $2.34\cdot 10^{-6\phantom{0}}$ & -- & $2.42\cdot 10^{-6\phantom{0}}$ & -- & $2.34\cdot 10^{-6\phantom{0}}$ & -- \\
 & 80 & $1.46\cdot 10^{-7\phantom{0}}$ & 4.01 & $1.51\cdot 10^{-7\phantom{0}}$ & 4.01 & $1.46\cdot 10^{-7\phantom{0}}$ & 4.00 \\
 & 160 & $9.08\cdot 10^{-9\phantom{0}}$ & 4.00 & $9.41\cdot 10^{-9\phantom{0}}$ & 4.00 & $9.12\cdot 10^{-9\phantom{0}}$ & 4.00 \\
 & 320 & $5.67\cdot 10^{-10}$ & 4.00 & $5.88\cdot 10^{-10}$ & 4.00 & $5.70\cdot 10^{-10}$ & 4.00 \\
\midrule
\multicolumn{8}{l}{\textit{$L^\infty$ error after one time step}} \\
1 & 40 & $1.21\cdot 10^{-2\phantom{0}}$ & -- & $1.27\cdot 10^{-2\phantom{0}}$ & -- & $1.21\cdot 10^{-2\phantom{0}}$ & -- \\
 & 80 & $3.04\cdot 10^{-3\phantom{0}}$ & 1.99 & $3.19\cdot 10^{-3\phantom{0}}$ & 1.99 & $3.19\cdot 10^{-3\phantom{0}}$ & 1.92 \\
 & 160 & $7.60\cdot 10^{-4\phantom{0}}$ & 2.00 & $7.98\cdot 10^{-4\phantom{0}}$ & 2.00 & $7.60\cdot 10^{-4\phantom{0}}$ & 2.07 \\
 & 320 & $1.90\cdot 10^{-4\phantom{0}}$ & 2.00 & $2.00\cdot 10^{-4\phantom{0}}$ & 2.00 & $1.90\cdot 10^{-4\phantom{0}}$ & 2.00 \\
\addlinespace
2 & 40 & $1.87\cdot 10^{-4\phantom{0}}$ & -- & $2.14\cdot 10^{-4\phantom{0}}$ & -- & $2.36\cdot 10^{-4\phantom{0}}$ & -- \\
 & 80 & $2.36\cdot 10^{-5\phantom{0}}$ & 2.99 & $2.61\cdot 10^{-5\phantom{0}}$ & 3.04 & $2.36\cdot 10^{-5\phantom{0}}$ & 3.32 \\
 & 160 & $2.96\cdot 10^{-6\phantom{0}}$ & 3.00 & $3.25\cdot 10^{-6\phantom{0}}$ & 3.01 & $2.96\cdot 10^{-6\phantom{0}}$ & 3.00 \\
 & 320 & $3.70\cdot 10^{-7\phantom{0}}$ & 3.00 & $4.05\cdot 10^{-7\phantom{0}}$ & 3.00 & $4.23\cdot 10^{-7\phantom{0}}$ & 2.80 \\
\addlinespace
3 & 40 & $8.07\cdot 10^{-6\phantom{0}}$ & -- & $9.23\cdot 10^{-6\phantom{0}}$ & -- & $8.07\cdot 10^{-6\phantom{0}}$ & -- \\
 & 80 & $5.08\cdot 10^{-7\phantom{0}}$ & 3.99 & $5.81\cdot 10^{-7\phantom{0}}$ & 3.99 & $5.81\cdot 10^{-7\phantom{0}}$ & 3.80 \\
 & 160 & $3.18\cdot 10^{-8\phantom{0}}$ & 4.00 & $3.64\cdot 10^{-8\phantom{0}}$ & 4.00 & $3.18\cdot 10^{-8\phantom{0}}$ & 4.19 \\
 & 320 & $1.99\cdot 10^{-9\phantom{0}}$ & 4.00 & $2.27\cdot 10^{-9\phantom{0}}$ & 4.00 & $2.11\cdot 10^{-9\phantom{0}}$ & 3.92 \\
\midrule
\multicolumn{8}{l}{\textit{$L^1$ error at $t = T$}} \\
1 & 40 & $1.68\cdot 10^{-2\phantom{0}}$ & -- & $9.37\cdot 10^{-3\phantom{0}}$ & -- & $1.55\cdot 10^{-2\phantom{0}}$ & -- \\
 & 80 & $4.10\cdot 10^{-3\phantom{0}}$ & 2.03 & $1.90\cdot 10^{-3\phantom{0}}$ & 2.30 & $3.83\cdot 10^{-3\phantom{0}}$ & 2.02 \\
 & 160 & $1.02\cdot 10^{-3\phantom{0}}$ & 2.01 & $4.28\cdot 10^{-4\phantom{0}}$ & 2.15 & $9.34\cdot 10^{-4\phantom{0}}$ & 2.03 \\
 & 320 & $2.54\cdot 10^{-4\phantom{0}}$ & 2.00 & $1.02\cdot 10^{-4\phantom{0}}$ & 2.07 & $2.34\cdot 10^{-4\phantom{0}}$ & 2.00 \\
\addlinespace
2 & 40 & $3.45\cdot 10^{-4\phantom{0}}$ & -- & $1.56\cdot 10^{-4\phantom{0}}$ & -- & $3.20\cdot 10^{-4\phantom{0}}$ & -- \\
 & 80 & $4.27\cdot 10^{-5\phantom{0}}$ & 3.01 & $1.94\cdot 10^{-5\phantom{0}}$ & 3.01 & $3.92\cdot 10^{-5\phantom{0}}$ & 3.03 \\
 & 160 & $5.33\cdot 10^{-6\phantom{0}}$ & 3.00 & $2.42\cdot 10^{-6\phantom{0}}$ & 3.00 & $4.91\cdot 10^{-6\phantom{0}}$ & 3.00 \\
 & 320 & $6.65\cdot 10^{-7\phantom{0}}$ & 3.00 & $3.02\cdot 10^{-7\phantom{0}}$ & 3.00 & $6.14\cdot 10^{-7\phantom{0}}$ & 3.00 \\
\addlinespace
3 & 40 & $5.22\cdot 10^{-6\phantom{0}}$ & -- & $2.99\cdot 10^{-6\phantom{0}}$ & -- & $4.88\cdot 10^{-6\phantom{0}}$ & -- \\
 & 80 & $3.28\cdot 10^{-7\phantom{0}}$ & 3.99 & $1.86\cdot 10^{-7\phantom{0}}$ & 4.00 & $3.08\cdot 10^{-7\phantom{0}}$ & 3.98 \\
 & 160 & $2.05\cdot 10^{-8\phantom{0}}$ & 4.00 & $1.16\cdot 10^{-8\phantom{0}}$ & 4.00 & $1.92\cdot 10^{-8\phantom{0}}$ & 4.00 \\
 & 320 & $1.28\cdot 10^{-9\phantom{0}}$ & 4.00 & $7.27\cdot 10^{-10}$ & 4.00 & $1.20\cdot 10^{-9\phantom{0}}$ & 4.00 \\
\midrule
\multicolumn{8}{l}{\textit{$L^\infty$ error at $t = T$}} \\
1 & 40 & $1.44\cdot 10^{-2\phantom{0}}$ & -- & $1.08\cdot 10^{-2\phantom{0}}$ & -- & $1.35\cdot 10^{-2\phantom{0}}$ & -- \\
 & 80 & $4.03\cdot 10^{-3\phantom{0}}$ & 1.83 & $3.42\cdot 10^{-3\phantom{0}}$ & 1.66 & $3.84\cdot 10^{-3\phantom{0}}$ & 1.81 \\
 & 160 & $1.06\cdot 10^{-3\phantom{0}}$ & 1.93 & $9.42\cdot 10^{-4\phantom{0}}$ & 1.86 & $1.17\cdot 10^{-3\phantom{0}}$ & 1.71 \\
 & 320 & $2.71\cdot 10^{-4\phantom{0}}$ & 1.96 & $2.46\cdot 10^{-4\phantom{0}}$ & 1.94 & $3.02\cdot 10^{-4\phantom{0}}$ & 1.96 \\
\addlinespace
2 & 40 & $4.29\cdot 10^{-4\phantom{0}}$ & -- & $5.11\cdot 10^{-4\phantom{0}}$ & -- & $5.44\cdot 10^{-4\phantom{0}}$ & -- \\
 & 80 & $5.39\cdot 10^{-5\phantom{0}}$ & 2.99 & $6.44\cdot 10^{-5\phantom{0}}$ & 2.99 & $5.23\cdot 10^{-5\phantom{0}}$ & 3.38 \\
 & 160 & $6.74\cdot 10^{-6\phantom{0}}$ & 3.00 & $8.07\cdot 10^{-6\phantom{0}}$ & 3.00 & $8.32\cdot 10^{-6\phantom{0}}$ & 2.65 \\
 & 320 & $8.43\cdot 10^{-7\phantom{0}}$ & 3.00 & $1.01\cdot 10^{-6\phantom{0}}$ & 3.00 & $1.12\cdot 10^{-6\phantom{0}}$ & 2.90 \\
\addlinespace
3 & 40 & $7.74\cdot 10^{-6\phantom{0}}$ & -- & $1.15\cdot 10^{-5\phantom{0}}$ & -- & $7.58\cdot 10^{-6\phantom{0}}$ & -- \\
 & 80 & $4.83\cdot 10^{-7\phantom{0}}$ & 4.00 & $7.23\cdot 10^{-7\phantom{0}}$ & 3.99 & $7.41\cdot 10^{-7\phantom{0}}$ & 3.36 \\
 & 160 & $3.02\cdot 10^{-8\phantom{0}}$ & 4.00 & $4.53\cdot 10^{-8\phantom{0}}$ & 4.00 & $4.06\cdot 10^{-8\phantom{0}}$ & 4.19 \\
 & 320 & $1.89\cdot 10^{-9\phantom{0}}$ & 4.00 & $2.83\cdot 10^{-9\phantom{0}}$ & 4.00 & $2.90\cdot 10^{-9\phantom{0}}$ & 3.81 \\
\bottomrule
\end{tabular}

%% file: onedim_coupling.tex
\begin{tabular}{c r rc rc rc}
\toprule
 & & \multicolumn{2}{c}{$D$} & \multicolumn{2}{c}{$e_{\rm ADER}$} & \multicolumn{2}{c}{$e_{\rm ADER\text{-}ZIST}$} \\
\cmidrule(lr){3-4}\cmidrule(lr){5-6}\cmidrule(lr){7-8}
$N$ & $I_{\max}$ & value & order & value & order & value & order \\
\midrule
1 & 40 & $1.01\cdot 10^{-3\phantom{0}}$ & -- & $2.60\cdot 10^{-3\phantom{0}}$ & -- & $2.11\cdot 10^{-3\phantom{0}}$ & -- \\
 & 80 & $2.53\cdot 10^{-4\phantom{0}}$ & 1.99 & $6.56\cdot 10^{-4\phantom{0}}$ & 1.99 & $5.30\cdot 10^{-4\phantom{0}}$ & 1.99 \\
 & 160 & $6.34\cdot 10^{-5\phantom{0}}$ & 2.00 & $1.64\cdot 10^{-4\phantom{0}}$ & 2.00 & $1.33\cdot 10^{-4\phantom{0}}$ & 2.00 \\
 & 320 & $1.58\cdot 10^{-5\phantom{0}}$ & 2.00 & $4.11\cdot 10^{-5\phantom{0}}$ & 2.00 & $3.32\cdot 10^{-5\phantom{0}}$ & 2.00 \\
 & 640 & $3.96\cdot 10^{-6\phantom{0}}$ & 2.00 & $1.03\cdot 10^{-5\phantom{0}}$ & 2.00 & $8.30\cdot 10^{-6\phantom{0}}$ & 2.00 \\
\addlinespace
2 & 40 & $3.55\cdot 10^{-5\phantom{0}}$ & -- & $7.52\cdot 10^{-5\phantom{0}}$ & -- & $8.64\cdot 10^{-5\phantom{0}}$ & -- \\
 & 80 & $4.45\cdot 10^{-6\phantom{0}}$ & 3.00 & $9.36\cdot 10^{-6\phantom{0}}$ & 3.01 & $1.08\cdot 10^{-5\phantom{0}}$ & 3.00 \\
 & 160 & $5.56\cdot 10^{-7\phantom{0}}$ & 3.00 & $1.17\cdot 10^{-6\phantom{0}}$ & 3.00 & $1.35\cdot 10^{-6\phantom{0}}$ & 3.00 \\
 & 320 & $6.95\cdot 10^{-8\phantom{0}}$ & 3.00 & $1.46\cdot 10^{-7\phantom{0}}$ & 3.00 & $1.69\cdot 10^{-7\phantom{0}}$ & 3.00 \\
 & 640 & $8.69\cdot 10^{-9\phantom{0}}$ & 3.00 & $1.82\cdot 10^{-8\phantom{0}}$ & 3.00 & $2.11\cdot 10^{-8\phantom{0}}$ & 3.00 \\
\addlinespace
3 & 40 & $8.83\cdot 10^{-7\phantom{0}}$ & -- & $3.94\cdot 10^{-6\phantom{0}}$ & -- & $3.09\cdot 10^{-6\phantom{0}}$ & -- \\
 & 80 & $5.58\cdot 10^{-8\phantom{0}}$ & 3.98 & $2.51\cdot 10^{-7\phantom{0}}$ & 3.97 & $1.95\cdot 10^{-7\phantom{0}}$ & 3.98 \\
 & 160 & $3.49\cdot 10^{-9\phantom{0}}$ & 4.00 & $1.57\cdot 10^{-8\phantom{0}}$ & 4.00 & $1.22\cdot 10^{-8\phantom{0}}$ & 4.00 \\
 & 320 & $2.18\cdot 10^{-10}$ & 4.00 & $9.82\cdot 10^{-10}$ & 4.00 & $7.64\cdot 10^{-10}$ & 4.00 \\
 & 640 & $1.37\cdot 10^{-11}$ & 4.00 & $6.14\cdot 10^{-11}$ & 4.00 & $4.78\cdot 10^{-11}$ & 4.00 \\
\bottomrule
\end{tabular}

%% file: ramp.tex
\begin{tabular}{r c ccc ccc c}
\toprule
$I_{\max}$ & $\alpha_{\min}$ & L1 & L2 & L$\infty$ & $p_{\mathrm{L1}}$ & $p_{\mathrm{L2}}$ & $p_{\mathrm{L\infty}}$ & steps \\
\midrule
\multicolumn{9}{l}{\textit{$N=2$ (expected order 3)}} \\
40 & 3.10e-07 & 1.154e-04 & 1.011e-04 & 2.716e-04 & -- & -- & -- & 19 \\
60 & 6.98e-07 & 4.306e-05 & 3.766e-05 & 1.301e-04 & 2.43 & 2.44 & 1.81 & 28 \\
80 & 1.24e-06 & 2.034e-05 & 1.782e-05 & 8.187e-05 & 2.61 & 2.60 & 1.61 & 38 \\
100 & 1.94e-06 & 1.116e-05 & 9.757e-06 & 3.813e-05 & 2.69 & 2.70 & 3.42 & 47 \\
150 & 4.36e-06 & 3.633e-06 & 3.152e-06 & 1.517e-05 & 2.77 & 2.79 & 2.27 & 70 \\
\midrule
\multicolumn{9}{l}{\textit{$N=3$ (expected order 4)}} \\
40 & 3.10e-07 & 1.652e-06 & 1.544e-06 & 1.170e-05 & -- & -- & -- & 27 \\
60 & 6.98e-07 & 3.114e-07 & 3.001e-07 & 2.813e-06 & 4.12 & 4.04 & 3.52 & 40 \\
80 & 1.24e-06 & 8.992e-08 & 8.847e-08 & 6.720e-07 & 4.32 & 4.25 & 4.98 & 53 \\
100 & 1.94e-06 & 3.137e-08 & 3.231e-08 & 5.230e-07 & 4.72 & 4.51 & 1.12 & 66 \\
150 & 4.36e-06 & 6.254e-09 & 6.373e-09 & 8.731e-08 & 3.98 & 4.00 & 4.41 & 98 \\
\midrule
\multicolumn{9}{l}{\textit{$N=4$ (expected order 5)}} \\
10 & 5.54e-03 & 5.318e-05 & 5.019e-05 & 3.275e-04 & -- & -- & -- & 9 \\
25 & 9.28e-05 & 1.342e-06 & 1.172e-06 & 6.125e-06 & 4.02 & 4.10 & 4.34 & 21 \\
50 & 4.85e-07 & 6.098e-08 & 5.326e-08 & 4.356e-07 & 4.46 & 4.46 & 3.81 & 42 \\
75 & 1.07e-04 & 9.193e-09 & 8.029e-09 & 7.571e-08 & 4.67 & 4.67 & 4.32 & 63 \\
100 & 1.94e-06 & 2.327e-09 & 2.038e-09 & 3.240e-08 & 4.78 & 4.77 & 2.95 & 84 \\
\midrule
\multicolumn{9}{l}{\textit{$N=5$ (expected order 6)}} \\
10 & 5.54e-03 & 9.485e-06 & 8.448e-06 & 3.338e-05 & -- & -- & -- & 11 \\
20 & 7.76e-08 & 1.334e-07 & 1.190e-07 & 1.114e-06 & 6.15 & 6.15 & 4.90 & 21 \\
30 & 3.61e-04 & 1.009e-08 & 9.374e-09 & 7.400e-08 & 6.37 & 6.27 & 6.69 & 31 \\
40 & 3.10e-07 & 1.453e-09 & 1.431e-09 & 3.494e-08 & 6.74 & 6.53 & 2.61 & 42 \\
50 & 4.85e-07 & 3.210e-10 & 3.296e-10 & 1.758e-08 & 6.77 & 6.58 & 3.08 & 52 \\
\bottomrule
\end{tabular}

%% file: ramp_conv_L1.tex
\begin{tikzpicture}
\begin{loglogaxis}[scale only axis, width=106pt, height=106pt, xmin=0.01, xmax=0.25, ymin=1e-10, ymax=1e-2,
  ytick={1e-10,1e-8,1e-6,1e-4,1e-2}, tick label style={font=\scriptsize}, label style={font=\footnotesize}, xlabel={$h=2/I_{\max}$}, ylabel={$L^1$ error of $\rho$},
  xtick={0.01,0.02,0.05,0.1,0.2}, xticklabels={0.01,0.02,0.05,0.1,0.2}, grid=major, grid style={gray!25},
  legend columns=1, legend cell align=left, legend style={font=\footnotesize, at={(0.5,-0.31)}, anchor=north, draw=none, inner xsep=0pt},
  every axis plot/.append style={line width=0.8pt}, mark size=2pt]
\addplot+[blue, mark=o, only marks] table[x=h, y=e] {
h e
1.00000e-01 5.85236e-04
6.66667e-02 2.24187e-04
5.00000e-02 1.15432e-04
4.00000e-02 6.80062e-05
3.33333e-02 4.30634e-05
2.85714e-02 2.89443e-05
2.50000e-02 2.03365e-05
2.22222e-02 1.48211e-05
2.00000e-02 1.11596e-05
1.81818e-02 8.60781e-06
1.66667e-02 6.77098e-06
1.53846e-02 5.43053e-06
1.42857e-02 4.41605e-06
1.33333e-02 3.63331e-06
};
\addlegendentry{$N=2$: fit 2.54, nominal 3}
\addplot[blue, no marks, solid, forget plot] coordinates {(1.00000e-01,6.49932e-04) (1.33333e-02,3.92439e-06)};
\addplot[blue, no marks, dashed, forget plot] coordinates {(1.33333e-02,3.63331e-06) (1.00000e-01,1.53280e-03)};
\addplot+[red, mark=square, only marks] table[x=h, y=e] {
h e
1.00000e-01 3.03107e-05
6.66667e-02 5.42318e-06
5.00000e-02 1.65244e-06
4.00000e-02 6.64563e-07
3.33333e-02 3.11417e-07
2.85714e-02 1.62562e-07
2.50000e-02 8.99201e-08
2.22222e-02 5.15941e-08
2.00000e-02 3.13718e-08
1.81818e-02 2.00778e-08
1.66667e-02 1.39723e-08
1.53846e-02 1.01676e-08
1.42857e-02 8.00329e-09
1.33333e-02 6.25422e-09
};
\addlegendentry{$N=3$: fit 4.27, nominal 4}
\addplot[red, no marks, solid, forget plot] coordinates {(1.00000e-01,3.16695e-05) (1.33333e-02,5.82632e-09)};
\addplot[red, no marks, dashed, forget plot] coordinates {(1.33333e-02,6.25422e-09) (1.00000e-01,1.97888e-05)};
\addplot+[green!50!black, mark=triangle, only marks] table[x=h, y=e] {
h e
2.00000e-01 5.31794e-05
1.00000e-01 3.40387e-06
8.00000e-02 1.34208e-06
6.66667e-02 6.09002e-07
5.00000e-02 1.68644e-07
4.00000e-02 6.09818e-08
3.33333e-02 2.62848e-08
2.85714e-02 1.27622e-08
2.66667e-02 9.19312e-09
2.50000e-02 6.77339e-09
2.22222e-02 3.85641e-09
2.00000e-02 2.32719e-09
};
\addlegendentry{$N=4$: fit 4.41, nominal 5}
\addplot[green!50!black, no marks, solid, forget plot] coordinates {(2.00000e-01,6.79759e-05) (2.00000e-02,2.65191e-09)};
\addplot[green!50!black, no marks, dashed, forget plot] coordinates {(2.00000e-02,2.32719e-09) (2.00000e-01,2.32719e-04)};
\addplot+[violet, mark=diamond, only marks] table[x=h, y=e] {
h e
2.00000e-01 9.48549e-06
1.33333e-01 8.12151e-07
1.00000e-01 1.33377e-07
8.00000e-02 3.26585e-08
6.66667e-02 1.00914e-08
5.71429e-02 3.61483e-09
5.00000e-02 1.45337e-09
4.44444e-02 7.01526e-10
4.00000e-02 3.21036e-10
3.63636e-02 2.09184e-10
3.33333e-02 1.23716e-10
};
\addlegendentry{$N=5$: fit 6.36, nominal 6}
\addplot[violet, no marks, solid, forget plot] coordinates {(2.00000e-01,1.03557e-05) (3.33333e-02,1.16281e-10)};
\addplot[violet, no marks, dashed, forget plot] coordinates {(3.33333e-02,1.23716e-10) (2.00000e-01,5.77208e-06)};
\end{loglogaxis}
\end{tikzpicture}

%% file: ramp_conv_L2.tex
\begin{tikzpicture}
\begin{loglogaxis}[scale only axis, width=106pt, height=106pt, xmin=0.01, xmax=0.25, ymin=1e-10, ymax=1e-2,
  ytick={1e-10,1e-8,1e-6,1e-4,1e-2}, tick label style={font=\scriptsize}, label style={font=\footnotesize}, xlabel={$h=2/I_{\max}$}, ylabel={$L^2$ error of $\rho$},
  xtick={0.01,0.02,0.05,0.1,0.2}, xticklabels={0.01,0.02,0.05,0.1,0.2}, grid=major, grid style={gray!25},
  legend columns=1, legend cell align=left, legend style={font=\footnotesize, at={(0.5,-0.31)}, anchor=north, draw=none, inner xsep=0pt},
  every axis plot/.append style={line width=0.8pt}, mark size=2pt]
\addplot+[blue, mark=o, only marks] table[x=h, y=e] {
h e
1.00000e-01 5.19430e-04
6.66667e-02 1.98051e-04
5.00000e-02 1.01094e-04
4.00000e-02 5.93133e-05
3.33333e-02 3.76582e-05
2.85714e-02 2.53537e-05
2.50000e-02 1.78223e-05
2.22222e-02 1.29757e-05
2.00000e-02 9.75678e-06
1.81818e-02 7.51189e-06
1.66667e-02 5.89796e-06
1.53846e-02 4.72365e-06
1.42857e-02 3.83589e-06
1.33333e-02 3.15206e-06
};
\addlegendentry{$N=2$: fit 2.54, nominal 3}
\addplot[blue, no marks, solid, forget plot] coordinates {(1.00000e-01,5.75174e-04) (1.33333e-02,3.41035e-06)};
\addplot[blue, no marks, dashed, forget plot] coordinates {(1.33333e-02,3.15206e-06) (1.00000e-01,1.32978e-03)};
\addplot+[red, mark=square, only marks] table[x=h, y=e] {
h e
1.00000e-01 2.77070e-05
6.66667e-02 5.03238e-06
5.00000e-02 1.54414e-06
4.00000e-02 6.30424e-07
3.33333e-02 3.00117e-07
2.85714e-02 1.57723e-07
2.50000e-02 8.84686e-08
2.22222e-02 5.19999e-08
2.00000e-02 3.23086e-08
1.81818e-02 2.10821e-08
1.66667e-02 1.47582e-08
1.53846e-02 1.07526e-08
1.42857e-02 8.26393e-09
1.33333e-02 6.37257e-09
};
\addlegendentry{$N=3$: fit 4.19, nominal 4}
\addplot[red, no marks, solid, forget plot] coordinates {(1.00000e-01,2.83906e-05) (1.33333e-02,6.10512e-09)};
\addplot[red, no marks, dashed, forget plot] coordinates {(1.33333e-02,6.37257e-09) (1.00000e-01,2.01632e-05)};
\addplot+[green!50!black, mark=triangle, only marks] table[x=h, y=e] {
h e
2.00000e-01 5.01889e-05
1.00000e-01 2.97366e-06
8.00000e-02 1.17189e-06
6.66667e-02 5.33619e-07
5.00000e-02 1.47846e-07
4.00000e-02 5.32567e-08
3.33333e-02 2.29514e-08
2.85714e-02 1.11458e-08
2.66667e-02 8.02949e-09
2.50000e-02 5.91484e-09
2.22222e-02 3.36615e-09
2.00000e-02 2.03788e-09
};
\addlegendentry{$N=4$: fit 4.43, nominal 5}
\addplot[green!50!black, no marks, solid, forget plot] coordinates {(2.00000e-01,6.17565e-05) (2.00000e-02,2.29314e-09)};
\addplot[green!50!black, no marks, dashed, forget plot] coordinates {(2.00000e-02,2.03788e-09) (2.00000e-01,2.03788e-04)};
\addplot+[violet, mark=diamond, only marks] table[x=h, y=e] {
h e
2.00000e-01 8.44806e-06
1.33333e-01 7.20910e-07
1.00000e-01 1.18992e-07
8.00000e-02 2.97108e-08
6.66667e-02 9.37429e-09
5.71429e-02 3.45937e-09
5.00000e-02 1.43143e-09
4.44444e-02 7.34430e-10
4.00000e-02 3.29597e-10
3.63636e-02 2.56851e-10
3.33333e-02 1.60397e-10
};
\addlegendentry{$N=5$: fit 6.18, nominal 6}
\addplot[violet, no marks, solid, forget plot] coordinates {(2.00000e-01,8.39026e-06) (3.33333e-02,1.31021e-10)};
\addplot[violet, no marks, dashed, forget plot] coordinates {(3.33333e-02,1.60397e-10) (2.00000e-01,7.48350e-06)};
\end{loglogaxis}
\end{tikzpicture}

%% file: ramp_conv_Linf.tex
\begin{tikzpicture}
\begin{loglogaxis}[scale only axis, width=106pt, height=106pt, xmin=0.01, xmax=0.25, ymin=1e-10, ymax=1e-2,
  ytick={1e-10,1e-8,1e-6,1e-4,1e-2}, tick label style={font=\scriptsize}, label style={font=\footnotesize}, xlabel={$h=2/I_{\max}$}, ylabel={$L^\infty$ error of $\rho$},
  xtick={0.01,0.02,0.05,0.1,0.2}, xticklabels={0.01,0.02,0.05,0.1,0.2}, grid=major, grid style={gray!25},
  legend columns=1, legend cell align=left, legend style={font=\footnotesize, at={(0.5,-0.31)}, anchor=north, draw=none, inner xsep=0pt},
  every axis plot/.append style={line width=0.8pt}, mark size=2pt]
\addplot+[blue, mark=o, only marks] table[x=h, y=e] {
h e
1.00000e-01 1.42783e-03
6.66667e-02 5.82484e-04
5.00000e-02 2.71576e-04
4.00000e-02 2.19283e-04
3.33333e-02 1.30150e-04
2.85714e-02 8.78144e-05
2.50000e-02 8.18738e-05
2.22222e-02 5.57058e-05
2.00000e-02 3.81332e-05
1.81818e-02 4.10921e-05
1.66667e-02 2.83295e-05
1.53846e-02 2.14251e-05
1.42857e-02 1.83809e-05
1.33333e-02 1.51657e-05
};
\addlegendentry{$N=2$: fit 2.21, nominal 3}
\addplot[blue, no marks, solid, forget plot] coordinates {(1.00000e-01,1.45907e-03) (1.33333e-02,1.69981e-05)};
\addplot[blue, no marks, dashed, forget plot] coordinates {(1.33333e-02,1.51657e-05) (1.00000e-01,6.39805e-03)};
\addplot+[red, mark=square, only marks] table[x=h, y=e] {
h e
1.00000e-01 2.20886e-04
6.66667e-02 2.92366e-05
5.00000e-02 1.17028e-05
4.00000e-02 7.11564e-06
3.33333e-02 2.81324e-06
2.85714e-02 1.82316e-06
2.50000e-02 6.71990e-07
2.22222e-02 6.45641e-07
2.00000e-02 5.22994e-07
1.81818e-02 2.79874e-07
1.66667e-02 1.93727e-07
1.53846e-02 1.16612e-07
1.42857e-02 1.16419e-07
1.33333e-02 8.73084e-08
};
\addlegendentry{$N=3$: fit 3.82, nominal 4}
\addplot[red, no marks, solid, forget plot] coordinates {(1.00000e-01,1.85632e-04) (1.33333e-02,8.49491e-08)};
\addplot[red, no marks, dashed, forget plot] coordinates {(1.33333e-02,8.73084e-08) (1.00000e-01,2.76249e-04)};
\addplot+[green!50!black, mark=triangle, only marks] table[x=h, y=e] {
h e
2.00000e-01 3.27498e-04
1.00000e-01 1.58226e-05
8.00000e-02 6.12509e-06
6.66667e-02 2.73133e-06
5.00000e-02 8.48083e-07
4.00000e-02 4.35615e-07
3.33333e-02 1.68359e-07
2.85714e-02 9.92656e-08
2.66667e-02 7.57120e-08
2.50000e-02 6.38292e-08
2.22222e-02 3.52768e-08
2.00000e-02 3.23981e-08
};
\addlegendentry{$N=4$: fit 4.04, nominal 5}
\addplot[green!50!black, no marks, solid, forget plot] coordinates {(2.00000e-01,2.66869e-04) (2.00000e-02,2.43354e-08)};
\addplot[green!50!black, no marks, dashed, forget plot] coordinates {(2.00000e-02,3.23981e-08) (2.00000e-01,3.23981e-03)};
\addplot+[violet, mark=diamond, only marks] table[x=h, y=e] {
h e
2.00000e-01 3.33838e-05
1.33333e-01 3.97697e-06
1.00000e-01 1.11436e-06
8.00000e-02 3.00496e-07
6.66667e-02 7.40031e-08
5.71429e-02 3.70102e-08
5.00000e-02 3.49419e-08
4.44444e-02 1.15518e-08
4.00000e-02 1.75843e-08
3.63636e-02 7.26650e-09
3.33333e-02 1.15573e-08
};
\addlegendentry{$N=5$: fit 4.75, nominal 6}
\addplot[violet, no marks, solid, forget plot] coordinates {(2.00000e-01,2.48780e-05) (3.33333e-02,5.02336e-09)};
\addplot[violet, no marks, dashed, forget plot] coordinates {(3.33333e-02,1.15573e-09) (2.00000e-01,5.39218e-05)};
\end{loglogaxis}
\end{tikzpicture}

%% file: vortex.tex
\begin{tabular}{r c ccc ccc c}
\toprule
$I_{\max}$ & $\alpha_{\min}$ & L1 & L2 & L$\infty$ & $p_{\mathrm{L1}}$ & $p_{\mathrm{L2}}$ & $p_{\mathrm{L\infty}}$ & steps \\
\midrule
\multicolumn{9}{l}{\textit{$N=2$ (expected order 3)}} \\
25 & 1.77e-01 & 1.977e-04 & 7.332e-05 & 1.201e-04 & -- & -- & -- & 28 \\
50 & 9.59e-03 & 2.881e-05 & 1.109e-05 & 1.658e-05 & 2.78 & 2.72 & 2.86 & 54 \\
75 & 7.31e-02 & 9.123e-06 & 3.529e-06 & 5.772e-06 & 2.84 & 2.83 & 2.60 & 81 \\
100 & 1.07e-03 & 3.994e-06 & 1.540e-06 & 5.352e-06 & 2.87 & 2.88 & 0.26 & 107 \\
150 & 1.91e-03 & 1.232e-06 & 4.726e-07 & 1.064e-06 & 2.90 & 2.91 & 3.99 & 160 \\
\midrule
\multicolumn{9}{l}{\textit{$N=3$ (expected order 4)}} \\
25 & 1.77e-01 & 4.924e-06 & 4.709e-06 & 4.096e-05 & -- & -- & -- & 39 \\
50 & 9.59e-03 & 2.847e-07 & 3.593e-07 & 1.070e-05 & 4.11 & 3.71 & 1.94 & 76 \\
75 & 7.31e-02 & 6.248e-08 & 9.681e-08 & 2.922e-06 & 3.74 & 3.23 & 3.20 & 113 \\
100 & 1.07e-03 & 1.824e-08 & 2.951e-08 & 1.587e-06 & 4.28 & 4.13 & 2.12 & 150 \\
150 & 1.91e-03 & 3.616e-09 & 8.484e-09 & 4.422e-07 & 3.99 & 3.07 & 3.15 & 224 \\
\midrule
\multicolumn{9}{l}{\textit{$N=4$ (expected order 5)}} \\
10 & 2.83e-02 & 1.163e-05 & 5.317e-06 & 3.764e-05 & -- & -- & -- & 20 \\
25 & 1.77e-01 & 1.033e-07 & 3.961e-08 & 1.274e-07 & 5.16 & 5.35 & 6.21 & 50 \\
50 & 9.59e-03 & 4.151e-09 & 1.697e-09 & 3.241e-09 & 4.64 & 4.54 & 5.30 & 97 \\
75 & 7.31e-02 & 6.412e-10 & 2.614e-10 & 6.251e-10 & 4.61 & 4.61 & 4.06 & 145 \\
100 & 1.07e-03 & 1.735e-10 & 6.919e-11 & 1.210e-10 & 4.54 & 4.62 & 5.71 & 193 \\
\midrule
\multicolumn{9}{l}{\textit{$N=5$ (expected order 6)}} \\
10 & 2.83e-02 & 1.219e-06 & 7.568e-07 & 5.129e-06 & -- & -- & -- & 25 \\
20 & 1.13e-01 & 6.649e-09 & 3.060e-09 & 1.472e-08 & 7.52 & 7.95 & 8.44 & 49 \\
25 & 1.77e-01 & 1.715e-09 & 7.725e-10 & 4.709e-09 & 6.07 & 6.17 & 5.11 & 60 \\
30 & 2.55e-01 & 5.218e-10 & 2.320e-10 & 1.981e-09 & 6.53 & 6.60 & 4.75 & 72 \\
50 & 9.59e-03 & 2.515e-11 & 1.941e-11 & 2.335e-09 & 5.94 & 4.86 & -0.32 & 119 \\
\bottomrule
\end{tabular}

%% file: ringleb.tex
\begin{tabular}{r c ccc ccc c}
\toprule
$h$ & $\alpha_{\min}$ & L1 & L2 & L$\infty$ & $p_{\mathrm{L1}}$ & $p_{\mathrm{L2}}$ & $p_{\mathrm{L\infty}}$ & steps \\
\midrule
\multicolumn{9}{l}{\textit{$N=2$ (expected order 3)}} \\
0.2 & 3.049e-04 & 8.740e-05 & 1.061e-04 & 1.041e-03 & -- & -- & -- & 2 \\
0.1 & 2.172e-04 & 1.233e-05 & 1.401e-05 & 2.396e-04 & 2.83 & 2.92 & 2.12 & 4 \\
0.05 & 1.114e-04 & 1.954e-06 & 1.917e-06 & 3.807e-05 & 2.66 & 2.87 & 2.65 & 8 \\
0.025 & 1.487e-05 & 2.983e-07 & 2.696e-07 & 5.564e-06 & 2.71 & 2.83 & 2.77 & 16 \\
\midrule
\multicolumn{9}{l}{\textit{$N=3$ (expected order 4)}} \\
0.2 & 3.049e-04 & 7.706e-06 & 1.286e-05 & 2.387e-04 & -- & -- & -- & 3 \\
0.1 & 2.172e-04 & 4.386e-07 & 9.424e-07 & 2.127e-05 & 4.14 & 3.77 & 3.49 & 6 \\
0.05 & 1.114e-04 & 2.209e-08 & 5.506e-08 & 2.281e-06 & 4.31 & 4.10 & 3.22 & 12 \\
0.025 & 1.487e-05 & 9.955e-10 & 2.455e-09 & 1.491e-07 & 4.47 & 4.49 & 3.93 & 23 \\
\midrule
\multicolumn{9}{l}{\textit{$N=4$ (expected order 5)}} \\
0.3 & 1.635e-03 & 4.315e-06 & 7.507e-06 & 1.273e-04 & -- & -- & -- & 3 \\
0.2 & 3.049e-04 & 9.902e-07 & 1.878e-06 & 3.732e-05 & 3.63 & 3.42 & 3.03 & 4 \\
0.1 & 2.172e-04 & 3.136e-08 & 7.393e-08 & 2.186e-06 & 4.98 & 4.67 & 4.09 & 8 \\
0.05 & 1.114e-04 & 6.405e-10 & 1.590e-09 & 5.559e-08 & 5.61 & 5.54 & 5.30 & 15 \\
\midrule
\multicolumn{9}{l}{\textit{$N=5$ (expected order 6)}} \\
0.4 & 1.297e-03 & 5.414e-06 & 7.810e-06 & 8.487e-05 & -- & -- & -- & 3 \\
0.3 & 1.635e-03 & 7.779e-07 & 1.496e-06 & 1.877e-05 & 6.74 & 5.74 & 5.24 & 3 \\
0.2 & 3.049e-04 & 1.355e-07 & 2.807e-07 & 4.164e-06 & 4.31 & 4.13 & 3.71 & 5 \\
0.1 & 2.172e-04 & 1.506e-09 & 4.683e-09 & 1.291e-07 & 6.49 & 5.91 & 5.01 & 9 \\
\bottomrule
\end{tabular}

%% file: quarterannulus.tex
\begin{tabular}{r c ccc ccc c}
\toprule
$I_{\max}$ & $\alpha_{\min}$ & L1 & L2 & L$\infty$ & $p_{\mathrm{L1}}$ & $p_{\mathrm{L2}}$ & $p_{\mathrm{L\infty}}$ & steps \\
\midrule
\multicolumn{9}{l}{\textit{$N=2$ (expected order 3)}} \\
40 & 2.31e-03 & 4.962e-06 & 7.303e-06 & 1.011e-04 & -- & -- & -- & 37 \\
60 & 2.53e-04 & 1.785e-06 & 2.593e-06 & 3.219e-05 & 2.52 & 2.55 & 2.82 & 55 \\
80 & 3.72e-06 & 8.568e-07 & 1.233e-06 & 1.377e-05 & 2.55 & 2.58 & 2.95 & 73 \\
100 & 2.65e-05 & 4.857e-07 & 6.939e-07 & 6.794e-06 & 2.54 & 2.58 & 3.16 & 91 \\
150 & 1.25e-05 & 1.700e-07 & 2.400e-07 & 2.764e-06 & 2.59 & 2.62 & 2.22 & 136 \\
\midrule
\multicolumn{9}{l}{\textit{$N=3$ (expected order 4)}} \\
20 & 2.63e-03 & 5.581e-07 & 1.324e-06 & 2.309e-05 & -- & -- & -- & 26 \\
30 & 1.15e-03 & 9.579e-08 & 2.266e-07 & 5.691e-06 & 4.35 & 4.35 & 3.45 & 38 \\
40 & 2.31e-03 & 2.864e-08 & 7.287e-08 & 2.127e-06 & 4.20 & 3.94 & 3.42 & 51 \\
50 & 3.64e-05 & 1.132e-08 & 3.063e-08 & 1.065e-06 & 4.16 & 3.88 & 3.10 & 64 \\
60 & 2.53e-04 & 4.977e-09 & 1.381e-08 & 5.365e-07 & 4.51 & 4.37 & 3.76 & 76 \\
\midrule
\multicolumn{9}{l}{\textit{$N=4$ (expected order 5)}} \\
15 & 9.65e-03 & 1.046e-07 & 2.344e-07 & 5.910e-06 & -- & -- & -- & 25 \\
20 & 2.63e-03 & 2.597e-08 & 5.989e-08 & 1.591e-06 & 4.84 & 4.74 & 4.56 & 33 \\
25 & 6.93e-05 & 8.679e-09 & 2.109e-08 & 7.738e-07 & 4.91 & 4.68 & 3.23 & 41 \\
30 & 1.15e-03 & 2.365e-09 & 4.545e-09 & 1.429e-07 & 7.13 & 8.42 & 9.26 & 49 \\
35 & 1.29e-03 & 1.112e-09 & 2.261e-09 & 8.976e-08 & 4.89 & 4.53 & 3.02 & 57 \\
40 & 2.31e-03 & 6.089e-10 & 1.281e-09 & 2.797e-08 & 4.51 & 4.26 & 8.73 & 66 \\
\midrule
\multicolumn{9}{l}{\textit{$N=5$ (expected order 6)}} \\
10 & 6.16e-03 & 3.790e-08 & 7.721e-08 & 2.202e-06 & -- & -- & -- & 21 \\
15 & 9.65e-03 & 1.601e-09 & 5.369e-09 & 1.613e-07 & 7.80 & 6.57 & 6.45 & 30 \\
20 & 2.63e-03 & 2.373e-10 & 9.027e-10 & 3.402e-08 & 6.64 & 6.20 & 5.41 & 40 \\
25 & 6.93e-05 & 8.009e-11 & 2.548e-10 & 1.801e-08 & 4.87 & 5.67 & 2.85 & 50 \\
30 & 1.15e-03 & 2.723e-11 & 8.817e-11 & 2.891e-08 & 5.92 & 5.82 & -2.60 & 60 \\
\bottomrule
\end{tabular}

%% file: potentialflow.tex
\begin{tabular}{r c cccc cccc r}
\toprule
$I_{\max}$ & $\alpha_{\min}$ & L1($\rho$) & L1($u$) & L1($v$) & L1($p$) & $p_{\rho}$ & $p_{u}$ & $p_{v}$ & $p_{p}$ & steps \\
\midrule
25 & 1.768e-01 & 1.717e-05 & 2.563e-05 & 2.373e-05 & 1.719e-05 & -- & -- & -- & -- & 16 \\
50 & 9.593e-03 & 1.442e-06 & 2.999e-06 & 2.820e-06 & 1.528e-06 & 3.57 & 3.10 & 3.07 & 3.49 & 32 \\
75 & 7.311e-02 & 4.119e-07 & 8.416e-07 & 7.903e-07 & 5.770e-07 & 3.09 & 3.13 & 3.14 & 2.40 & 48 \\
100 & 1.071e-03 & 1.591e-07 & 3.489e-07 & 3.332e-07 & 3.911e-07 & 3.31 & 3.06 & 3.00 & 1.35 & 64 \\
\bottomrule
\end{tabular}

%% file: cpu.tex
\begin{tabular}{c l r rr r rr r}
\toprule
$N$ & mesh & steps & explicit cells & implicit cells & explicit [\%] & CPU mixed [s] & CPU impl.\ [s] & ratio \\
\midrule
\multicolumn{9}{l}{\textit{Ringleb flow}} \\
\midrule
2 & $h=0.2$   & 20 & 52   & 134  & 28 & 10.8 & 11.3  & 1.0 \\
2 & $h=0.1$   & 20 & 371  & 279  & 57 & 10.1 & 12.4  & 1.2 \\
2 & $h=0.05$  & 20 & 1844 & 572  & 76 & 11.6 & 19.4  & 1.7 \\
2 & $h=0.025$ & 20 & 8164 & 1154 & 88 & 21.7 & 54.4  & 2.5 \\
3 & $h=0.2$   & 20 & 52   & 134  & 28 & 24.6 & 25.8  & 1.1 \\
3 & $h=0.1$   & 20 & 371  & 279  & 57 & 29.7 & 39.6  & 1.3 \\
3 & $h=0.05$  & 20 & 1844 & 572  & 76 & 41.7 & 89.9  & 2.2 \\
3 & $h=0.025$ & 20 & 8164 & 1154 & 88 & 45.7 & 221.0 & 4.8 \\
\midrule
\multicolumn{9}{l}{\textit{Acoustic pulse scattered by five cylinders}} \\
\midrule
3 & $106\times106$ & 1113 & 8918 & 758 & 92 & 1065.0 & 5533.7 & 5.2 \\
\midrule
\multicolumn{9}{l}{\textit{Sod shock tube with walls}} \\
\midrule
2 & $100\times10$ & 1000 & 600 & 400 & 60 & 418.1 & 846.6 & 2.0 \\
\bottomrule
\end{tabular}

%% file: dod.tex
\begin{tabular}{r ccc ccc rrr}
\toprule
 & \multicolumn{3}{c}{$L^1$ error} & \multicolumn{3}{c}{$L^\infty$ error} & \multicolumn{3}{c}{CPU [s]} \\
\cmidrule(lr){2-4}\cmidrule(lr){5-7}\cmidrule(lr){8-10}
$I_{\max}$ & DoD & mixed & impl. & DoD & mixed & impl. & DoD & mixed & impl. \\
\midrule
\multicolumn{10}{l}{\textit{$N=1$ (expected order 2)}} \\
100 & 2.338e-04 & 1.995e-04 & 2.370e-04 & 5.819e-04 & 1.095e-03 & 1.095e-03 & 0.34 & 0.64 & 1.79 \\
150 & 1.067e-04 & 9.032e-05 & 1.078e-04 & 2.688e-04 & 5.301e-04 & 5.306e-04 & 1.09 & 1.57 & 6.31 \\
200 & 6.043e-05 & 5.102e-05 & 6.103e-05 & 1.564e-04 & 3.078e-04 & 3.078e-04 & 2.26 & 4.07 & 11.80 \\
250 & 3.876e-05 & 3.264e-05 & 3.915e-05 & 1.162e-04 & 2.027e-04 & 2.026e-04 & 4.52 & 7.40 & 21.54 \\
300 & 2.694e-05 & 2.258e-05 & 2.721e-05 & 8.713e-05 & 1.508e-04 & 1.509e-04 & 8.26 & 12.54 & 37.21 \\
\midrule
\multicolumn{10}{l}{\textit{$N=2$ (expected order 3)}} \\
40 & 7.701e-05 & 6.922e-05 & 7.818e-05 & 1.883e-04 & 2.401e-04 & 2.402e-04 & 0.10 & 0.27 & 0.59 \\
60 & 2.570e-05 & 2.234e-05 & 2.590e-05 & 7.972e-05 & 1.090e-04 & 1.090e-04 & 0.31 & 0.78 & 2.02 \\
80 & 1.064e-05 & 9.103e-06 & 1.068e-05 & 3.057e-05 & 5.921e-05 & 5.920e-05 & 0.65 & 1.53 & 4.22 \\
100 & 5.248e-06 & 4.466e-06 & 5.262e-06 & 1.397e-05 & 2.375e-05 & 2.375e-05 & 1.28 & 3.12 & 8.92 \\
150 & 1.403e-06 & 1.151e-06 & 1.405e-06 & 4.608e-06 & 8.481e-06 & 8.482e-06 & 3.94 & 8.97 & 26.73 \\
\midrule
\multicolumn{10}{l}{\textit{$N=3$ (expected order 4)}} \\
40 & 2.859e-06 & 2.785e-06 & 2.963e-06 & 8.468e-06 & 1.538e-05 & 1.538e-05 & 0.41 & 0.96 & 2.44 \\
60 & 6.399e-07 & 6.139e-07 & 6.573e-07 & 1.780e-06 & 3.821e-06 & 3.821e-06 & 1.25 & 2.82 & 8.07 \\
80 & 2.085e-07 & 1.924e-07 & 2.128e-07 & 5.608e-07 & 1.166e-06 & 1.166e-06 & 2.93 & 6.23 & 18.46 \\
100 & 8.193e-08 & 7.268e-08 & 8.350e-08 & 2.775e-07 & 6.493e-07 & 6.490e-07 & 5.20 & 10.63 & 30.52 \\
150 & 1.569e-08 & 1.443e-08 & 1.606e-08 & 5.821e-08 & 1.383e-07 & 1.383e-07 & 16.59 & 32.58 & 109.14 \\
\bottomrule
\end{tabular}